\documentclass[11pt]{article}

\usepackage[english]{babel}
\usepackage[a4paper,scale=0.75,hcentering]{geometry}
\usepackage[hyperfootnotes=false]{hyperref} 

\usepackage{setspace}

\hypersetup{
	colorlinks = true,
	linkcolor = {blue},
	urlcolor = {red},
	citecolor = {blue}
}
\usepackage{amsmath, amsthm, amssymb, mathrsfs, mathtools, enumitem}
\usepackage[x11names]{xcolor}

\allowdisplaybreaks

\newcounter{results}[section] 

\theoremstyle{plain}
\newtheorem{theorem}[results]{Theorem}
\newtheorem{lemma}[results]{Lemma}
\newtheorem{proposition}[results]{Proposition}

\theoremstyle{remark}
\newtheorem{remark}[results]{Remark}

\theoremstyle{definition}
\newtheorem{definition}[results]{Definition}

\numberwithin{equation}{section}

\newcommand{\N}{\ensuremath{\mathbb N}} 
\newcommand{\R}{\ensuremath{\mathbb R}} 
\newcommand{\B}{\ensuremath{\mathbb B}} 
\newcommand{\F}{\ensuremath{\mathcal {F}}}
\DeclareMathOperator{\dis}{d} 

\makeatletter
\DeclarePairedDelimiter{\@tmpabs}{\lvert}{\rvert}
\newcommand{\@absstar}[1]{{\@tmpabs*{#1}}}
\newcommand{\@absnostar}[2][]{{\@tmpabs[#1]{#2}}}
\newcommand{\abs}{\@ifstar\@absstar\@absnostar}
\makeatother

\makeatletter
\DeclarePairedDelimiter{\@tmpnorm}{\lVert}{\rVert}
\newcommand{\@normstar}[1]{{\@tmpnorm*{#1}}}
\newcommand{\@normnostar}[2][]{{\@tmpnorm[#1]{#2}}}
\newcommand{\norm}{\@ifstar\@normstar\@normnostar}
\makeatother

\newcommand{\scalprod}[2]{\ensuremath{\langle #1, #2\rangle}}
\newcommand{\defeq}{\ensuremath{\coloneqq}}
\newcommand{\lapl}{\ensuremath{\Delta}}

\begin{document}
	
\title{\bf Sharp quantitative stability for the fractional Sobolev trace inequality\thanks{Supported by  National Key R\&D Program of China (Grant 2023YFA1010001) and NSFC(12171265 and 12271184).\quad\quad\quad\quad\quad\quad\quad\quad  E-mail addresses: zhangyf22@mails.tsinghua.edu.cn (Zhang),    zhou-yx22@mails.tsinghua.edu.cn(Zhou),    zou-wm@mail.tsinghua.edu.cn (Zou)}}
	
\author{{\bf Yingfang Zhang, Yuxuan Zhou and Wenming Zou}\\ \\ {\footnotesize \it  Department of Mathematical Sciences, Tsinghua University, Beijing 100084, China.} }

\date{}

\maketitle

\begin{abstract}
	In this paper, we study the stability of the fractional Sobolev trace inequality within both the functional and critical point settings.
	
	In the functional setting, we establish the following sharp estimate:
	\begin{equation*}
		C_{\mathrm{BE}}(n,m,\alpha)\inf_{v\in\mathcal{M}_{n,m,\alpha}}\norm{f-v}_{D_\alpha(\R^n)}^2 \le \norm{f}_{D_\alpha(\R^n)}^2 - S(n,m,\alpha) \norm{\tau_mf}_{L^{q}(\R^{n-m})}^2,
	\end{equation*}
	where $0\leq m< n$, $\frac{m}{2}<\alpha<\frac{n}{2}, q=\frac{2(n-m)}{n-2\alpha}$ and $\mathcal{M}_{n,m,\alpha}$ denotes the manifold of extremal functions. Additionally, we find an explicit bound for the stability constant $C_{\mathrm{BE}}$. Furthermore, we establish a compactness result ensuring the existence of minimizers for the Bianchi-Egnell type functional:
	\begin{equation*}
		S_{\mathrm{Tr}}(f)\coloneqq\frac{\norm{f}_{D_\alpha}^2 - S(n,m,\alpha)\norm{\tau_mf}_{L^q}^2}{\inf\limits_{v\in \mathcal{M}_{n,m,\alpha}} \norm{f-v}_{D_\alpha}^2},\quad \text{for }f\in D_\alpha(\R^n)\backslash\mathcal{M}_{n,m,\alpha}.
	\end{equation*}
Our stability results extend previous works on the Escobar trace inequality and fractional Sobolev inequality. As a corollary, we derive some improved trace inequalities for functions supported in general domains. Applying a  dual scheme, we also obtain a sharp a priori estimate for Neumann problem on the half-space.
	
	In the critical point setting, we investigate the validity of a sharp quantitative profile decomposition related to the Escobar trace inequality and  establish a qualitative profile decomposition for the critical elliptic equation
	\begin{equation*}
		\begin{cases}
			\Delta u= 0 & \text{in }\R_+^n \\
			\frac{\partial u}{\partial t}=-|u|^{\frac{2}{n-2}}u & \text{on }\partial\R_+^n.
		\end{cases}
	\end{equation*}
 We then derive the sharp stability estimate:
	\begin{equation*}
		C_{\mathrm{CP}}(n,\nu)d(u,\mathcal{M}_{\mathrm{E}}^{\nu})\leq \left\Vert \Delta u +|u|^{\frac{2}{n-2}}u\right\Vert_{H^{-1}(\R_+^n)},
	\end{equation*}
	where $\nu=1,n\geq 3$ or $\nu\geq2,n=3$ and $\mathcal{M}_{\mathrm{E}}^\nu$ represents the manifold consisting of $\nu$ weak-interacting Escobar bubbles. Through some refined estimates, we also give a strict upper bound for $C_{\mathrm{CP}}(n,1)$, which is $\frac{2}{n+2}$.
\end{abstract}

\newpage

\tableofcontents

\vskip1in

\section{Introduction}
	
\subsection{Background and motivations}
	
The classical Sobolev inequality with exponent $1<p<n$ states that, for any $n\geq 2$ and $u\in W^{1,p}(\R^n)$, it holds
	\begin{equation}\label{sobo}
		S_p\Vert u\Vert_{L^{p^*}}\leq \Vert\nabla u\Vert_{L^p},
	\end{equation}
	where $p^*=\frac{np}{n-p}$ and $S_p$ is the sharp constant. It is well known that the extremal functions of \eqref{sobo} (see \cite{Aubin, Talenti}) form an $(n+2)$-dimensional manifold:
	\begin{equation*}
		\mathcal{M}_p\coloneqq \Big\{ w_{a,b,x_{0}}\coloneqq \frac{a}{(1+b|x-x_{0}|^{\frac{p}{p-1}})^{\frac{n-p}{p}}}: a\in \R\setminus\{0\},\  b>0,\ x_{0}\in \mathbb{R}^n\Big\}.
	\end{equation*}
	When $p=2$ ,the Euler-Lagrange equation associated to the inequality \eqref{sobo} is, up to suitable normalization, given by
	\begin{equation}\label{eleqt}
		\Delta u+u|u|^{2^*-2}=0.
	\end{equation}
	It was shown in \cite{Gidas} that all the positive solutions of \eqref{eleqt} are Talenti bubbles:
	\begin{equation*}
		T[z,\lambda](x)\coloneqq(n(n-2))^{\frac{n-2}{4}}
		\frac{\lambda^{\frac{n-2}{2}}}{(1+\lambda^2\abs{x-z}^2)^{\frac{n-2}{2}}} ,\quad  \lambda>0,\  z\in \R^n.
	\end{equation*}
Once the rigidity results for \eqref{sobo} and \eqref{eleqt} are established, it is natural to consider stability versions. In the functional setting, a fundamental question arises --- does the deviation of a function from attaining equality in \eqref{sobo} control its distance from the extremal manifold $\mathcal{M}_p$? Similarly, in the critical point setting, does almost being a positive solution of \eqref{eleqt} imply almost being a Talenti bubble?
\vskip0.2in
The question on the stability of functional inequalities was first raised by Br\'ezis and Lieb  in \cite{Brezis-Lieb}. In their work, they improved the Sobolev embedding on bounded domains by introducing a remainder term, thereby raising an open problem about whether the homogeneous inequality \eqref{sobo} can be reinforced. A comprehensive answer was provided by Bianchi and Egnell in \cite{bianchi1991} for the case $p=2$, presenting a complete result as follows:
	\begin{equation}\label{BE}
		C_{\mathrm{BE}}\inf_{z\in\R^n,\lambda>0,\alpha\in\R}\|\nabla(u-\alpha U[z,\lambda])\|_{L^2}^2\leq \|\nabla u\|_{L^2}^2-S_2^2\|u\|^2_{L^{2^*}}.
	\end{equation}
	The result is considered complete because both the exponent $2$ and the norm in the left hand side are sharp. For the case $p\neq 2$, due to the absence of Hilbertian structure, the stability of \eqref{sobo} remained an open problem for a long time and was recently solved by Figalli and Zhang in \cite{figalli-Zhang1} (see \cite{cianchi-fusco-maggi-pratelli,figalli-maggi-Pratelli-2013,figalli-Zhang1, neumayer3, figalli-Neumayer} and the references therein for earlier works).
\vskip0.2in	
Besides the classical Sobolev inequality, there are many other important geometric and functional inequalities with well studied stability. Examples include the fractional Sobolev inequality \cite{bartsch, chen-frank}, the Sobolev trace inequality \cite{Ho, Chenlu}, the isoperimetric inequality \cite{fusco-maggi-pratelli, figalli-Maggi-Pratelli-iso, maggi, neumayer1, cicalese-Leonardi, neumayer2, figalli-fusco}, the Hardy-Littlewood-Sobolev inequality \cite{figalli-Carlen, carlen1}, the Gagliardo-Nirenberg-Sobolev inequality \cite{figalli-Carlen, dolbeault-Toscani, ruffini, bonforte1, bonforte2} and the Caffarelli-Kohn-Nirenberg inequality \cite{Wu-Wei,De,Do, Frank2,zho1}.
	
\vskip0.2in
	Once a complete stability result is established, a natural and interesting question is to quantify the stability. For example, what can be said about the constant $C_{\mathrm{BE}}$ in \eqref{BE}? Traditional stability results of Sobolev-type inequalities rely heavily on local spectrum analysis and global concentration-compactness principle, which offer limited information about the constants involved. Recent advancements by Dolbeault, Esteban, Figalli, Frank and Loss \cite{dol-fig} stated shedding light on explicit lower bound of $C_{\mathrm{BE}}$ using competing symmetries and continuous Steiner symmetrization. The subsequent work by K\"{o}nig in \cite{konig1, konig3} gave an upper bound of $C_{\mathrm{BE}}$ for general fractional Sobolev inequality, and showed that $C_{\mathrm{BE}}$ can be attained by certain functions. In a parallel vein, Chen, Lu and Tang \cite{Chenlu} gave a lower bound for stability constants of the Hardy-Littlewood-Sobolev inequality, the fractional Sobolev inequality and some trace inequalities. Wei, Wu \cite{Wu-Wei2} and Deng, Tian \cite{Deng1} established analogous results for the Caffarelli-Kohn-Nirenberg inequality.
\vskip0.2in
	The question on the stability of critical points traces back to the celebrated global compactness principle of Struwe \cite{struwe1984}. This principle states that a bounded nonnegative sequence $\{u_n\}$ satisfying $\left\Vert\Delta u+u|u|^{2^*-2}\right\Vert_{H^{-1}(\R^n)}\rightarrow 0$ can be approximated by several Talenti bubbles up to a subsequence. Ciraolo, Figalli and Maggi \cite{9} provided the first quantitative version of this principle, yielding an optimal linear estimate when dealing with a single bubble:
	\begin{equation*}
		C\dis(u,\mathcal{M}_T)\leq \left\Vert \Delta u +|u|^{2^*-2}u\right\Vert_{H^{-1}(\R^n)},
	\end{equation*}
	where $\mathcal{M}_T$ is the manifold of Talenti bubbles. When there are more bubbles, it becomes complicated due to the interaction of distinct bubbles. Figalli and Glaudo \cite{figalli} tackled such case and found an interesting phenomenon: linear estimates were possible for dimension $3\leq n\leq 5$, but for $n>5$, counter-examples were constructed. The remaining case $n>5$ was ultimately solved by Deng, Sun and Wei in \cite{dsw} using finite dimensional reduction method, highlighting that logarithmic ($n=6$) or sublinear ($n>6$) estimates are sharp in high dimensions.
\vskip0.2in	
	In addition to the classical Yamabe equation \eqref{eleqt}, the stability of several inequality-related critical equations have been extensively studied, such as the fractional Sobolev inequality \cite{Ary,Che,konig2}, the Caffarelli-Kohn-Nirenberg inequality \cite{Wu-Wei,zho2} and the Hardy-Littlewood-Sobolev inequality \cite{liu, yang, Lu,Yan}.
\vskip0.2in	
	As in the functional setting, a natural question is to study the optimal stability constant for critical equations. However, this question seems much harder, for the possible reason that in general, dual norms do not behave as well as Sobolev norms. Standard tools such as expansions and Br\'ezis-Lieb type arguments, which are useful in functional settings, are difficult to handle dual norms. To our best knowledge, the first in this direction was given by De Nitti and K\"{o}nig in \cite{konig2}, which found an explicit upper bound for the stability constant of the fractional Yamabe equation in one bubble case by treating the dual norm carefully and using a third-order expansion.
	
	\subsection{Problem setup and main results}
	
	In this paper, we focus on the following sharp fractional Sobolev trace inequality given by Einav and Loss in \cite{fractr}:
	\begin{equation}\label{trtr}
		S(n,m,\alpha)\norm{\tau_mf}_{L^{\frac{2(n-m)}{n-2\alpha}}(\R^{n-m})}^2 \le \norm{f}_{D_\alpha(\R^n)}^2,
	\end{equation}
	where $0\leq m<n$, $\frac{m}{2}<\alpha<\frac{n}{2}$, $f\in D_\alpha(\R^n)$ and $\tau_mf$ is the restriction of $f$ on $\R^{n-m}$. The space $D_\alpha(\R^n)$ is the standard fractional Sobolev space in $\R^n$. It is the closure of $C_c^{\infty}(\R^n)$ under the norm:
	\begin{equation*}
		\norm{f}_{D_\alpha(\R^n)} := \left(\int_{\R^n} \abs{\hat f(k)}^2 \abs{2\pi k}^{2\alpha}\,dk\right)^{1/2}.
	\end{equation*}
	The sharp constant $S(n,m,\alpha)$ and the extremal manifold $\mathcal{M}_{n,m,\alpha}$ are explicit and given by
	\begin{equation}\label{snma}
		S(n,m,\alpha) = 2^{2\alpha}\pi^{\alpha}\frac{\Gamma(\alpha)\Gamma(n/2+\alpha-m)}{\Gamma(n/2-\alpha)\Gamma(\alpha-m/2)}\left(\frac{\Gamma(n-m)}{\Gamma((n-m)/2)}\right)^{\frac{m-2\alpha}{n-m}},
	\end{equation}
	\begin{equation*}
		\mathcal{M}_{n,m,\alpha}=\left\{\lambda\int_{\R^{n-m}} \frac{1}{(|x'|^2+|x''-y''|^2)^{(n-2\alpha)/2}}\frac{1}{(\gamma^2+|y''-a|^2)^{(n+2\alpha-2m)/2}}\,dy''\right\},
	\end{equation*}
where $\lambda\in\R,\gamma>0,a\in\R^{n-m},x'\in\R^m,x''\in\R^{n-m}$.
\vskip0.2in
\begin{remark}
     It is worth mentioning that the restriction $\frac{m}{2}<\alpha<\frac{n}{2}$ arises naturally from the proof of \eqref{trtr}. As shown in  \cite{fractr}, the inequality \eqref{trtr} can be viewed as a composition of two inequalities: the reduction inequality $\tau_m: D_\alpha(\R^n) \to D_{\alpha-m/2}(\R^{n-m})$ (see Theorem \ref{thm: ineq2}) and the fractional Sobolev inequality $D_{\alpha-m/2}(\R^{n-m})\hookrightarrow L^{\frac{2(n-m)}{n-2\alpha}}(\R^{n-m})$ (see Theorem \ref{thm: ineq1}). For the first inequality to hold, we need $\alpha>\frac{m}{2}$. For the second inequality to hold, we need $\alpha<\frac{n}{2}$.
\end{remark}
\vskip0.2in

When $m=0$ and $0<\alpha<\frac{n}{2}$, \eqref{trtr} reduces to the classical fractional Sobolev inequality:
\begin{align}\label{ppp1}
    S(n,0,\alpha)\norm{\tau_mf}_{L^{\frac{2n}{n-2\alpha}}(\R^{n})}^2 \le \norm{f}_{D_\alpha(\R^n)}^2.
\end{align}

When $m=\alpha=1$, \eqref{trtr} is equivalent to the Escobar trace inequality in \cite{Escobar}:
	\begin{equation}\label{escobar}
		S_{\mathrm{E}}(n)\norm{f}_{L^{\frac{2(n-1)}{n-2}}(\R^{n-1})}^2 \le \norm{f}_{H^1(\R_+^n)}^2,
	\end{equation}
	where $S_{\mathrm{E}}(n)=S(n,1,1)$ is the sharp constant.

\vskip0.1in
There are various quantitative stability results on these two inequalities. For the fractional Sobolev inequality \eqref{ppp1}, we refer to the works \cite{bianchi1991,bartsch,chen-frank}. Regarding the Escobar trace inequality \eqref{escobar}, we refer to \cite{Ho}. Recently, we observed that in \cite{Chenlu}, the authors obtained quantitative stability estimates for the inequality \eqref{trtr} in the case $m=1$ and $\frac{1}{2}<\alpha<\frac{n}{2}$, thereby generalizing the results in \cite{Ho}. Their approach relies on the dual scheme developed by Carlen in \cite{carlen1}, which is fundamentally different from the methods we employ in this paper.
\vskip0.1in
In this paper, we consider the stability of \eqref{trtr} in both functional and critical point settings.
\vskip0.26in	

{\bf $\diamondsuit$ In the functional setting,}  we prove the following sharp estimate:
	\begin{theorem}\label{main thm1}
		Let $0\le m<n$ and $\frac{m}{2} < \alpha < \frac n2$. Then there exists a constant $C=C(n,m,\alpha)$ such that for any $f\in D_\alpha(\R^n)$, we have
		\begin{equation}\label{fractional-trace stability}
			C\norm{f-v}_{D_\alpha(\R^n)}^2 \le \norm{f}_{D_\alpha(\R^n)}^2 - S(n,m,\alpha) \norm{\tau_mf}_{L^{\frac{2(n-m)}{n-2\alpha}}(\R^{n-m})}^2
		\end{equation}
		for some $v\in \mathcal{M}_{n,m,\alpha}$. We denote $C_{\mathrm{BE}}(n,m,\alpha)$ to be the best stability constant $C$ above.
	\end{theorem}

\vskip0.1in
The inequality 	\eqref{fractional-trace stability} is sharp in the following two key aspects. First, we use the norm $\norm{\cdot}_{D_\alpha(\R^n)}$ to quantify the distance between $f$ and $v$, which is evidently optimal. Second, the exponent $2$ of $\norm{f-v}_{D_\alpha(\R^n)}$ is also sharp in the following sense: there exists a sequence $\{f_i\}_i\subset D_\alpha(\R^n)$ such that $f_i\notin \mathcal{M}_{n,m,\alpha}$, $\inf\limits_{v\in \mathcal{M}_{n,m,\alpha}}\norm{f_i-v}_{D_\alpha(\R^n)}\rightarrow 0$, and for any $\gamma<2$,
\begin{align}\label{bbb2}
    \frac{\norm{f_i}_{D_\alpha(\R^n)}^\gamma - S(n,m,\alpha)^\frac{\gamma}{2}\norm{\tau_mf_i}_{L^{\frac{2(n-m)}{n-2\alpha}}(\R^{n-m})}^\gamma}{\inf\limits_{v\in \mathcal{M}_{n,m,\alpha}}\norm{f_i-v}^\gamma_{D_\alpha(\R^n)}}\rightarrow 0.
\end{align}
The construction of the sequence $\{f_i\}$ is detailed in Remark \ref{bbb1}.

A crucial method employed in proving \eqref{fractional-trace stability} is a reduction principle (see Theorem \ref{thm: ineq2}) proved in \cite{fractr}. This reduction allows us to connect \eqref{trtr} with standard fractional Sobolev inequality.
	Precisely, we show that the stability of \eqref{trtr} can be reduced to that of the fractional Sobolev inequality.

\vskip0.1in
	With the establishment of \eqref{fractional-trace stability}, we are interested in the following minimization problem:
	\begin{equation}\label{mini}
		C_{\mathrm{BE}}(n,m,\alpha):=\inf_{f\in D_\alpha(\R^n)\backslash\mathcal{M}_{n,m,\alpha}}S_{\mathrm{Tr}}(f),
	\end{equation}
	where
	\begin{equation*}
		S_{\mathrm{Tr}}(f)\coloneqq\frac{\norm{f}_{D_\alpha(\R^n)}^2 - S(n,m,\alpha)\norm{\tau_mf}_{L^{\frac{2(n-m)}{n-2\alpha}}(\R^{n-m})}^2}{\dis^2(f,\mathcal{M}_{n,m,\alpha})}.
	\end{equation*}
    Here $\dis(f,\mathcal{M}_{n,m,\alpha})$ denotes $\inf\limits_{v\in \mathcal{M}_{n,m,\alpha}}\norm{f_i-v}_{D_\alpha(\R^n)}$.
We prove the following theorem.
\vskip0.1in

\begin{theorem}\label{main+++} Under the same assumptions of Theorem \ref{main thm1}, we have
\begin{itemize}
\item [(1)] $C_{\mathrm{BE}}(n,m,\alpha)=C_{\mathrm{BE}}(n-m,0,\alpha-\frac{m}{2})$;
\item [(2)] it has the following upper-  and lower-bound estimates:
\begin{equation*}
C_{\mathrm{BE}}(n,m,\alpha)
\begin{cases}
				<\min\left\{\frac{4\alpha-2m}{n+2\alpha+2-2m},2-2^{\frac{n-2\alpha}{n-m}}\right\},& \hbox{ if }   n-m\geq 2,\\
				\leq \frac{4\alpha-2m}{2\alpha+3-m},&\hbox{ if }   n-m=1,
			\end{cases}
		\end{equation*}
 and
\begin{equation*}
C_{\mathrm{BE}}(n,m,\alpha)\geq	\frac{\min\{K_{n-m,\alpha-m/2},1,2^{\frac{n+2\alpha-2m}{n-m}}-2\}}{4}.
\end{equation*}
Here $K_{\cdot,\cdot}$ is a complicated number with an explicit expression (see \cite[Theorem 1.2]{Chenlu}).
\item [(3)] when $n-m\geq2$, $C_{\mathrm{BE}}(n,m,\alpha)$ is attainable at  some minimizer   $u\in D_\alpha(\R^n)\backslash\mathcal{M}_{n,m,\alpha}$. In fact, any minimizing sequence contains a subsequence that converges to some minimizer.

\end{itemize}
\end{theorem}
\vskip0.15in
\begin{remark}
    When $n-m=1$, it remains an open and interesting problem whether $C_{\mathrm{BE}}(n,m,\alpha)$ is attainable and whether $C_{\mathrm{BE}}(n,m,\alpha)<\frac{4\alpha-2m}{2\alpha+3-m}$. In fact, based on our proof, to address this open problem, it suffices to consider the case $n=1$, $m=0$ and $0<\alpha<\frac{1}{2}$, which corresponds to the fractional Sobolev inequality in one dimension. It has been observed in \cite{konig3,konig1} that the fractional Sobolev inequality in one dimension may behave very differently from the fractional Sobolev inequalities in higher dimensions. We refer readers to the interesting paper \cite{konig4}, in which by estimating suitable test functions, K\"{o}nig provided several pieces of evidence to support the following \emph{Conjecture}: $C_{\mathrm{BE}}(1,0,\alpha)$ is not attainable and  $C_{\mathrm{BE}}(1,0,\alpha)<\frac{4\alpha}{2\alpha+3}$.
\end{remark}
\vskip0.15in
\begin{remark}
    From Theorem \ref{main+++}, to estimate $C_{\mathrm{BE}}(n,m,\alpha)$, it suffices to estimate $C_{\mathrm{BE}}(n-m,0,\alpha-\frac{m}{2})$. The lower bound we present in Theorem \ref{main+++} \emph{(2)} is derived from  \cite{Chenlu}. After completing this paper, we learned that in \cite{Chenlu1,Chenlu2}, Chen, Lu, and Tang improved their previous results in \cite{Chenlu}. Specifically, they showed that for any fixed $s>0$, their exists an explicit positive constant $\beta_s$ such that for any $n>2s$,
    \begin{align}
        C_{\mathrm{BE}}(n,0,s)\geq \frac{\beta_s}{n}.\nonumber
    \end{align}
    Moreover, when $0<s<\frac{1}{6}$, one can express $\beta_s=s\beta$, where $\beta$ is an explicit positive constant independent of $n$ and $s$. We refer readers to \cite[Theorem 1.1, Remark 1.2]{Chenlu1} and \cite[Theorem 1.2]{Chenlu2} for details. Based on their results and our Theorem \ref{main+++} \emph{(1)}, we can derive that
    \begin{align}
        C_{\mathrm{BE}}(n,m,\alpha)=C_{\mathrm{BE}}(n-m,0,\alpha-\frac{m}{2})\geq	\frac{\beta_{\alpha-\frac{m}{2}}}{n-m},\nonumber
    \end{align}
    and when $0<\alpha-m/2<\frac{1}{6}$, we have $\beta_{\alpha-\frac{m}{2}}=(\alpha-\frac{m}{2})\beta$. This lower bound is better than the lower bound we present in Theorem \ref{main+++} \emph{(2)} in the following sense: when $\alpha-\frac{m}{2}$ is fixed and $n-m$ is large, or when $\alpha-\frac{m}{2}$ approaches $0$, the lower bound is always comparable to the upper bound.  
\end{remark}
\vskip0.15in

	Based on the gradient stability result \eqref{fractional-trace stability}, we derive enhanced Sobolev embeddings for general domains. Specifically, we can add a remainder term equipped with some weak $L^q$-norm $\Vert \cdot\Vert_{L_w^q}$ (see Theorem \ref{cons1} and Theorem \ref{cons3}). Such type results were initially established by Br\'ezis and Lieb in \cite{Brezis-Lieb}. Later, Bianchi and Egnell \cite{bianchi1991} presented an alternative proof using gradient stability. Chen, Frank and Weth \cite{{chen-frank}} generalized their results to the fractional Sobolev inequality. We follow similar ideas in our proofs.
\vskip0.1in
	Another corollary of \eqref{fractional-trace stability} is a stability estimate of a sharp a priori estimate for the Neumann problem:
	\begin{equation}\label{dual0}
		\begin{cases}
			\Delta \mathcal{P}[u]= 0 & \text{in }\R_+^n, \\
			\frac{\partial \mathcal{P}[u]}{\partial t}=-u & \text{on }\partial\R_+^n,
		\end{cases}
	\end{equation}
	where $\R_+^n=\left\{(x,t)|x\in\R^{n-1},t>0\right\}$. Given any $u\in L^{\frac{2n-2}{n}}(\mathbb{R}^{n-1})$, we have
	\begin{equation}\label{dual1}
		\Vert u\Vert_{L^{\frac{2n-2}{n}}(\mathbb{R}^{n-1})}^2- \frac{4}{S(n,1,1)}\left\Vert \nabla \mathcal{P}[u]\right\Vert_{L^2(\mathbb{R}_{+}^n)}^2\geq C\inf_{v\in \mathcal{M}_{\mathrm{Neu}}}\Vert u-v\Vert_{L^{\frac{2n-2}{n}}(\mathbb{R}^{n-1})}^2.
	\end{equation}
	Here $\frac{4}{S(n,1,1)}$ is the sharp constant (see \eqref{snma}) and $\mathcal{M}_{\mathrm{Neu}}$ is the extremal manifold (see Theorem \ref{cons4} in the present  paper). The proof of \eqref{dual1} relies on a dual scheme developed by Carlen in \cite{carlen1}. Precisely, we show that this a priori estimate is the duality of the Escobar trace inequality \eqref{escobar}, \emph{i.e.} the inequality \eqref{trtr}
	with $m=\alpha=1$.

\vskip0.31in

{\bf $\diamondsuit$  In the critical point setting,}  we consider the Euler-Lagrange equation associated with the Escobar trace inequality \eqref{escobar}:
	\begin{equation}\label{eqeqeq}
		\begin{cases}
			\Delta u= 0 & \text{in }\R_+^n, \\
			\frac{\partial u}{\partial t}= -|u|^{\frac{2}{n-2}}u & \text{on }\partial\R_+^n,
		\end{cases}
	\end{equation}
	It was shown by Ou in \cite{ou} that any nontrivial nonnegative solution of \eqref{eqeqeq} is an Escobar bubble:
	\begin{equation}\label{esbu=1}
		U[z,\lambda](x,t)=(n-2)^{\frac{n-2}{2}}\left(\frac{\lambda}{(1+\lambda t)^2+\lambda^2 |x-z|^2)}\right)^{\frac{n-2}{2}}
	\end{equation}
	for some $z\in\R^{n-1},\lambda>0$. Motivated by Struwe's work, we initially establish a Struwe-type profile decomposition for \eqref{eqeqeq} (see Theorem \ref{ql}). Specifically, we demonstrate that any almost solution of \eqref{eqeqeq} can be approximated by multiple exact solutions. Furthermore, we present a nonnegative counterpart of this result  (see Theorem \ref{ql1}): every nonnegative almost solution of \eqref{eqeqeq} can be approximated by several Escobar bubbles. The transition from the general version to the nonnegative version requires a Br\'ezis-Lieb type lemma (see Lemma \ref{re3}), which was previously used in \cite{carlo}. Following ideas of Figalli and Glaudo in \cite{figalli}, we derive the following optimal quantitative profile decomposition:

	\begin{theorem}\label{thm:main_close}
		For $n = 3,\nu\in\N_+$ or $n\geq 3,\nu=1$, there exist a small constant $\delta=\delta(n,\nu)>0$ and
		a large constant $C=C(n,\nu)>0$ such that the following statement holds.
		Let $u\in H^1(\R^n_+)$ be a function such that
		\begin{equation}\label{condition_1}
			\norm*{\nabla u - \sum\limits_{i=1}^{\nu}\nabla \tilde U_i}_{L^2(\R_+^n)} \le \delta ,
		\end{equation}
		where $(\tilde U_i)_{1\le i\le \nu}$ is a $\delta$-interacting family of Escobar bubbles (see Definition \ref{interaction!} below).
		Then there exist $\nu$ Escobar bubbles $U_1,U_2,\dots, U_\nu$  such that
		\begin{equation*}
			\norm*{\nabla u - \sum\limits_{i=1}^{\nu}\nabla U_i}_{L^2(\R_+^n)} \le C\norm{\lapl u+\abs{u}^{p-1}u}_{H^{-1}},
		\end{equation*}
		where $p=\frac{n}{n-2}$. Furthermore, for any $i\not= j$, the interaction between the bubbles can be estimated as
		\begin{equation}\label{eq:interaction_estimate_statement}
			\int_{\R^n}U_i^pU_j \le C\norm{\lapl u+\abs{u}^{p-1}u}_{H^{-1}} .
		\end{equation}
	\end{theorem}

\vskip0.3in
The term in the right-hand side of \eqref{eq:interaction_estimate_statement} is defined by
	\begin{equation}\label{dualnorm}
		\norm{\lapl u+\abs{u}^{\frac{2}{n-2}}u}_{H^{-1}(\R_+^n)}=\sup_{v\in H^1(\R_+^n)}\frac{\int_{\R_+^n}\nabla u\cdot \nabla v-\int_{\R^{n-1}}|u|^{\frac{2}{n-2}}uv}{\norm{v}_{H^1(\R_+^n)}},
	\end{equation}
	and the weak-interaction is defined as follows.
	\begin{definition}[interaction of Escobar bubbles]\label{interaction!}
		Let $U_1=U[z_1,\lambda_1]$,...,$U_\nu=U[z_\nu,\lambda_\nu]$ be a family of Escobar bubbles (see \eqref{esbu=1}). We say that the family is $\delta$-interacting for some $\delta >0$ if
		\begin{equation}\label{interaction}
			\mu_{ij}\defeq\min\left(\frac{\lambda_i}{\lambda_j},\frac{\lambda_j}{\lambda_i},\frac{1}{\lambda_i\lambda_j |z_i-z_j|^2}\right)\leq \delta,\quad \forall\, 1\le i\neq j\le \nu.
		\end{equation}
		Moreover, we say the family with some positive coefficients $\alpha_1,\dots,\alpha_\nu$ is $\delta$-interacting, if \eqref{interaction} holds and
		\begin{equation*}
			\max_{1\le i\le \nu} \abs{\alpha_i-1} \le \delta.
		\end{equation*}
	\end{definition}
	Thanks to qualitative results, the weak-interacting assumption \eqref{condition_1} in Theorem \ref{thm:main_close} can be substituted with nonnegativity. Specifically, we have the following stability result for nonnegative functions.
	\begin{theorem}\label{them4}
		Assume $n\geq 3,\nu=1$ or $n=3,\nu\geq 2$. Set $p=\frac{n}{n-2}$. Then there exists a constant $C=C(n,\nu)$ such that, for any nonnegative function $u\in H^1(\R^n_+)$ with
		\begin{equation}\label{rty}
			(\nu-\frac{1}{2})S_{\mathrm{E}}(n)^{n-1}\leq \int_{\R^n_+}|\nabla u|^2\leq(\nu+\frac{1}{2})S_{\mathrm{E}}(n)^{n-1},
		\end{equation}
		there exist $\nu$ Escobar bubbles $U_1,U_2,\dots, U_\nu$ such that
		\begin{equation*}
			\norm*{\nabla u - \sum\limits_{i=1}^{\nu}\nabla U_i}_{L^2(\R_+^n)} \le C\norm{\lapl u+u^p}_{H^{-1}}.
		\end{equation*}
		Furthermore, for any $i\not= j$, the interaction between the bubbles can be estimated as
		\begin{equation*}
			\int_{\R^{n}}U_i^pU_j \le C\norm{\lapl u+u^p}_{H^{-1}} .
		\end{equation*}
 In particular, the  linear decay is the sharp.
	\end{theorem}
\vskip0.2in
\begin{remark}
     In this work, we only address the stability in the critical point setting for the case $\alpha=m=1$. Investigating stability in the general case where $0<m<n$ and $\frac{m}{2}<\alpha<\frac{n}{2}$ is an interesting and challenging problem. When $\alpha\notin\N_+$ or $m\geq 2$, the Euler-Lagrange equation corresponding to the inequality \eqref{trtr} is not explicit, and it may require more effort to find an appropriate setting for stability. When $\alpha\in\N_+$, $\alpha\geq 2$ and $m=1$, the Euler-Lagrange equation is the following poly-harmonic equation on $\R_+^n$ with mixed boundary conditions:
     \begin{align}\label{rq1}
         \begin{cases}
			\Delta^\alpha u= 0 & \text{in }\R_+^n, \\
			\frac{\partial \Delta^{k_1} u}{\partial t}=\Delta^{k_2}u=0,\; & \text{on }\partial\R_+^n,\\
            (-1)^{\alpha}\frac{\partial \Delta^{\alpha-1}u}{\partial t}=|u|^\frac{2(2\alpha-1)}{n-2\alpha}u & \text{on }\partial\R_+^n,
		\end{cases}
     \end{align}
     where $k_1,k_2\in\N_+$, $\frac{\alpha-1}{2}\leq k_1\leq \alpha-2$ and $\frac{\alpha}{2}\leq k_2\leq \alpha-1$. To the best of our knowledge, the classification of positive solutions to the above system \eqref{rq1} is open (we refer to \cite{Dai,Sun} and the references therein for classification results related to poly-harmonic equations with other boundary conditions).
     
     In addition to the lack of classification results, the establishment of non-degeneracy results -- which are crucial for the derivation of quantitative stability -- is also challenging (see Remark \ref{ert} below for details). When $\alpha=1$, this can be derived from known results related to the Steklov eigenvalue problem in $\B^n$ through a suitable conformal transformation (see Section \ref{sec7} for details). For the case $\alpha\geq2$, we conjecture that the non-degeneracy result might also relate to certain eigenvalue problems, although to the best of our knowledge, such results remain open.
\end{remark}
\vskip0.2in
\begin{remark}\label{ert}
    To illustrate the importance of the non-degeneracy result, we begin by recalling a classical method for deriving stability estimates, which dates back to the pioneering work \cite{bianchi1991}. In general, to quantify an inequality with known extremizers, one first establishes a concentration-compactness type principle. This allows the analysis to be restricted to functions that are close to a given extremizer. The next step is to expand the inequality around the extremizer. By isolating the principal terms, one arrives at a spectral gap inequality, which corresponds to the eigenvalue problem associated with the linearized Euler–Lagrange equation near the extremizer. The non-degeneracy result typically refers to the classification of the lower-order eigenvalues and the structure of the corresponding eigenspaces of this linearized operator. Once such a classification is obtained, the spectral gap inequality can often be derived using appropriate orthogonality conditions.
\end{remark}

\vskip0.2in
Similar to \eqref{mini}, we study the following minimization problem in the critical point setting:
	\begin{equation}\label{mini2}
		C_{\mathrm{CP}}(n,\nu)=\inf_{u}\frac{\norm{\lapl u+\abs{u}^{\frac{2}{n-2}}u}_{H^{-1}(\R_+^n)}}{\dis(u,\mathcal{M}_{\mathrm{E}}^\nu)}
	\end{equation}
for $u\in H^1(\R_+^n)\backslash\mathcal{M}_{\mathrm{E}}^\nu $ satisfying \eqref{rty}. Here $\mathcal{M}_{\mathrm{E}}^\nu:= \{(U[z_1,\lambda_1], ...,U[z_\nu,\lambda_\nu]):  z_i\in\R^{n-1},\lambda_i>0  \}.$ In the case $\nu=1$, we obtain a strict upper bound.

\vskip0.23in

\begin{theorem}\label{them5}
Let $n\geq 3$ and $C_{\mathrm{CP}}(n,1)$ denotes the optimal constant in \eqref{mini2} with $\nu=1$. Then
		\begin{equation*}
			C_{\mathrm{CP}}(n,1)<\frac{2}{n+2}.
		\end{equation*}
\end{theorem}

\vskip0.2in
\begin{remark}
The combination of the above  Theorem  \ref{them5}  and the Proposition \ref{asymp} in Section \ref{sec7} of the current paper indicates that: Any minimizing sequence must be away from $\mathcal{M}_{\mathrm{E}}^1$. However, whether minimizers exist is still unknown.
\end{remark}

\vskip0.2in
The proof of Theorem \ref{them5} is inspired by De Nitti and K\"onig's work in \cite{konig2}. Our primary methods involve a precise spectral gap inequality, a well-defined expression for the  $H^{-1}$ norm and a third-order expansion. To derive the spectral gap inequality, we employ a conformal transformation, linking non-degeneracy results on $\R_+^n$ to the classical Steklov eigenvalue problem on the unit ball $\B^n$. We will use the operator $\mathcal{P}$ in \eqref{dual0} to express the $H^{-1}$ norm, which makes it clear to do the third-order expansion.

\vskip0.24in

\subsection{Structure of the paper}
	
	After a section of notations and preliminaries, in Section \ref{sec3} we present the proof of the functional stability result \eqref{fractional-trace stability}. A discussion about the optimal stability constant and the existence of minimizers will also be given. Based on results in Section \ref{sec3}, we derive some corollaries in Section \ref{sec4}, including some refined Sobolev embeddings and a stability result of a priori estimate. Section \ref{sec5} and Section \ref{sec6} are devoted to qualitative and quantitative profile decomposition respectively. We focus on both the general and the nonnegative case. Finally, in Section \ref{sec7}, we provide some improved estimates of the stability constant for critical point in one bubble case. We begin with the derivation of an explicit spectral gap inequality. Based on it, we can proceed to give an upper bound of $C_{\mathrm{CP}}(n,1)$ in terms of specturm constants.

   \vskip0.3in

\noindent$\textbf{Notations.}$ Throughout this paper, for simplicity, we consistently use $C$ and $c$ to denote positive quantities depending only on the parameters $n,m,\alpha$ and possibly $\nu$. The values of $C$ and $c$ may vary from line to line. Generally, $C$ is used to denote large numbers, while $c$ represents small numbers. For any two quantities $a_1$ and $a_2$, we shall write $a_1\lesssim a_2$ (resp. $a_1\gtrsim a_2$) if $a_1\leq Ca_2$ (resp. $a_1\geq ca_2)$. Also we say $a_1\approx a_2$ if $a_1\lesssim a_2$ and $a_1\gtrsim a_2$. Given any quantity $\varepsilon$, we say $a_1\lesssim_\varepsilon a_2$ if $a_1\lesssim C(\varepsilon)a_2$, where $C(\varepsilon)$ is a positive number depending only on $\varepsilon$. We similarly define $\gtrsim_\varepsilon$ and $\approx_\varepsilon$. 

For convenience, we will also omit the domain of norms in the following sections when no confusion arises.

\vskip0.3in
	
	\section{Preliminaries}
	We begin with setting the notations and definitions we need in our paper.
	
	For $n\ge 1$ and $0<\alpha<\frac{n}{2}$, we define the space $D_\alpha(\R^n)$ as the space of tempered distributions $\mathcal{S}'(\R^n)$ whose Fourier transform belongs to $L^2(\R^n,|k|^{2\alpha}\,dk)$. That is to say, for any element $T\in\mathcal{S}'(\R^n)$ of it, there exists a function $\hat{f}\in L^2(\R^n,|k|^{2\alpha}\,dk)$ such that for all $\phi\in \mathcal{S}(\R^n)$,
	\begin{equation*}
		\widehat{T}(\phi) =: T(\hat{\phi}) = \int_{\R^n} \hat{f}(k)\phi(k)\,dk.
	\end{equation*}
	Here the Fourier transform is defined by
	\begin{equation*}
		\hat{f}(k)=\int_{\R^n}f(x)e^{-2\pi ix\cdot k}\,dx.
	\end{equation*}
	$D_\alpha(\R^n)$ equipped with the following norm is a Hilbert space, in which $\mathcal{S}(\R^n)$ is dense:
	\begin{equation*}
		\norm{f}_{D_\alpha(\R^n)} := \left(\int_{\R^n} \abs{\hat f(k)}^2 \abs{2\pi k}^{2\alpha}\,dk\right)^{1/2}.
	\end{equation*}
	We should claim that, when $\alpha\in\N_+$, the norm above is equal to classical Sobolev norm:
	\begin{equation*}
		\norm{\Delta^{\frac{\alpha}{2}}f}_{L^2(\R^n)} = \left(\int_{\R^n} |\Delta^{\frac{\alpha}{2}}f(x)|^2\,dx\right)^{1/2}.
	\end{equation*}
	
	For $f\in\mathcal{S}(\R^n)$, we define the restriction of $f$ to the $(n-m)$-dimensional hyperplane given by $\{x\in\R^n:x_{n-m+1}=\cdots=x_n=0\}$ as
	\begin{equation*}
		(\tau_mf)(x_1,\cdots,x_{n-m}) :=f(x_1,\cdots,x_{n-m},0,\cdots,0).
	\end{equation*}
	For general $f\in D_\alpha(\R^n)$, $\tau_mf$ can be uniquely defined through a density argument.
	
	For $n\ge 3$, in modern terminology, $2^*$ always denotes the critical Sobolev exponent $\frac{2n}{n-2}$. In our settings, to avoid confusions, we will always use $2^{\dagger}$ to denote the trace critical exponent $\frac{2(n-1)}{n-2}$. We also define $$p=2^{\dagger}-1=\frac{n}{n-2}$$ for simplicity.
	
	In the whole paper, we denote by $S(n,m,\alpha)$ the sharp constant of the fractional Sobolev trace inequality with respect to parameters $0\leq m<n$ and $\frac m2 < \alpha < \frac n2$, i.e.
	\begin{equation*}
		S(n,m,\alpha) := \inf \left\{ \frac{\norm{f}_{D_\alpha(\R^n)}^2}{\norm{\tau_m f}_{L^{s}(\R^{n-m})}^2}: f\in D_\alpha(\R^n)\backslash\{0\},\, s=\frac{2(n-m)}{n-2\alpha}\right\}.
	\end{equation*}
	Note that, under this notation, $S(n,0,\alpha)$ is the sharp constants of fractional Sobolev inequality.
	
	A useful remark is that, when $m=\alpha=1$, the trace inequality above is equivalent to the original Escobar trace inequality in \cite{Escobar} via a simple reflection:
	\begin{equation*}
		S_{\mathrm{E}}(n)\norm{f}_{L^{2^\dagger}(\R^{n-1})}^2 \le \norm{f}_{H^1(\R_+^n)}^2,
	\end{equation*}
	where $S_{\mathrm{E}}(n)=\frac{S(n,1,1)}{4}$ is the sharp constant, and $H^1(\R_+^n)$ is the closure of $C_c^\infty(\R^n)|_{\R_+^n}$ with respect to the norm
	\begin{equation*}
		\norm{f}_{H^1(\R_+^n)}=\left(\int_{\R_+^n}|\nabla f|^2\,dx\right)^{\frac{1}{2}}.
	\end{equation*}
	The extremal functions, up to normalization, are precisely the so-called Escobar bubbles below.
	
	
	
	
	For any $z\in\R^{n-1}$ and $\lambda>0$, the Escobar bubble $U[z,\lambda]$ is defined as
	\begin{equation}\label{esco}
		U[z,\lambda](x,t)=(n-2)^{\frac{n-2}{2}}\left(\frac{\lambda}{(1+\lambda t)^2+\lambda^2 |x-z|^2)}\right)^{\frac{n-2}{2}},\quad x\in\R^{n-1},\,t>0.
	\end{equation}
	The corresponding manifold is denoted by $\mathcal{M}_{\mathrm{E}}$. We state some important properties of Escobar bubbles $U=U[z,\lambda]$.
	\begin{itemize}
		\item They are extremal functions of the Escobar trace inequality:
		\begin{equation*}
			\int_{\R^n_+} \abs{\nabla U}^2 = \int_{\partial\R_+^n} U^{2^\dagger} = S_{\mathrm{E}}(n)^{n-1}.
		\end{equation*}
		
		\item They and their derivatives satisfy
		\begin{equation*}
			\begin{cases}
				\Delta U= 0 & \text{in }\R_+^n, \\
				\frac{\partial U}{\partial t}= -U^p & \text{on }\partial\R_+^n,
			\end{cases}\quad
			\begin{cases}
				\Delta (\partial_s U)= 0 & \text{in }\R_+^n, \\
				\frac{\partial (\partial_s U)}{\partial t}= -pU^{p-1}\partial_s U & \text{on }\partial\R_+^n,
			\end{cases}
		\end{equation*}
		for $s=\lambda,z_1,\cdots,z_{n-1}$.
		
		\item We have the following expression for $\partial_\lambda U$:
		\begin{equation*}
			\partial_\lambda U(x,t) = \frac{n-2}{2\lambda} U(x,t) \left(\frac{1-\lambda^2 t^2 - \lambda^2\abs{x-z}^2}{(1+\lambda t)^2 + \lambda^2 \abs{x-z}^2}\right),
		\end{equation*}
		which yields the estimate of $\abs{\partial_\lambda U}$:
		\begin{equation*}
			\abs{\partial_\lambda U} \le \frac{n-2}{2\lambda}\abs{U}.
		\end{equation*}
		
		\item Symmetries: for $k\in\N$ and exponents set $(\gamma_i)_{i=1}^k$ with $\gamma_1+\cdots+\gamma_k = 2^\dagger$, $k$ Escobar bubbles $U[z_1,\lambda_1],\dots,U[z_k,\lambda_k]$ satisfy
		\begin{equation*}
			\int_{\partial\R_+^n} U[z_1,\lambda_1]^{\gamma_1}\cdots U[z_k,\lambda_k]^{\gamma_k} = \int_{\partial\R_+^n} U[0,1]^{\gamma_1}\prod_{j=2}^k U[\lambda_1(z_j-z_1),\lambda_j/\lambda_1]^{\gamma_j}.
		\end{equation*}
		In addition, when $k=2$, we have
		\begin{equation*}
			\int_{\R^n_+} \nabla U[z_1,\lambda_1]\cdot \nabla U[z_2,\lambda_2] = \int_{\R^n_+} \nabla U[0,1] \cdot \nabla U[\lambda_1(z_2-z_1),\lambda_2/\lambda_1].
		\end{equation*}
	\end{itemize}
	With these symmetries, for the stability problem, it is enough to consider $U[0,1]$ instead of a generic Escobar bubble. When dealing with multiple bubbles, we need to evaluate their interaction (recall Definition \ref{interaction!}). The following lemma is a general version of \cite[Lemma B.1]{figalli}, and we provide a proof here for completeness.
	\begin{lemma}\label{estimate of bubbles}
		Given $n\ge 1$, let us fix $\alpha+\beta=2n$ with $\alpha,\beta \ge 0$, $\lambda \in (0,1]$ and $z\in \R^n$, set $D\defeq |z|$. If $|\alpha-\beta|\ge \epsilon$ for some $\epsilon >0$, then
		\begin{equation*}
			\int_{\R^n} \left(\frac{1}{1+|x|^2}\right)^{\alpha/2}\left(\frac{\lambda}{1+\lambda^2|x-z|^2}\right)^{\beta/2} \approx_{\epsilon} \mu^{\frac{\min(\alpha,\beta)}{2}},
		\end{equation*}
		where $\mu = \min\left\{\lambda,\frac{1}{\lambda D^2}\right\}$, If instead $\alpha=\beta=n$, then
		\begin{equation*}
			\int_{\R^n} \left(\frac{1}{1+|x|^2}\right)^{\alpha/2}\left(\frac{\lambda}{1+\lambda^2|x-z|^2}\right)^{\beta/2} \approx (1-\ln\mu)\mu^{n/2}.
		\end{equation*}
	\end{lemma}
    \vskip0.2in
    \noindent{\bf  Proof of Lemma \ref{estimate of bubbles}.}
        We split the proof into two cases.
        
        \textbf{Case $1$: $D\le \lambda^{-1}$.} In the ball $B(0,2\lambda^{-1})$, we have $0\le |x-z| \le 3\lambda^{-1}$, which implies $1+\lambda^2|x-z|^2 \approx 1$. In the complement $B(0,2\lambda^{-1})^c$, since $\lambda\in (0,1]$, it follows that $1+|x|^2\approx |x|^2$ and $1+\lambda^2|x-z|^2 \approx \lambda^2|x|^2$. Combining these estimates yields:
        \begin{align*}
            & \int_{\R^n} \left(\frac{1}{1+|x|^2}\right)^{\alpha/2}\left(\frac{\lambda}{1+\lambda^2|x-z|^2}\right)^{\beta/2} \\
            \approx{}& \int_{B(0,2\lambda^{-1})} (1+|x|^2)^{-\alpha/2}\lambda^{\beta/2} + \int_{B(0,2\lambda^{-1})^c} |x|^{-\alpha}(\lambda|x|^2)^{-\beta/2} \\
            \approx{}& \int_0^{2\lambda^{-1}} (1+t^2)^{-\alpha/2}\lambda^{\beta/2} t^{n-1}\,dt + \int_{2\lambda^{-1}}^\infty \lambda^{-\beta/2}t^{n-1-\alpha-\beta}\,dt \\
            \approx{}& \lambda^{\beta/2}\int_1^{2\lambda^{-1}} t^{n-1-\alpha}\,dt + \lambda^{\alpha+\frac{\beta}2 -n}.
        \end{align*}
        \begin{itemize}
            \item If $\alpha\ge n+\epsilon$, the expression becomes comparable to $\lambda^{\beta/2}$.
            \item If $\alpha\le n-\epsilon$, the expression becomes comparable to $\lambda^{\alpha+\frac{\beta}{2}-n} = \lambda^{\alpha/2}$.
            \item If $\alpha=n=\beta$, the expression becomes comparable to $\lambda^{n/2}(1-\ln\lambda)$.
        \end{itemize}
    
        \textbf{Case $2$: $D\ge \lambda^{-1}\ge 1$.} In $B(0,D/2)$, we have $1+\lambda^2|x-z|^2 \approx \lambda^2|z|^2=\lambda^2D^2$. Similarly, in $B(z,D/2)$, we find $1+|x|^2 \approx |z|^2=D^2$. These estimates hold in $B(0,2D)\backslash \big(B(0,D/2)\cup B(z,D/2)\big)$, and outside $B(0,2D)$, we have $1+|x|^2 \approx |x|^2$ and $1+\lambda^2|x-z|^2\approx \lambda^2|x|^2$. Combining these estimates yields:
        \begin{align*}
            & \int_{\R^n} \left(\frac{1}{1+|x|^2}\right)^{\alpha/2}\left(\frac{\lambda}{1+\lambda^2|x-z|^2}\right)^{\beta/2} \\
            \approx{}& \int_{B(0,D/2)} (1+|x|^2)^{-\alpha/2}(\lambda D^2)^{-\beta/2} + \int_{B(0,D/2)} D^{-\alpha}(\lambda^{-1}+\lambda|x|^2)^{-\beta/2}\\
            & + D^n\cdot D^{-\alpha}(\lambda D^2)^{-\beta/2} + \int_{B(0,2D)^c} |x|^{-\alpha}(\lambda|x|^2)^{-\beta/2}\\
            \approx{}& (\lambda D^2)^{-\beta/2}\int_0^{D/2} (1+t^2)^{-\alpha/2}t^{n-1}\,dt + \lambda^{\frac{\beta}2 - n}D^{-\alpha}\int_0^{\lambda D/2} (1+t^2)^{-\beta/2}t^{n-1}\,dt\\
            & + \lambda^{-\beta/2}D^{-n} + \lambda^{-\beta/2}\int_{2D}^\infty t^{n-1-\alpha-\beta}\,dt \\
            \approx{}& (\lambda D^2)^{-\beta/2} \int_{1/4}^{D/2} t^{n-1-\alpha}\,dt + (\lambda D^2)^{-\alpha/2}\int_{1/4}^{D\lambda/2} t^{n-1-\beta}\,dt + \lambda^{-\beta/2}D^{-n}.
      \end{align*}
      \begin{itemize}
        \item If $\alpha\ge n+\epsilon$, the expression becomes comparable to $(\lambda D^2)^{-\beta/2}$.
        \item If $\alpha\le n-\epsilon$, the expression becomes comparable to $(\lambda D^2)^{-\alpha/2}$.
        \item If $\alpha=n=\beta$, the expression becomes comparable to $(1+\ln D)(\lambda D^2)^{-n/2}$, which is also comparable to $(1+\ln(\lambda D^2))(\lambda D^2)^{-n/2}$.
      \end{itemize}
      The proof is complete.
    \hfill$\Box$
    \vskip0.2in
	\begin{remark}
		For general $z_1,z_2\in \R^n$ and $\lambda_1\ge \lambda_2$, if $|\alpha-\beta|\ge \epsilon$ for some $\epsilon >0$, from the above lemma we have
		\begin{align*}
			\int_{\R^n} \left(\frac{\lambda_1}{1+\lambda_1^2|x-z_1|^2}\right)^{\alpha/2}\left(\frac{\lambda_2}{1+\lambda_2^2|x-z_2|^2}\right)^{\beta/2} &= \int_{\R^n} \left(\frac{1}{1+|x|^2}\right)^{\alpha/2}\left(\frac{\lambda}{1+\lambda^2|x-z|^2}\right)^{\beta/2}\\
			&\approx_\epsilon\ \mu^{\frac{\min(\alpha,\beta)}{2}}
		\end{align*}
		with $\lambda=\frac{\lambda_2}{\lambda_1}\in(0,1]$ and $z=\lambda_1(z_2-z_1)$. So $\lambda D^2 = \lambda_1\lambda_2|z_1-z_2|^2$ and $\mu=\min\left\{\frac{\lambda_2}{\lambda_1},\frac{1}{\lambda_1\lambda_2|z_1-z_2|^2}\right\}$. If instead $\lambda_1<\lambda_2$, the result is valid with $\mu=\min\left\{\frac{\lambda_1}{\lambda_2},\frac{1}{\lambda_1\lambda_2|z_1-z_2|^2}\right\}$.
	
	Thus, in order to identify whether the interactions of bubbles are weak, it is natural to introduce the definition of weak-interaction (see Definition \ref{interaction}).
    \end{remark}
	\vskip0.2in
	It is remarkable that the interaction of bubbles concentrates on some area depending on the bubbles:
	\begin{lemma}\label{estimate of bubbles 2}
		Given $n\ge 3$ and two Escobar bubbles $U_1=U[z_1,\lambda_1]$ and $U_2=U[z_2,\lambda_2]$ with $\lambda_1\ge \lambda_2$, it holds that
		\begin{equation*}
			\int_{\partial\R_+^n} U_1^p U_2 \approx \int_{B^{n-1}(z_1,\lambda_1^{-1})\cap \partial\R_+^n} U_1^p U_2.
		\end{equation*}
	\end{lemma}
	This lemma is similar to \cite[Corollary B.4]{figalli}, and we give a proof here for completeness.
    \vskip0.2in
    \noindent{\bf  Proof of Lemma \ref{estimate of bubbles 2}.}
        Without loss of generality, we assume $z_1=0$, $\lambda_1=1$, and denote $\lambda\coloneqq \lambda_2,\, z\coloneqq z_2$. Since $\lambda\le 1$, for any $x\in B^{n-1}(0,1)$, we have:
        \[ 1+\lambda^2|x-z|^2 \le 1 + 2\lambda^2(1+|z|^2) \le 3(1+\lambda^2|z|^2);\]
        \[ 1+ \lambda^2|x-z|^2 \ge \begin{cases}
            1 \ge \frac 15 (1+\lambda^2|z|^2), & |z|\le 2; \\
            1 + \frac 14 \lambda^2|z|^2 \ge \frac 14 (1+\lambda^2|z|^2), & |z|>2\ge 2|x|.
        \end{cases}\]
        Thus, we conclude that $U_2(x,0) \approx U_2(0,0)$. This leads to:
        \[\int_{B^{n-1}(0,1)\cap \partial\R_+^{n}} U_1^p U_2 \approx U_2(0)\int_{B^{n-1}(0,1)}\left(\frac{1}{1+|x|^2}\right)^{\frac n2}\approx \left(\frac{\lambda}{1+\lambda^2|z|^2}\right)^{\frac{n-2}{2}}.\]
        By applying Lemma \ref{estimate of bubbles} (replacing $n$ by $n-1$ and taking $\alpha=n,\, \beta=n-2$), we obtain:
        \[ \int_{\partial\R^n_+} U_1^p U_2 \approx \min\{\lambda,\frac{1}{\lambda|z|^2}\}^{\frac{n-2}{2}}\approx \left(\frac{\lambda}{1+\lambda^2|z|^2}\right)^{\frac{n-2}{2}}.\]
        This completes the proof.
    \hfill$\Box$
	
	\section{Stability of the fractional Sobolev trace inequality}\label{sec3}
	In this section, we are devoted to analyzing the stability of the following sharp fractional Sobolev trace inequality given by Einav and Loss in \cite{fractr}:
	\begin{theorem}
		Let $0\le m < n$ and $\frac m2 < \alpha < \frac n2$. For any $f\in D_\alpha(\R^n)$, we have
		\begin{equation*}
			S(n,m,\alpha)\norm{\tau_mf}_{L^{\frac{2(n-m)}{n-2\alpha}}(\R^{n-m})}^2 \le \norm{f}_{D_\alpha(\R^n)}^2,
		\end{equation*}
		where
		\begin{equation*}
			S(n,m,\alpha) = 2^{2\alpha}\pi^{\alpha}\frac{\Gamma(\alpha)\Gamma(n/2+\alpha-m)}{\Gamma(n/2-\alpha)\Gamma(\alpha-m/2)}\left(\frac{\Gamma(n-m)}{\Gamma((n-m)/2)}\right)^{\frac{m-2\alpha}{n-m}}
		\end{equation*}
		and the equality holds if and only if $f(x)=f(x'',x')$, $x'\in \R^{m},\, x''\in \R^{n-m}$ is proportional to
		\begin{equation*}
			\int_{\R^{n-m}} \frac{1}{(|x'|^2+|x''-y''|^2)^{(n-2\alpha)/2}}\frac{1}{(\gamma^2+|y''-a|^2)^{(n+2\alpha-2m)/2}}\,dy''
		\end{equation*}
		for some $a\in\R^{n-m}$ and $\gamma\neq 0$. We denote the set of minimizers by $\mathcal{M}_{n,m,\alpha}$.
	\end{theorem}
	
	It was obtained by showing the following two inequalities and combining them together:
	\begin{theorem}[fractional Sobolev inequality]\label{thm: ineq1}
		Let $0< \alpha < n/2$. For any $f\in D_\alpha(\R^n)$,
		\begin{equation*}
			S(n,0,\alpha) \norm{f}_{L^s(\R^n)}^2 \le \norm{f}_{D_\alpha(\R^n)}^2,
		\end{equation*}
		where $s=\frac{2n}{n-2\alpha}$ and the equality holds if and only if
		\begin{equation*}
			f(x) = A(\gamma^2 + |x-a|^2)^{-\frac{n-2\alpha}{2}}
		\end{equation*}
		for some $A\in\R$, $\gamma\neq 0$ and $a\in\R^n$. The set of minimizers is denoted by $\mathcal{M}_{n,0,\alpha}$.
	\end{theorem}
	
	\begin{theorem}[reduction principle]\label{thm: ineq2}
		Assume $0\leq m<n$ and $\frac m2 < \alpha < \frac n2$. Then the trace $\tau_m$ has a unique extension to a bounded operator $\tau_m: D_\alpha(\R^n) \to D_{\alpha-m/2}(\R^{n-m})$. Moreover, for any $f\in D_\alpha(\R^n)$,
		\begin{equation}\label{inequality of thm2}
			R(n,m,\alpha) \norm{\tau_m f}_{D_{\alpha-m/2}(\R^{n-m})}^2 \le \norm{f}_{D_\alpha(\R^n)}^2,
		\end{equation}
		where
		\begin{equation*}
			R(n,m,\alpha) = 2^{m}\pi^{m/2}\frac{\Gamma(\alpha)}{\Gamma((2\alpha-m)/2)},
		\end{equation*}
		and the equality holds if and only if
		\begin{equation}\label{trace equality: f}
			\hat{f}(k_1,k_2) = \frac{\hat g(k_1)}{(|k_1|^2+|k_2|^2)^\alpha}
		\end{equation}
		with
		\begin{equation*}
			\int_{\R^{n-m}} \frac{\abs{\hat g(k_1)}^2}{\abs{k_1}^{2\alpha-m}}\,dk_1 < \infty.
		\end{equation*}
		Here $k_1\in\R^{n-m}$ and $k_2\in\R^m$.
	\end{theorem}
	To get the stability result (Theorem \ref{main thm1}), we also need a stability result of the fractional Sobolev inequality proved by Chen, Frank and Weth in \cite{chen-frank}:
	\begin{theorem}\label{thm: stability1}
		Let $0<\beta<N$, $\mathcal{M} = \mathcal{M}_{N,0,\beta}$, then $C_{\mathrm{BE}}(N,0,\beta) \in(0,1)$ and we have
		\begin{equation}\label{fractional stability}
			\dis^2(u,\mathcal{M}) \ge \norm{u}_{D_{\beta}(\R^N)}^2 - S(N,0,\beta)\norm{u}_{L^\frac{2N}{N-2\beta}(\R^N)}^2 \ge C_{\mathrm{BE}}(N,0,\beta) \dis^2(u,\mathcal{M})
		\end{equation}
		for all $u\in D_{\beta}(\R^N)$, where $\dis(u,\mathcal{M}) = \min \{\norm{u-\varphi}_{D_{\beta}(\R^N)}:\varphi\in\mathcal{M}\}$.
	\end{theorem}
	
	The main idea of deriving Theorem \ref{main thm1} is using Theorems \ref{thm: ineq2} and \ref{thm: stability1} to construct $v$ directly. Precisely, any function $f\in D_\alpha(\R^n)$ can be orthogonally decomposed into two parts $g$ and $f-g$, where $g$ attains the equality of \eqref{inequality of thm2} with the same trace as $f$, and $\norm{f-g}_{D_\alpha}$ can be calculated directly. Then, estimating $\dis(\tau_m g,\mathcal{M}_{n-m,0,\alpha-m/2})$ by \eqref{fractional stability} leads to the result.


\vskip0.3in

\noindent{\bf  Proof of Theorem \ref{main thm1}. } For a given $f\in D_\alpha(\R^n)$, we claim that there exists a function $g\in D_\alpha(\R^n)$ such that $g$ takes equality of \eqref{inequality of thm2} and $\tau_m g=\tau_m f$. Set
		\begin{equation*}
			\hat{g}(k_1,k_2) = C_1(m,\alpha)^{-1}\frac{\widehat{\tau_m f}(k_1) |k_1|^{2\alpha-m}}{(|k_1|^2+|k_2|^2)^{\alpha}},\quad k_1\in\R^{n-m}, k_2\in\R^m,
		\end{equation*}
		where $C_1(m,\alpha)= \int_{\R^m}(1+|x|^2)^{-\alpha}\,dx$ is a constant. Since $\tau_m f\in D_{\alpha-m/2}(\R^{n-m})$, if we take $\hat{h}(k_1)\defeq C^{-1}(m,\alpha)\widehat{\tau_m f}(k_1) |k_1|^{2\alpha-m}$, then $h \in D_{\alpha-m/2}(\R^{n-m})$, and $g$ satisfies \eqref{trace equality: f}. On the other hand,
		\begin{equation*}
			\begin{aligned}
				\widehat{\tau_m g}(k_1) = &\int_{\R^m} \hat{g}(k_1,k_2)\,dk_2 = \int_{\R^m} \frac{\hat{h}(k_1)}{(|k_1|^2+|k_2|^2)^{\alpha}}\,dk_2\\
				= &C_1(m,\alpha)\hat{h(k_1)}|k_1|^{-2\alpha+m} = \widehat{\tau_m f}(k_1),
			\end{aligned}
		\end{equation*}
		thus $\tau_m g = \tau_m f$. This proves the claim. Next, we compute the distance from $f$ to $g$:
		\begin{equation}\label{distance of f and g}
			\begin{split}
				\norm{f-g}_{D_\alpha}^2 ={}& \norm{f}_{D_\alpha}^2 + \norm{g}_{D_\alpha}^2 - 2\int_{\R^n} \overline{\hat{f}(k_1,k_2)}\hat{g}(k_1,k_2) (|2\pi k_1|^2+|2\pi k_2|^2)^{\alpha}\,dk_1dk_2 \\
				={}& \norm{f}_{D_\alpha}^2 + R(n,m,\alpha)\norm{\tau_mf}_{D_{\alpha-m/2}}^2\\ &- 2C_1(m,\alpha)^{-1}(2\pi)^{m} \int_{R^{n-m}}\overline{\widehat{\tau_mf}(k_1)}\widehat{\tau_mf}(k_1)|2\pi k_1|^{2\alpha-m}\,dk_1 \\
				={}& \norm{f}_{D_\alpha}^2 - R(n,m,\alpha)\norm{\tau_mf}_{D_{\alpha-m/2}}^2.
			\end{split}
		\end{equation}
		Since $\tau_mf\in D_{\alpha-m/2}(\R^{n-m})$, picking $N=n-m$ and $\beta=\alpha-\frac m2$ in Theorem \ref{thm: stability1}, we get some $\tilde{v}\in \mathcal{M}_{N,0,\beta}$ such that
		\begin{equation*}
			\begin{split}
				\norm{\tau_m f-\tilde{v}}_{D_{\alpha-m/2}}^2 \le{}& C_{\mathrm{BE}}(N,0,\beta)^{-1} \left(\norm{\tau_m f}_{D_{\alpha-m/2}}^2 - S(N,0,\beta)\norm{\tau_m f}_{L^s}^2\right),
			\end{split}
		\end{equation*}
		where $s=\frac{2(n-m)}{n-2\alpha}$. Finally, taking $v\in D_\alpha(\R^n)$ that attains equality of \eqref{inequality of thm2} and $\tau_m v=\tilde{v}$, we know that $g-v$ also attains the equality by checking \eqref{trace equality: f}. Hence,
		\begin{equation}\label{distance of g and v}
			\begin{split}
				\norm{g-v}_{D_{\alpha}}^2 ={}& R(n,m,\alpha)\norm{\tau_mg-\tau_mv}_{D_{\alpha-m/2}}^2 \\
				={}& R(n,m,\alpha)\norm{\tau_mf-\tilde{v}}_{D_{\alpha-m/2}}^2\\
				\le{}& C_{\mathrm{BE}}(N,0,\beta)^{-1}R(n,m,\alpha)\left(\norm{\tau_m f}_{D_{\alpha-m/2}}^2 - S(N,0,\beta)\norm{\tau_m f}_{L^s}^2\right).
			\end{split}
		\end{equation}
		By similar computation as in \eqref{distance of f and g}, we know by $\tau_m f=\tau_m g$ that
		\begin{equation*}
			\begin{split}
				& \int_{\R^n}\overline{(\hat{f}-\hat{g})(k_1,k_2)}(\hat{g}-\hat{v})(k_1,k_2)(\abs{2\pi k_1}^2+\abs{2\pi k_2}^2)^\alpha\,dk_1dk_2\\
				={}& C_1(m,\alpha)^{-1}(2\pi)^m\int_{\R^{n-m}} \overline{(\widehat{\tau_m f}-\widehat{\tau_m g})(k_1)}(\widehat{\tau_mg}-\widehat{\tau_m v})(k_1)\abs{2\pi k_1}^{2\alpha-m} dk_1=0.
			\end{split}
		\end{equation*}
		Thus, combining \eqref{distance of f and g}, \eqref{distance of g and v}, \eqref{inequality of thm2}, and noticing that $C_{\mathrm{BE}}(N,0,\beta)\in(0,1)$, we get the desired estimate \eqref{fractional-trace stability}:
		\begin{align*}
			\norm{f-v}_{D_\alpha}^2
			={}& \norm{f-g}_{D_\alpha}^2 + \norm{g-v}_{D_\alpha}^2 \\
			\le{}& \norm{f}_{D_\alpha}^2 - R(n,m,\alpha)\norm{\tau_m f}_{D_{\alpha-m/2}}^2\\
			&+ C_{\mathrm{BE}}(N,0,\beta)^{-1}\,R(n,m,\alpha)\left(\norm{\tau_m f}_{D_{\alpha-m/2}}^2 - S(N,0,\beta)\norm{\tau_m f}_{L^s}^2\right) \\
			={}& \norm{f}_{D_\alpha}^2 + \left(C_{\mathrm{BE}}(N,0,\beta)^{-1}-1\right)R(n,m,\alpha)\norm{\tau_m f}_{D_{\alpha-m/2}}^2\\
			& - C_{\mathrm{BE}}(N,0,\beta)^{-1}\,S(n,m,\alpha)\norm{\tau_mf}_{L^s}^2 \\
			\le{}& C_{\mathrm{BE}}(N,0,\beta)^{-1}\left(\norm{f}_{D_\alpha}^2 - S(n,m,\alpha)\norm{\tau_mf}_{L^s}^2\right). \qedhere
		\end{align*}
  The proof is complete.
\hfill{$\Box$}
\vskip0.2in
\begin{remark}\label{bbb1}
    Here, we illustrate the sharpness of the estimate \eqref{fractional-trace stability}. We set $N=n-m$ and $\beta=\alpha-\frac m2$ as defined above. As discussed below the statement of Theorem \ref{main thm1}, it suffices to construct a sequence $\{f_i\}_i\subset D_\alpha(\R^n)$ such that \eqref{bbb2} holds. First, thanks to the sharpness of \eqref{fractional stability} (see \cite{chen-frank}), there exists a sequence $\{g_i\}_i\subset D_{\beta}(\R^{N})$ such that such that $g_i\notin \mathcal{M}_{N,0,\beta}$, $\inf\limits_{v_0\in \mathcal{M}_{N,0,\beta}}\norm{g_i-v_0}_{D_{\beta}(\R^{N})}\rightarrow 0$, and for any $\gamma<2$,
\begin{align}\label{bbb3}
    \frac{\norm{g_i}_{D_\beta(\R^N)}^\gamma - S(N,0,\beta)^{\frac{\gamma}{2}}\norm{g_i}_{L^{\frac{2N}{N-2\beta}}(\R^{N})}^\gamma}{\inf\limits_{v_0\in \mathcal{M}_{N,0,\beta}}\norm{g_i-v_0}^\gamma_{D_\beta(\R^N)}}\rightarrow 0.
\end{align}
Next, we can take $f_i\in D_\alpha(\R^n)$ such that $\tau_m f_i=g_i$ and $f_i$ attains equality in \eqref{inequality of thm2} (see our proof of Theorem \ref{main thm1} above). Now we have $$R(n,m,\alpha)^\frac{1}{2} \norm{g_i}_{D_{\beta}(\R^{N})}= \norm{f_i}_{D_\alpha(\R^n)}$$ and $$R(n,m,\alpha)^\frac{1}{2}{\inf\limits_{v_0\in \mathcal{M}_{N,0,\beta}}} \norm{g_i-v_0}_{D_{\beta}(\R^{N})}= {\inf\limits_{v\in \mathcal{M}_{n,m,\alpha}}}\norm{f_i-v}_{D_\alpha(\R^n)}.$$
Note that we have $S(n,m,\alpha)=R(n,m,\alpha) S(N,0,\beta)$, it follows that $\{f_i\}_i$ satisfies \eqref{bbb2} as a straightforward conclusion from \eqref{bbb3}.

\end{remark}

\vskip0.1in
Next we show Theorem \ref{main+++} by further exploiting the proof of Theorem \ref{main thm1}.
\vskip0.2in

\noindent{\bf Proof of Theorem \ref{main+++}.}  We first discuss the best constant $C_{\mathrm{BE}}(n,m,\alpha)$ of \eqref{fractional-trace stability}. Define
		\begin{equation}\label{eeff}
			S_{\mathrm{Tr}}(f)\defeq\frac{\norm{f}_{D_\alpha}^2 - S(n,m,\alpha)\norm{\tau_mf}_{L^s}^2}{\dis^2(f,\mathcal{M}_{n,m,\alpha})},\quad \text{for }f\in D_\alpha(\R^n)\backslash\mathcal{M}_{n,m,\alpha},
		\end{equation}
		\begin{equation*}
			\mathcal{E}(f) \defeq \frac{\norm{f}_{D_{\beta}}^2-S(N,0,\beta)\norm{\tau_m f}_{L^s}^2}{\dis^2(f,\mathcal{M}_{N,0,\beta})},\quad \text{for }f\in D_{\beta}(\R^{n-m})\backslash\mathcal{M}_{N,0,\beta},
		\end{equation*}
		with $N=n-m,\,\beta=\alpha-m/2,\,s=\frac{2(n-m)}{n-2\alpha}$. The above proof shows that for any $f\in D_\alpha(\R^n)$,
		\begin{equation*}
			\begin{split}
				\dis^2(f,\mathcal{M}_{n,m,\alpha}) ={}& \inf_{h\in\mathcal{M}_{n,m,\alpha}} \norm{f-g}_{D_\alpha}^2 + \norm{g-h}_{D_\alpha}^2\\
				={}& \norm{f}_{D_\alpha}^2 - R(n,m,\alpha)\norm{\tau_mf}_{D_{\alpha-m/2}}^2 + \inf_{h\in\mathcal{M}_{n,m,\alpha}}\norm{g-h}_{D_\alpha}^2\\
				={}& \norm{f}_{D_\alpha}^2 - R(n,m,\alpha)\norm{\tau_mf}_{D_{\alpha-m/2}}^2 + R(n,m,\alpha)\inf_{\tilde h\in\mathcal{M}_{N,0,\beta}}\norm{\tau_mf - \tilde{h}}_{D_{\alpha-m/2}}^2\\
				\defeq {}& \delta_1(f) + R(n,m,\alpha)\dis^2(\tau_m f,\mathcal{M}_{N,0,\beta})
			\end{split}
		\end{equation*}
		with $g$ defined as in the proof of Theorem \ref{main thm1}. Hence, if $\tau_mf\notin \mathcal{M}_{N,0,\beta}$, by
		\begin{equation*}
			\begin{aligned}
				S_{\mathrm{Tr}}(g)&=\frac{R(n,m,\alpha)\norm{\tau_mf}_{D_{\alpha-m/2}}^2 - S(n,m,\alpha) \norm{\tau_m f}_{L^s}^2}{R(n,m,\alpha) \dis^2(\tau_m f,\mathcal{M}_{N,0,\beta})}\\
				&= \frac{\norm{\tau_mf}_{D_{\alpha-m/2}}^2 - S(N,0,\beta)\norm{\tau_m f}_{L^s}^2}{\dis^2(\tau_m f,\mathcal{M}_{N,0,\beta})} = \mathcal{E}(\tau_mf) \le 1,
			\end{aligned}
		\end{equation*}
		we get
		\begin{equation}\label{compare}
			S_{\mathrm{Tr}}(f) = \frac{\delta_1(f) + R(n,m,\alpha)\left(\norm{\tau_mf}_{D_{\alpha-m/2}}^2 - S(N,0,\beta)\norm{\tau_mf}_{L^s}^2\right)}{\delta_1(f) + R(n,m,\alpha)\dis^2(\tau_m f,\mathcal{M}_{N,0,\beta})} \ge \mathcal{E}(\tau_mf),
		\end{equation}
		and equality holds if and only if $f=g$. If $\tau_mf\in\mathcal{M}_{N,0,\beta}$, then $\dis^2(f,\mathcal{M}_{n,m,\alpha}) = \delta_1(f)$, $\norm{\tau_mf}_{D_{\alpha-m/2}}^2 = S(N,0,\beta)\norm{\tau_mf}_{L^s}^2$, and so
		\begin{equation*}
			S_{\mathrm{Tr}}(f) = \frac{\norm{f}_{D_\alpha}^2 - R(n,m,\alpha)S(N,0,\beta)\norm{\tau_mf}_{L^s}^2}{\delta_1(f)} =  1.
		\end{equation*}
		As a result,
		\begin{equation*}
			C_{\text{BE}}(n,m,\alpha)=\inf_{f\in D_{\alpha}(\R^n)\backslash \mathcal{M}_{n,m,\alpha}} S_{\mathrm{Tr}}(f) = \inf_{g\in D_{\alpha-m/2}(\R^{n-m})\backslash\mathcal{M}_{N,0,\beta}}\mathcal{E}(g)=C_{\mathrm{BE}}(N,0,\beta),
		\end{equation*}
		and the best constant of Theorem \ref{main thm1} is exactly that of Theorem \ref{thm: stability1}.
		
		Now based on the works in \cite{dol-fig,konig1, konig3,Chenlu}, we have the bound:
		\begin{equation*}
			\frac{\min\{K_{n-m,\alpha-m/2},1,2^{\frac{n+2\alpha-2m}{n-m}}-2\}}{4}\leq C_{\mathrm{BE}}(n,m,\alpha),
		\end{equation*}
		and
		\begin{equation*}
			C_{\mathrm{BE}}(n,m,\alpha)
			\begin{cases}
				<\min\left\{\frac{4\alpha-2m}{n+2\alpha+2-2m},2-2^{\frac{n-2\alpha}{n-m}}\right\},&n-m\geq 2,\\
				\leq \frac{4\alpha-2m}{2\alpha+3-m},&n-m=1,
			\end{cases}
		\end{equation*}
		where $K_{n,s}$ is a complicated number. We refer to \cite[Theorem 1.2]{Chenlu} for an explicit expression of $K_{n,s}$. The proof of $(1)$ and $(2)$ is complete.

\vskip0.23in In the following we show $(3)$. We consider the existence of minimizers for the stability constant $C_{\mathrm{BE}}(n,m,\alpha)$ in the spirit of K\"{o}nig's work in \cite{konig3}.   Let $(u_n)$ be a minimizing sequence for \eqref{eeff} with $\norm{u_n}_{D_\alpha(\R^n)}=1$.  We still set $N=n-m$ and $\beta=\alpha-m/2$. From the proof of $(1)$ and $(2)$, we may assume each $S_{\mathrm{Tr}}(u_n)<1$ to avoid the case $\tau_m u_n\in\mathcal{M}_{N,0,\beta}$. Denote by $\tilde{u}_n$ the minimizer of \eqref{inequality of thm2} with $\tau_m \tilde{u}_n = \tau_m u_n\in D_{\beta}(\R^{n-m})\backslash\mathcal{M}_{N,0,\beta}$, then $\mathcal{E}(\tau_m u_n) = S_{\mathrm{Tr}}(\tilde{u}_n)\le S_{\mathrm{Tr}}(u_n)$ and so $\lim\limits_{n\rightarrow\infty} \mathcal{E}(\tau_m u_n) = \lim\limits_{n\rightarrow\infty} S_{\mathrm{Tr}}(u_n) = C_{\mathrm{BE}}(n,m,\alpha)$, and $\left(\frac{\tau_m u_n}{\norm{\tau_m u_n}_{D_{\beta}}}\right)_n$ is a minimizing sequence for $\mathcal{E}$ with norm 1. Using \cite[Theorem 1.2]{konig3}, there exists a function $v\in D_{\beta}(\R^{n-m})\backslash \mathcal{M}_{N,0,\beta}$ such that, up to a subsequence, $\frac{\tau_m u_n}{\norm{\tau_m u_n}_{D_{\beta}}}$ converge to $v$ strongly in $D_{\beta}(\R^{n-m})$, and in addition $\mathcal{E}(v) = C_{\mathrm{BE}}(n,m,\alpha)$.

\vskip0.16in		
		We now claim $\lim\limits_{n\rightarrow \infty} \norm{\tau_m u_n}_{D_{\beta}}^2=R(n,m,\alpha)^{-1}$. By \eqref{compare},
		\begin{align*}
			\left(1-S_{\mathrm{Tr}}(u_n)\right)\delta_1(u_n) ={}& R(n,m,\alpha) S_{\mathrm{Tr}}(u_n)\dis^2(\tau_m u_n,\mathcal{M}_{N,0,\beta})\\
			&-R(n,m,\alpha)\left(\norm{\tau_m u_n}_{D_{\beta}}^2 - S(N,0,\beta)\norm{\tau_m u_n}_{L^s}^2\right)\\
			={}& R(n,m,\alpha)\dis^2(\tau_m u_n,\mathcal{M}_{N,0,\beta})\big(S_{\mathrm{Tr}}(u_n)-\mathcal{E}(\tau_mu_n)\big).
		\end{align*}
		Since $\lim\limits_{n\to \infty} S_{\mathrm{Tr}}(u_n)=\lim\limits_{n\to \infty} \mathcal{E}(\tau_m u_n) = C_{\mathrm{BE}}<1$, the right side tends to 0 as $n\to \infty$. On the other hand, the left side is larger than $\frac{1-C_{\mathrm{BE}}(n,m,\alpha)}{2}>0$ if $n$ is large enough, and so
		\begin{equation*}
			\begin{aligned}
				\lim\limits_{n\to \infty} \delta_1(u_n) = 0 \implies \lim\limits_{n\to \infty} \norm{\tau_m u_n}_{D_{\beta}}^2 =& \lim\limits_{n\to \infty} R(n,m,\alpha)^{-1}\big(\norm{u_n}_{D_\alpha}^2 - \delta_1(u_n)\big)\\
				=&\;R(n,m,\alpha)^{-1}.
			\end{aligned}
		\end{equation*}
		Hence, up to a subsequence, $\tau_m u_n\to R(n,m,\alpha)^{-1} v$ strongly in $D_{\beta}(\R^{n-m})$. As a result, if we set $u\in D_\alpha(\R^n)$ to be the minimizer of \eqref{inequality of thm2} with $\tau_m u=R(n,m,\alpha)^{-1}v$, then
		\begin{align*}
			\norm{u_n-u}_{D_\alpha}^2 ={}& \norm{u_n-\tilde{u}_n}_{D_\alpha}^2 + \norm{\tilde{u}_n-\tilde{v}}_{D_{\alpha}}^2\\
			={}& \delta_1(u_n) + R(n,m,\alpha)\norm{\tau_m u_n - R(n,m,\alpha)^{-1} v}_{D_{\beta}}^2 \to 0,\quad n\to \infty,
		\end{align*}
		yielding that $u_n\to u$ strongly in $D_\alpha$ up to a subsequence. The only thing left is to show $S_{\mathrm{Tr}}(u)=C_{\mathrm{BE}}(n,m,\alpha)$. By the definition of $u$, we have
		\begin{equation*}
			S_{\mathrm{Tr}}(u) = \mathcal{E}(\tau_m u) = \mathcal{E} (v) = C_{\mathrm{BE}}(n,m,\alpha). \qedhere
		\end{equation*}
  The proof of $(3)$ is complete.
\hfill$\Box$


\section{Some consequences of the gradient stability result}\label{sec4}
	
In this section we give some corollaries of the gradient stability result \eqref{fractional-trace stability}.
We first prove the following theorem concerning refined Sobolev embedding in domains with finite measure.
	\begin{theorem}\label{cons1}
		Let $0\leq m<n, \frac{m}{2}<\alpha<\frac{n}{2}$ and $s=\frac{2(n-m)}{n-2\alpha},q_1=\frac{n-m}{n-2\alpha},q_2=\frac{n-m}{n-\alpha-\frac{m}{2}}$. Then there exists a constant $C$ depending only on $n, \alpha$ such that, for any domain $\Omega\subset\mathbb{R}^n$ with $A:=\mathcal{H}^{n-m}(\Omega\cap\mathbb{R}^{n-m})<\infty$, and any $u\in D_\alpha(\R^n)$ with $\tau_m{u}=0$ in $\Omega^c$, we have
		\begin{equation}\label{www1}
			\Vert u\Vert_{D_\alpha(\R^n)}^2-S(n,m,\alpha)\Vert \tau_m{u}\Vert_{L^s(\R^{n-m})}^2\geq CA^{-\frac{1}{q_1}}\Vert \tau_m{u}\Vert_{L_w^{q_1}(\Omega\cap\mathbb{R}^{n-m})}^2.
		\end{equation}
		Here $\mathcal{H}$ denotes the Hausdorff measure. If $\alpha-\frac{m}{2}\in\mathbb{N}$, then we also have
		\begin{equation}\label{www2}
			\Vert u\Vert_{D_\alpha(\R^n)}^2-S(n,m,\alpha)\Vert \tau_m{u}\Vert_{L^s(\R^{n-m})}^2\geq CA^{-\frac{1}{q_1}}\Vert \Delta^{\frac{2\alpha-m}{4}}(\tau_m{u})\Vert_{L_w^{q_2}(\Omega\cap\mathbb{R}^{n-m})}^2.
		\end{equation}
		Moreover, the weak norm on the right-hand side cannot be replaced by the strong norm.
	\end{theorem}
	
	The proof of \eqref{www1} and \eqref{www2} can be reduced to the case $m=0$ using Theorem \ref{thm: ineq2}. Based on Theorem \ref{main thm1}, it suffices to show the following lemma:
	
	\begin{lemma}\label{cons2}
		Let $0<\alpha<\frac{n}{2},q_1=\frac{n}{n-2\alpha},q_2=\frac{n}{n-\alpha}$. Then there exists a constant $C$ depending only on $\alpha,n$ such that
		\begin{equation}\label{ww1}
			\left\Vert u\right\Vert_{L^{q_1}_{w}(\Omega)} \leq C|\Omega|^{\frac{1}{2q_1}}d(u,\mathcal{M}_{n,0,\alpha})
		\end{equation}
		for all subdomains $\Omega\in\mathbb{R}^{n}$ with $|\Omega|<\infty$ and all $u\in D_{\alpha}(\R^n)$ with $u=0$ in $\Omega^c$.
		
		If in addition $\alpha\in\mathbb{N}$, then the following estimate holds:
		\begin{equation}\label{ww2}
			\left\Vert \Delta^{\frac{\alpha}{2}}u\right\Vert_{L^{q_2}_{w}(\Omega)} \leq C|\Omega|^{\frac{1}{2q_1}}d(u,\mathcal{M}_{n,0,\alpha}).
		\end{equation}
	\end{lemma}

\begin{remark}
		The inequality \eqref{ww1} has been proved by Chen, Frank and Weth in \cite{chen-frank}. When $\alpha=1$, \eqref{ww2} was proved by Br\'ezis and Lieb in \cite{Brezis-Lieb}. Although the proof of \eqref{ww2} for general $\alpha$ are similar to that applied in \cite{chen-frank}, for completeness, we give a detailed proof below.
	\end{remark}
	
\vskip0.3in
\noindent{\bf Proof of Lemma \ref{cons2}.} Assume \eqref{ww2} fails and $|\Omega|=1$. Then there exists a sequence $(u_n)\subset H_0^1(\Omega)\backslash\{0\}$ such that
		\begin{equation}\label{con1}
			\begin{aligned}
				\frac{d(u_n,\mathcal{M}_{n,0,\alpha})}{\left\Vert \Delta^{\frac{\alpha}{2}}u_n\right\Vert_{L^{q_2}_{w}(\Omega)}}\rightarrow 0.
			\end{aligned}
		\end{equation}
		By homogeneity, we can assume $\left\Vert \Delta^{\frac{\alpha}{2}}u_n\right\Vert_{L^2}=1$. By the H\"{o}lder inequality, $\left\Vert \Delta^{\frac{\alpha}{2}}u_n\right\Vert_{L^{q_2}_{w}(\Omega)}$ is bounded and hence $d(u_n,\mathcal{M}_{n,0,\alpha})\to 0$. This shows that there exist $c_n\rightarrow 1, z_n\in\Omega$ and $\lambda_n\rightarrow\infty$ such that
		\begin{equation*}
			\begin{aligned}
				d(u_n,\mathcal{M}_{n,0,\alpha})=\left\Vert \Delta^{\frac{\alpha}{2}}(u_n-c_nU[z_n,\lambda_n])\right\Vert_{L^2},
			\end{aligned}
		\end{equation*}
		with $U[z,\lambda] = \left(\frac{\lambda}{1+\lambda^2\abs{x-z}^2}\right)^{\frac{n-2\alpha}{2}}$. A direct computation yields
		\begin{equation}\label{ww4}
			\begin{aligned}
				d(u_n,\mathcal{M}_{n,0,\alpha})^2
				&\geq \left\Vert \Delta^{\frac{\alpha}{2}}c_nU[0,\lambda_n]\right\Vert_{L^2(\Omega^c)}^2\\
				&\gtrsim \left\Vert \Delta^{\frac{\alpha}{2}}U[0,\lambda_n]\right\Vert_{L^2(B_2(0)^c)}^2\\
				&\gtrsim \int_{B_2(0)^c}\frac{\lambda_n^{2\alpha-n}}{|x|^{2(n-\alpha)}}\gtrsim \lambda_n^{2\alpha-n}.
			\end{aligned}
		\end{equation}
		Therefore we have
		\begin{equation*}
			\begin{aligned}
				\left\Vert \Delta^{\frac{\alpha}{2}}u_n\right\Vert_{L^{q_2}_{w}(\Omega)}
				&\leq \left\Vert \Delta^{\frac{\alpha}{2}}(u_n-c_nU[z_n,\lambda_n])\right\Vert_{L^{q_2}_{w}(\Omega)}+\left\Vert \Delta^{\frac{\alpha}{2}}c_nU[z_n,\lambda_n]\right\Vert_{L^{q_2}_{w}(\Omega)}\\
				&\lesssim \left\Vert \Delta^{\frac{\alpha}{2}}(u_n-c_nU[z_n,\lambda_n])\right\Vert_{L^2}+\lambda_n^{\frac{2\alpha-n}{2}}\left\Vert \Delta^{\frac{\alpha}{2}}U[0,1]\right\Vert_{L^{q_2}_{w}(\R^n)}\\
				&\lesssim d(u_n,\mathcal{M}_{n,0,\alpha}).
			\end{aligned}
		\end{equation*}
		This contradicts \eqref{con1}. \hfill$\Box$

\vskip0.36in

	By further exploiting the proof above, we can establish similar results for strips.
	\begin{theorem}\label{cons3}
		Assume $m,n\in\mathbb{N}^{+}$ satisfy $\frac{n}{4}\leq m<\frac{n}{2}$. Let $q=\frac{n}{n-2m}$. Then there exists a constant $C$ depending only on $m,n$ such that
		\begin{equation}\label{ww3}
			\Vert u\Vert_{D_m(\R^n)}^2-S(n,0,m)\Vert u\Vert_{L^{2q}(\R^n)}^2\geq C\left\Vert u\right\Vert_{L^{q}_{w}(I\times \R^{n-1})}^2
		\end{equation}
		for all $u\in D_{m}(\R^n)$ whose support belongs to $I\times\R^{n-1}$; here $I=[-1,1]$ is an interval in $\R$.
	\end{theorem}
\vskip0.136in	
	\begin{remark}
		When $m=1$, \eqref{ww3} was proved by Wang and Willem in \cite{wang}. They in fact considered the refined Caffarelli-Kohn-Nirenberg inequality in more general domains. Their methods relied on a maximum principle. In our cases, due to the absence of maximum principle for the polyharmonic operator, we restrict our analysis to domains that are bounded in one direction.
	\end{remark}

\vskip0.35in
\noindent{\bf Proof of Theorem \ref{cons3}.} Assume that \eqref{ww3} fails, then there exists a sequence $(u_n)\subset H_0^1(\Omega)$ such that
		\begin{equation*}
			\begin{aligned}
				\frac{d(u_n,\mathcal{M}_{n,0,m})}{\left\Vert u_n\right\Vert_{L^{q}_{w}(I\times \R^{n-1})}}\rightarrow 0.
			\end{aligned}
		\end{equation*}
		Assume $ \Vert u_n\Vert_{D_m}=1$. Note that
		\begin{equation*}
			\int_{I\times\R^{n-1}}|\Delta^\frac{m}{2}u_n|^2\geq \int_{\R^{n-1}}\int_{I}\left|\frac{\partial^m u_n}{\partial x_1^m}\right|^2\gtrsim \int_{\R^{n-1}}\int_{I}|u_n|^2.
		\end{equation*}
		The assumption $\frac{n}{4}\leq m<\frac{n}{2}$, the interpolation inequality and the fractional Sobolev inequality yield
		\begin{equation}\label{ww5}
			\left\Vert u_n\right\Vert_{L^{q}_{w}(I\times \R^{n-1})}\leq \left\Vert u_n\right\Vert_{L^2}^{\lambda}\left\Vert u_n\right\Vert_{L^{2q}}^{1-\lambda}\lesssim \Vert u_n\Vert_{D_m}
		\end{equation}
		for some $0\leq \lambda\leq 1$. Hence $\left\Vert u_n\right\Vert_{L^{q}_{w}(I\times \R^{n-1})}$ is bounded and $d(u,\mathcal{M}_{n,0,m})\rightarrow 0$. This indicates the existence of $c_n\rightarrow 1, z_n\in I\times\R^{n-1}$ and $\lambda_n\rightarrow\infty$ such that
		\begin{equation*}
			\begin{aligned}
				d(u_n,\mathcal{M}_{n,0,m})=\left\Vert \Delta^{\frac{m}{2}}(u_n-c_nU[z_n,\lambda_n])\right\Vert_{L^2}.
			\end{aligned}
		\end{equation*}
		Similarly to the computation in \eqref{ww4}, we derive that $d(u_n,\mathcal{M}_{n,0,m})\gtrsim \lambda_n^{\frac{2m-n}{2}}$. Now fix a cut-off function $\eta\in C_c^{\infty}(\R)$ such that $0\leq\eta\leq 1$, $\eta=1$ in $2I$, $\eta=0$ outside $3I$. Using \eqref{ww5}, we can estimate:
		\begin{equation*}
			\begin{aligned}
				\left\Vert u_n\right\Vert_{L^{q}_{w}(I\times\R^{n-1})}&\leq \left\Vert u_n-c_n\eta U[z_n,\lambda_n]\right\Vert_{L^{q}_{w}(I\times \R^{n-1})}+\left\Vert c_n\eta U[z_n,\lambda_n]\right\Vert_{L^{q}_{w}(I\times \R^{n-1})}\\
				&\lesssim \left\Vert u_n-c_n\eta U[z_n,\lambda_n]\right\Vert_{D_m}+\left\Vert c_nU[z_n,\lambda_n]\right\Vert_{L^{q}_{w}(I\times \R^{n-1})}\\
				&\lesssim \left\Vert u_n-c_nU[z_n,\lambda_n]\right\Vert_{D_m}+\left\Vert c_n(1-\eta) U[z_n,\lambda_n]\right\Vert_{D_m}+\lambda_n^{\frac{2m-n}{2}}\\
				&\lesssim d(u_n,\mathcal{M}_{n,0,m})+\left\Vert (1-\eta)U[z_n,\lambda_n]\right\Vert_{D_m}.
			\end{aligned}
		\end{equation*}
		It will lead to a contradiction if there exists a uniform constant $C$ such that
		\begin{equation}\label{ww6}
			\left\Vert (1-\eta)U[z_n,\lambda_n]\right\Vert_{D_m}\leq C\lambda_n^{\frac{2m-n}{2}},\quad \forall n\in \N.
		\end{equation}
		Based on our choice of $z_n,\lambda_n$ and $\eta$, \eqref{ww6} can be derived from the following three simple observations:
		\begin{align*}
			& \left\Vert (1-\eta)U[z_n,\lambda_n]\right\Vert_{D^m}\lesssim \sum\limits_{i=0}^{m-1}\left\Vert \nabla^{i}U[z_n,\lambda_n]\right\Vert_{L^2((3I-2I)\times\R^{n-1})}+\left\Vert \nabla^m U[z_n,\lambda_n]\right\Vert_{L^2((2I)^c\times\R^{n-1})}, \\
			& \left\Vert |\nabla^{i}U[z_n,\lambda_n]|\right\Vert_{L^2((3I-2I)\times\R^{n-1})}\lesssim \lambda_n^{\frac{2m-n}{2}},\quad 0\leq i\leq m-1,\\
			& \left\Vert \nabla^m U[z_n,\lambda_n]\right\Vert_{L^2((2I)^c\times\R^{n-1})}\lesssim \lambda_n^{\frac{2m-n}{2}}.
		\end{align*}
        The proof is complete.
\hfill$\Box$

\vskip0.36in

	Next we establish dual stability for the Escobar trace inequality:
	\begin{equation*}
		S_{\mathrm{E}}(n)\norm{f}_{L^{\frac{2(n-1)}{n-2}}(\R^{n-1})}^2 \le \norm{f}_{H^1(\R_+^n)}^2.
	\end{equation*}
	The extremal functions are given by \eqref{esco} up to normalizations. Consider the pairing $(X,Y,\scalprod{\cdot}{\cdot} )$:
	\begin{equation*}
		X=H^1(\R_+^n),\;Y=L^{\frac{2(n-1)}{n}}(\R^{n-1}),\;\scalprod{f}{g}=\int_{\R^{n-1}}fg,
	\end{equation*}
	and two functionals $\Phi$ and $\Psi$ on $X$:
	\begin{equation*}
		\Phi(f)=\norm{f}_{H^1(\R_+^n)}^2,\quad \Psi(f)=S_{\mathrm{E}}(n)\norm{f}_{L^{\frac{2(n-1)}{n-2}}(\R^{n-1})}^2.
	\end{equation*}
	It is straightforward to compute the Legendre transform $\Phi^*$ and $\Psi^*$ on $Y$:
	\begin{align*}
		\Phi^*(g)&=\sup_{f\in X}\left\{\scalprod{f}{g}-\Phi(f)\right\}\\
		&=\sup_{f\in X}\left\{\int_{\R^{n-1}}fg-\int_{\R_+^n}|\nabla f|^2\right\}\\
		&=\frac{1}{4}\Vert \nabla \mathcal{P}[g]\Vert_{L^2(\mathbb{R}_{+}^n)}^2\\
		\Psi^*(g)&=\sup_{f\in X}\left\{\scalprod{f}{g}-\Phi(f)\right\}\\
		&=\sup_{f\in X}\left\{\int_{\R^{n-1}}fg-S_{\mathrm{E}}(n)\norm{f}_{L^{\frac{2(n-1)}{n-2}}(\R^{n-1})}^2\right\}\\
		&=\frac{1}{4S_{\mathrm{E}}(n)}\norm{g}_{L^{\frac{2(n-1)}{n}}(\R^{n-1})}^2.
	\end{align*}
	The function $\mathcal{P}[g]$ in the above computations is the unique solution of
	\begin{equation}\label{neumann}
		\begin{cases}
			\Delta \mathcal{P}[g]= 0 & \text{in }\R_+^n, \\
			\frac{\partial \mathcal{P}[g]}{\partial t}=-g & \text{on }\partial\R_+^n.
		\end{cases}
	\end{equation}
	Based on the standard dual scheme developed by Carlen in \cite{carlen1}, we can obtain the following stability result:
	
	\begin{theorem}\label{cons4}
		For any $n\geq 3$ and $u\in L^{\frac{2n-2}{n}}(\mathbb{R}^{n-1})$, the following inequality holds:
		\begin{align}\label{ooo1}
			\Vert u\Vert_{L^{\frac{2n-2}{n}}(\mathbb{R}^{n-1})}^2\geq S_{\mathrm{E}}(n)^{-1}\Vert \nabla \mathcal{P}[u]\Vert_{L^2(\mathbb{R}_{+}^n)}^2,
		\end{align}
		where $\mathcal{P}[u]$ is the unique solution of \eqref{neumann}. The equality holds precisely when $u(x)=(1+|x|^2)^{\frac{-n}{2}}$, after suitable translation, scaling and normalization.
		
		Furthermore, let $\mathcal{M}_{\mathrm{Neu}}$ denote the extremal manifold. The following sharp stability result holds:
		\begin{equation*}
			\Vert u\Vert_{L^{\frac{2n-2}{n}}(\mathbb{R}^{n-1})}^2- S_{\mathrm{E}}(n)^{-1}\Vert \nabla \mathcal{P}[u]\Vert_{L^2(\mathbb{R}_{+}^n)}^2\geq C\inf_{v\in \mathcal{M}_{\mathrm{Neu}}}\Vert u-v\Vert_{L^{\frac{2n-2}{n}}(\mathbb{R}^{n-1})}^2
		\end{equation*}
		for a positive constant $C$ that can be explicitly expressed in terms of $C_{\mathrm{BE}}(n,1,1)$.
	\end{theorem}
    \vskip0.2in
\noindent{\bf Proof of Theorem \ref{cons4}.}
We follow the proof of \cite[Theorem 1.19]{carlen1}. From \cite[Theorem 2.12]{carlen1}, the operator $\Psi^*$ is $\frac{n-2}{n}$-convex. For any $g\in Y$, we can apply \cite[Lemma 3.1]{carlen1} to derive that
\begin{align}\label{mmm1}
    \Psi^*(g)-\Phi^*(g)\geq \Phi(f)-\Psi(f)+\frac{n-2}{4S_{\mathrm{E}}(n)n}\norm*{g-\nabla\Psi(f)}_{L^{\frac{2(n-1)}{n}}(\R^{n-1})}^2,
\end{align}
where $f=\nabla\Phi^*(g)$. For simplicity, we will denote $C_{\mathrm{BE}}(n,1,1)$ as $C_{\mathrm{BE}}$ in what follows. From Theorem \ref{main thm1}, we have
\begin{align}\label{mmm2}
    \Phi(f)-\Psi(f)\geq&\; C_{\mathrm{BE}}\inf_{h\in\mathcal{M}_{n,1,1}}\norm{f-h}_{H^1(\R_+^n)}^2\nonumber\\
    \geq&\;C_{\mathrm{BE}}\, S_{\mathrm{E}}(n)\inf_{h\in\mathcal{M}_{n,1,1}}\norm{f-h}_{L^{\frac{2(n-1)}{n-2}}(\R^{n-1})}^2.
\end{align}
Assume that the infimum in \eqref{mmm2} is attained by $h_0$. Combining \eqref{mmm1} and \eqref{mmm2} yields
\begin{align}\label{mmm3}
    \Psi^*(g)-\Phi^*(g)\geq C_{\mathrm{BE}}\,S_{\mathrm{E}}(n)\norm{f-h_0}_{L^{\frac{2(n-1)}{n-2}}(\R^{n-1})}^2+\frac{n-2}{4S_{\mathrm{E}}(n)n}\norm*{g-\nabla\Psi(f)}_{L^{\frac{2(n-1)}{n}}(\R^{n-1})}^2.
\end{align}
Thanks to \cite[Theorem 2.13]{carlen1}, $\nabla \Psi$ is Lipschitz from $L^{\frac{2(n-1)}{n-2}}(\R^{n-1})$ to $L^{\frac{2(n-1)}{n}}(\R^{n-1})$ with Lipschitz constant $\frac{n}{n-2}$. Therefore, we have
\begin{align}\label{mmm4}
    \norm{f-h_0}_{L^{\frac{2(n-1)}{n-2}}(\R^{n-1})}\geq\frac{n-2}{n} \norm{\nabla\Psi(f)-\nabla\Psi(h_0)}_{L^{\frac{2(n-1)}{n}}(\R^{n-1})}.
\end{align}
Note that $\nabla\Psi(h_0)\in\mathcal{M}_{\mathrm{Neu}}$.
From \eqref{mmm3} and \eqref{mmm4}, we obtain
\begin{align}
    \Psi^*(g)-\Phi^*(g)\geq&\; C\norm{\nabla\Psi(f)-\nabla\Psi(h_0)}_{L^{\frac{2(n-1)}{n}}(\R^{n-1})}^2+C\norm*{g-\nabla\Psi(f)}_{L^{\frac{2(n-1)}{n}}(\R^{n-1})}^2\nonumber\\
    \geq&\;C\norm{\nabla\Psi(h_0)-g}_{L^{\frac{2(n-1)}{n}}(\R^{n-1})}^2\nonumber\\
    \geq &\;C\inf_{h\in\mathcal{M}_{\mathrm{Neu}}}\norm{h-g}_{L^{\frac{2(n-1)}{n}}(\R^{n-1})}^2.\nonumber
\end{align}
Here, $C$ is a positive constant that can be explicitly expressed in terms of $C_{\mathrm{BE}}$.
\hfill$\Box$

\vskip0.36in

	\begin{remark}
		Through further calculations, it can be shown that the a priori estimate \eqref{ooo1} is equivalent to the Hardy-Littlewood-Sobolev inequality in $\R^{n-1}$, a fact that is implicitly contained in \cite{beck}. Specifically, we express $\mathcal{P}[u]$ explicitly in terms of $u$ using the Green's representation formula:
        \begin{align}
            \mathcal{P}[u](x',t)=C_n\int_{\R^{n-1}}\frac{u(y')}{\left(|x'-y'|^2+t^2\right)^\frac{n-2}{2}}dy',\nonumber
        \end{align}
        where $x'\in\R^{n-1},t\in\R$, and $C_n$ is a dimensional constant. Substituting this representation into \eqref{ooo1}, we obtain:
        \begin{align}
           \Vert u\Vert_{L^{\frac{2n-2}{n}}(\mathbb{R}^{n-1})}^2\geq&\; S_{\mathrm{E}}(n)^{-1}\Vert \nabla \mathcal{P}[u]\Vert_{L^2(\mathbb{R}_{+}^n)}^2=S_{\mathrm{E}}(n)^{-1}\int_{\R^{n-1}}\mathcal{P}[u]u\nonumber\\ 
           =&\;S_{\mathrm{E}}(n)^{-1}C_n\int_{\R^{n-1}\times\R^{n-1}}\frac{u(x')u(y')}{|x'-y'|^{n-2}}dx'dy'.\nonumber
        \end{align}
        Moreover, the extremal functions of \eqref{ooo1} coincide with the extremal functions of the Hardy-Littlewood-Sobolev inequality in $\R^{n-1}$. Consequently, $S_{\mathrm{E}}(n)^{-1} C_n$ must equal the sharp constant of the Hardy-Littlewood-Sobolev inequality, thereby establishing the equivalence.
        
        We present Theorem \ref{cons4} here for two primary reasons. First, we believe that this estimate is of independent interest. Second, the operator $\mathcal{P}$ and the inequality \eqref{ooo1} will play a crucial role in Section \ref{sec7}, where they are applied to address the dual Sobolev norm.
	\end{remark}
	
	\section{Qualitative profile decomposition of the   Neumann   problem}\label{sec5}
	
	In this section we establish qualitative profile decomposition for the following nonlinear critical Neumann boundary problem in the same spirit of pioneer works of Struwe in \cite{struwe1984}:
	\begin{equation}\label{eqt}
		\begin{cases}
			\Delta u= 0 & \text{in }\R_+^n, \\
			\frac{\partial u}{\partial t}=-|u|^{p-1}u & \text{on }\partial\R_+^n.
		\end{cases}
	\end{equation}
	Although we deal with the trace version here, our main methods are parallel to those in \cite{Struwe2}, where \eqref{eleqt} is treated. We first consider the global compactness for general functions.

\vskip0.3in

	\begin{theorem}\label{ql}
		Let $n\geq 3$ and $(u_k)_{k\in \mathbb{N}}\subset H^1(\mathbb{R}_{+}^n)$ be a sequence of functions with $\int_{\mathbb{R}_{+}^{n}}|\nabla u_k|^2$ uniformly bounded. Assume that (recall the dual norm is defined in \eqref{dualnorm})
		\begin{equation}\label{wc}
			\left\Vert \Delta u_k +|u_k|^{p-1}u_k\right\Vert_{H^{-1}}\rightarrow 0\quad \text{as}\; k\rightarrow \infty.
		\end{equation}
		Then there exist a number $\nu\in\mathbb{N}$, a subsequence of $(u_k)_k$, which we still denote by $(u_k)_k$, a sequence $(z_1^{(k)},...,z_{\nu}^{(k)})$ of $\nu$-tuples of points in $\mathbb{R}^{n-1}$, a sequence $(\lambda_1^{(k)},...,\lambda_{\nu}^{k})$ of $\nu$-tuples of positive real numbers and a sequence of functions $(w_i)_{1\leq i\leq \nu}$, which are nontrivial solutions of \eqref{eqt}, such that
		\begin{equation}\label{rr1}
			\left\Vert \nabla\left(u_k-\sum\limits_{i=1}^{\nu}w_i[z_i^{(k)},\lambda_i^{(k)}]\right)\right\Vert_{L^2}\rightarrow 0\quad \text{as}\; k\rightarrow \infty,
		\end{equation}
		\begin{equation}\label{rr2}
			\left\Vert\nabla u_k\right\Vert_{L^2}^2-\sum\limits_{i=1}^{\nu}\left\Vert\nabla w_i[z_i^{(k)},\lambda_i^{(k)}]\right\Vert_{L^2}^2\rightarrow 0\quad \text{as}\; k\rightarrow \infty.
		\end{equation}
		Here $w_i[z,\lambda](x,t):=\lambda^{\frac{n-2}{2}}w_i(\lambda(x-z),\lambda t)$ for every $z\in\mathbb{R}^{n-1}, \lambda>0$.
	\end{theorem}

\vskip0.3in
\noindent{\bf Proof of Theorem \ref{ql}.} In the following proof, we will use $B_r(x,t)$ to denote the ball centered at $(x,t)$ with radius $r$ and omit the subscript $r$ when $r=1$. We also define $B_r^+(x,t)=B_r(x,t) \cap \R_+^n$, $\partial_bB_r^+(x,t)=\partial B_r^+(x,t)\cap\R^{n-1}$.
		
		Since $\int_{\mathbb{R}_{+}^{n}}|\nabla u_k|^2$ is uniformly bounded, we may assume $\int_{\mathbb{R}_{+}^{n}}|\nabla u_k|^2\rightarrow L\in [0,\infty)$. If $L=0$, then $u_k\rightarrow 0$ and the results hold clearly. In the following we assume $L>0$ and $\frac{L}{2}\leq\int_{\mathbb{R}_{+}^{n}}|\nabla u_k|^2\leq\frac{3L}{2}$ for any $k$. Due to translational and dilational invariances, we can take a sequence $(z_k)_{k}$ of points in $\mathbb{R}^{n-1}$, a sequence $(\lambda_k)_{k}$ of positive real numbers and a sequence $(t_k)_{k}$ of nonnegative numbers such that
		\begin{equation}\label{wc2}
			\int_{B^+(0,t_k)}\left|\nabla u_k\left[z_k,\lambda_k\right]\right|^2=\sup_{(z,t)\in\mathbb{R}^{n}_{+}}\int_{B^+(z,t)}\left|\nabla u_k\left[z_k,\lambda_k\right]\right|^2=\epsilon
		\end{equation}
		for any $k$ and some $\epsilon$ sufficiently small that will be determined later.
		
		If $t_k\rightarrow\infty$ as $k\to \infty$, we may assume $t_k\nearrow \infty$ and set $v_k(x,t)=u_k\left[z_k,\lambda_k\right](x,t+t_k)$ and $\Omega_k=\mathbb{R}^{n-1}\times (-t_k,\infty)$. Then $v_k\subset H^1(\Omega_k)$ and $\Omega_k\rightarrow\mathbb{R}^n$. Since $u_k\left[z_k,\lambda_k\right]$ is uniformly bounded, there exist a subsequence and a function $v\in H^1(\mathbb{R}^n)$ such that
		\begin{equation*}
			v_k\rightharpoonup v\quad \text{in}\; H^1(\Omega_m)\quad \text{for any fixed }m,
		\end{equation*}
		and
		\begin{equation*}
			v_k\rightarrow v\quad \text{in}\; L^2_{\mathrm{loc}}(\mathbb{R}^n).
		\end{equation*}
		From \eqref{wc} we also know that
		\begin{equation}\label{wc1}
			\left\Vert \Delta v_k +|v_k|^{p-1}v_k\right\Vert_{H^{-1}(\Omega_k)}\rightarrow 0\quad \text{as}\; k\rightarrow \infty.
		\end{equation}
		Hence $v$ is a global harmonic function in $\mathbb{R}^n$. But $v\in H^1(\mathbb{R}^n)$ and so $v$ must be a zero function. Now take a cut-off function $\eta\in C^{\infty}_c(\mathbb{R}^n)$ such that $0\leq\eta\leq 1$, $\eta=1$ in $B(0,0)$ and $\eta=0$ outside $B_2(0,0)$. Using \eqref{wc2} and \eqref{wc1} and H\"{o}lder inequality we can obtain
		\begin{equation*}
			\begin{aligned}
				o(1)&=\left\langle v_k\eta, \Delta v_k +|v_k|^{p-1}v_k\right\rangle\\
				&=\int_{B_2(0,0)} \nabla(v_k\eta)\cdot \nabla v_k\\
				&=\int_{B_2(0,0)}\eta |\nabla v_k|^2+\int_{B_2(0,0)}v_k\nabla\eta\cdot \nabla v_k\\
				&\geq \epsilon+o(1),
			\end{aligned}
		\end{equation*}
		where $o(1)\rightarrow 0$ as $k\rightarrow\infty$. The computation above yields a contradiction!
		
		So $t_k$ must be uniformly bounded. Without loss of generality, we assume $t_k\rightarrow \Bar{t}$ for some $\Bar{t}\in [0,\infty)$ and $|t_k-\Bar{t}|\leq \frac{1}{8}$. Choose cut-off functions $\eta_i\in C^{\infty}_c(\mathbb{R}^n)$ such that $0\leq\eta_i\leq 1$, $i=1,2$, $\eta_1=1$ in $B_2(0,\Bar{t})$, $\eta_1=0$ outside $B_3(0,\Bar{t})$, while $\eta_2=1$ in $B_3(0,\Bar{t})$, $\eta_2=0$ outside $B_4(0,\Bar{t})$. Set $v_k(x,t)=u_k\left[z_k,\lambda_k\right](x,t)$. There exists a function $v\in H^1(\mathbb{R}^{n}_{+})$ such that
		\begin{equation*}
			v_k\rightharpoonup v\quad \text{in}\; H^1(\mathbb{R}^{n}_{+}),
		\end{equation*}
		\begin{equation*}
			v_k\rightarrow v\quad \text{in}\; L^2_{\mathrm{loc}}(\mathbb{R}^n_{+}),
		\end{equation*}
		and
		\begin{equation*}
			v_k\rightarrow v\quad \text{in}\; L^q_{\mathrm{loc}}(\partial\mathbb{R}^{n}_+)\; \text{for any }1\leq q<p+1.
		\end{equation*}
		From \eqref{wc} we know that
		\begin{equation}\label{wc3}
			\left\Vert \Delta v_k +|v_k|^{p-1}v_k\right\Vert_{H^{-1}}\rightarrow 0\quad \text{as}\; k\rightarrow \infty.
		\end{equation}
		Using $(v_k-v)\eta_1^2$ as a test function, we then obtain
		\begin{align}
			o(1)&=\left\langle (v_k-v)\eta_1^2, \Delta v_k +|v_k|^{p-1}v_k\right\rangle \label{aabb} \\
			&=\int_{B^{+}_3(0,\Bar{t})} \nabla\left((v_k-v)\eta_1^2\right)\cdot \nabla v_k-\int_{\partial_b B^{+}_3(0,\Bar{t})}v_k|v_k|^{p-1}(v_k-v)\eta_1^2 \notag \\
			&=\int_{B^{+}_3(0,\Bar{t})}\left|\nabla ((v_k-v)\eta_1)\right|^2+o(1)-\int_{\partial_b B^{+}_3(0,\Bar{t})}|v_k-v|^{p+1}\eta_1^2+o(1) \notag \\
			&=\int_{B^{+}_3(0,\Bar{t})}\left|\nabla ((v_k-v)\eta_1)\right|^2-\int_{\partial_b B^{+}_3(0,\Bar{t})}((v_k-v)\eta_1)^2|(v_k-v)\eta_2|^{p-1}+o(1) \notag \\
			&\geq \int_{B^{+}_3(0,\Bar{t})}\left|\nabla ((v_k-v)\eta_1)\right|^2+o(1) \notag \\
			&\quad -\left(\int_{\partial_b B^{+}_3(0,\Bar{t})}|(v_k-v)\eta_1|^{p+1}\right)^{\frac{2}{p+1}}\left(\int_{\partial_b B^{+}_4(0,\Bar{t})}|(v_k-v)\eta_2|^{p+1}\right)^{\frac{p-1}{p+1}} \notag \\
			&\geq \int_{B^{+}_3(0,\Bar{t})}\left|\nabla ((v_k-v)\eta_1)\right|^2\left(1-S_{\mathrm{E}}(n)^{-\frac{p+1}{2}}\left(\int_{B^{+}_4(0,\Bar{t})}\left|\nabla ((v_k-v)\eta_2)\right|^2\right)^{p-1}\right)+o(1),\notag
		\end{align}
		where $o(1)\rightarrow 0$ as $k\rightarrow\infty$. Note that
		\begin{equation}\label{aabb1}
			\begin{aligned}
				\int_{B^{+}_4(0,\Bar{t})}\left|\nabla ((v_k-v)\eta_2)\right|^2&=\int_{B^{+}_4(0,\Bar{t})}\left|\nabla (v_k-v)\right|^2\eta_2^2+o(1)\\
				&=\int_{B^{+}_4(0,\Bar{t})}(|\nabla v_k|^2-|\nabla v|^2)\eta_2^2+o(1)\\
				&\leq \int_{B^{+}_4(0,\Bar{t})}|\nabla v_k|^2+o(1)\\
				&\leq C_0\epsilon+o(1).
			\end{aligned}
		\end{equation}
		Here $C_0$ is a number such that $B_4(0,0)$ can be covered by $C_0$ half unit balls. If we choose $\epsilon<\min\left\{\frac{L}{2},\frac{S_{\mathrm{E}}(n)^{\frac{p+1}{2(p-1)}}}{C_0}\right\}$, then \eqref{aabb} and \eqref{aabb1} indicate that
		\begin{equation*}
			\int_{B^{+}_3(0,\Bar{t})}\left|\nabla ((v_k-v)\eta_1)\right|^2\rightarrow 0\quad as\; k\rightarrow\infty.
		\end{equation*}
		Hence $v_k\rightarrow v$ strongly in $H^1(B_2^+(0,\Bar{t}))$. Since
		\begin{equation*}
			\int_{B^{+}_2(0,\Bar{t})}|\nabla v_k|^2\geq \int_{B^{+}(0,t_k)}|\nabla v_k|^2=\epsilon,
		\end{equation*}
		we know that $v\not\equiv 0$.
		
		To conclude the proof, first note that from \eqref{wc3} and the weak convergence, $v$ is indeed a nontrivial solution of \eqref{eqt}. Moreover we have
		\begin{equation}\label{re1}
			\left\Vert \nabla v_k\right\Vert_{L^2}^2=\left\Vert \nabla v\right\Vert_{L^2}^2+\left\Vert \nabla(v_k-v)\right\Vert_{L^2}^2+o(1)
		\end{equation}
		and
		\begin{equation}\label{re2}
			\begin{aligned}
				&\left\langle \phi, \Delta (v_k-v) +|v_k-v|^{p-1}(v_k-v)\right\rangle \\
				={}& \left\langle \phi, \Delta v_k +|v_k|^{p-1}v_k\right\rangle-\left\langle \phi, \Delta v +|v|^{p-1}v\right\rangle+o(1)\\
				={}& o(1)+0+o(1)=o(1)
			\end{aligned}
		\end{equation}
		for any $\phi\in H^1(\mathbb{R}^{n}_+)$ with $\Vert\nabla\phi\Vert_{L^2}\leq 1$. Here $o(1)$ is independent of the choice of $\phi$ and $o(1)\rightarrow 0$ as $k\rightarrow\infty$. Now we can choose a sequence $(z_1^{(k)})_{k}$ of points in $\mathbb{R}^{n-1}$ and a sequence $(\lambda_1^{(k)})_k$ of positive numbers such that $u_k-v[z_1^{(k)},\lambda_1^{(k)}]$ is a dilation of $v_k-v$. Define $w_1=v$, then the relations \eqref{re1} and \eqref{re2} reduce to
		\begin{equation*}
			\left\Vert \nabla u_k\right\Vert_{L^2}^2=\left\Vert \nabla w_1\right\Vert_{L^2}^2+\left\Vert \nabla\left(u_k-w_1\left[z_1^{(k)},\lambda_1^{(k)}\right]\right)\right\Vert_{L^2}^2+o(1)
		\end{equation*}
		and
		\begin{equation*}
			\left\Vert \Delta \left(u_k-w_1\left[z_1^{(k)},\lambda_1^{(k)}\right]\right) +\left|u_k-w_1\left[z_1^{(k)},\lambda_1^{(k)}\right]\right|^{p-1}\left(u_k-w_1\left[z_1^{(k)},\lambda_1^{(k)}\right]\right)\right\Vert_{H^{-1}}\rightarrow 0
		\end{equation*}
		as $k\rightarrow 0$. If $\left\Vert \nabla\left(u_k-w_1\left[z_1^{(k)},\lambda_1^{(k)}\right]\right)\right\Vert_{L^2}\rightarrow 0$, then the proof is end.

		If not, viewing $u_k-w_1\left[z_1^{(k)},\lambda_1^{(k)}\right]$ as new $u_k$ and using the above arguments repeatedly, we may find a sequence of nontrivial solutions $(w_i)_i$, some suitable dilations $(\lambda_i^{(k)})_{k,i}$ and translations $(z_i^{(k)})_{k,i}$ such that
		\begin{equation*}
			\left\Vert\nabla u_k\right\Vert_{L^2}^2=\sum\limits_{i=1}\left\Vert\nabla w_i[z_i^{(k)},\lambda_i^{(k)}]\right\Vert_{L^2}^2+\left\Vert \nabla\left(u_k-\sum\limits_{i=1}w_i\left[z_1^{(k)},\lambda_1^{(k)}\right]\right)\right\Vert_{L^2}^2+o(1).
		\end{equation*}
		Note that, for any nontrivial solution $v$ of \eqref{eqt}, by the Sobolev inequality we have
		\begin{equation*}
			\Vert\nabla v\Vert_{L^2}^2=\Vert v\Vert_{L^{2^\dagger}}^{2^\dagger}\leq S_{\mathrm{E}}(n)^{-\frac{2^\dagger}{2}}\Vert\nabla v\Vert_{L^2}^{2^\dagger}.
		\end{equation*}
		Hence we have the energy estimate
		\begin{equation*}
			\Vert\nabla v\Vert_{L^2}\geq S_{\mathrm{E}}(n)^{\frac{2^\dagger}{2(p-1)}}.
		\end{equation*}
		Since $u_k$ are uniformly bounded in $H^1$, there must exists a finite number $\nu$ such that
		$$\left\Vert \nabla\left(u_k-\sum\limits_{i=1}^{\nu}w_i\left[z_i^{(k)},\lambda_i^{(k)}\right]\right)\right\Vert_{L^2}=o(1),$$
		and the proof is completed.   \hfill$\Box$

\vskip0.36in

	To transition from the general case to the nonnegative case, we need the following Br\'ezis-Lieb type lemma, which has already been used by Mercuri and Willem in \cite{carlo} to obtain similar results for the p-Laplacian.
	\begin{lemma}\label{re3}
		Let $n\geq 3$ and $(u_k)_{k\in\mathbb{N}}\subset H^1(\mathbb{R}_{+}^n)$, assume that
		\begin{enumerate}
			\item[$(a)$] $\sup\limits_{k\in\N}\left\Vert\nabla u_k\right\Vert_{L^2}\leq \infty$,
			\item[$(b)$] $u_k\rightharpoonup u$ in $H^1(\mathbb{R}_{+}^n)$,
			\item[$(c)$] $\norm{(u_k)_{-}}_{L^{2^\dagger}(\mathbb{R}^{n-1})}\rightarrow 0$ as $k\rightarrow \infty$.
		\end{enumerate}
		Then $u\geq 0$ in $\R^{n-1}$ and
		\begin{equation*}
			\left\Vert (u_k-u)_{-}\right\Vert_{L^{2^\dagger}(\mathbb{R}^{n-1})}\rightarrow 0\quad as\; k\rightarrow \infty.
		\end{equation*}
	\end{lemma}
	
	\vskip0.3in
\noindent{\bf Proof of Lemma \ref{re3}.}
		By the Sobolev compact embedding inequalities, we know that $u_k\rightarrow u$ a.e. in $\mathbb{R}^{n-1}$. After passing to a subsequence, we may assume that there exists a function $g\in L^{2^\dagger}$ such that
		\begin{equation*}
			|(u_k)_{-}|\leq g
		\end{equation*}
		for any $k$. Then we have
		\begin{equation*}
			|(u_k-u)_{-}|\leq u^+ +g.
		\end{equation*}
		By Lebesgue dominated convergence theorem, we immediately get that
		$ \left\Vert (u_k-u)_{-}\right\Vert_{L^{2^\dagger}}$ goes to 0 and $u\geq 0$ on $\mathbb{R}^{n-1}$. Since the subsequence is arbitrary, the proof is complete.
	\hfill$\Box$
	
	\begin{theorem}\label{ql1}
		Let $n\geq 3$ and $\nu\geq 1$ be positive numbers. Let $(u_k)_{k\in \mathbb{N}}\subset H^1(\mathbb{R}_{+}^n)$ be a sequence of functions such that $(\nu-\frac{1}{2})S_{\mathrm{E}}(n)^{n-1}\leq \int_{\mathbb{R}_{+}^{n}}|\nabla u_k|^2\leq (\nu+\frac{1}{2})S_{\mathrm{E}}(n)^{n-1}$, and assume that
		\begin{equation*}
			\left\Vert \Delta u_k +u_k^p\right\Vert_{H^{-1}}\rightarrow 0\quad as\; k\rightarrow \infty,
		\end{equation*}
		\begin{equation}\label{aqaqaq}
			\left\Vert (u_k)_{-}\right\Vert_{L^{2^\dagger}(\mathbb{R}^{n-1})}\rightarrow 0\quad as\; k\rightarrow \infty.
		\end{equation}
		Then there exist a subsequence of $(u_k)_k$, which we still denote by $(u_k)_k$, a sequence $(z_1^{(k)},...,z_{\nu}^{(k)})$ of $\nu$-tuples of points in $\mathbb{R}^{n-1}$ and a sequence $(\lambda_1^{(k)},...,\lambda_{\nu}^{k})$ of $\nu$-tuples of positive real numbers such that
		\begin{equation}\label{rer1}
			\left\Vert \nabla\left(u_k-\sum\limits_{i=1}^{\nu}U[z_i^{(k)},\lambda_i^{(k)}]\right)\right\Vert_{L^2}\rightarrow 0\quad as\; k\rightarrow \infty,
		\end{equation}
		\begin{equation}\label{rer2}
			\left\Vert\nabla u_k\right\Vert_{L^2}^2-\sum\limits_{i=1}^{\nu}\left\Vert\nabla U[z_i^{(k)},\lambda_i^{(k)}]\right\Vert_{L^2}^2\rightarrow 0\quad as\; k\rightarrow \infty,
		\end{equation}
		and
		\begin{equation}\label{rer3}
			\min\left(\frac{\lambda_i}{\lambda_j},\frac{\lambda_j}{\lambda_i},\frac{1}{\lambda_i\lambda_j |z_i-z_j|^2}\right)\rightarrow 0\quad as\; k\rightarrow \infty.
		\end{equation}
	\end{theorem}
	
	\begin{remark}
		We do not assume that $u_k$ is nonnegative in the statement, because nonnegativity is not a property preserved under the repeating argument in the proof of Theorem \ref{ql}. In other words, we cannot ensure that $u_k-w_1\left[z_1^{(k)},\lambda_1^{(k)}\right]$ is nonnegative!
	\end{remark}
	
\vskip0.36in

\noindent{\bf Proof of Theorem \ref{ql1}.}  From the proof of Theorem $\ref{ql}$, there exist a number $\nu_0\in\mathbb{N}$, a sequence $(z_1^{(k)},...,z_{\nu_0}^{(k)})$ of $\nu_0$-tuples of points in $\mathbb{R}^{n-1}$, a sequence $(\lambda_1^{(k)},...,\lambda_{\nu_0}^{k})$ of $\nu_0$-tuples of positive numbers and a sequence $(w_i)_{1\leq i\leq \nu_0}$ of nontrivial solutions to \eqref{eqt} satisfying \eqref{rr1} and \eqref{rr2}. From the construction of the functions $(w_i)_{1\leq i\leq \nu_0}$, Lemma \ref{re3}, and the assumption \eqref{aqaqaq}, we deduce that each $w_i\geq 0$ in $\R^{n-1}$ for all $1\leq i\leq\nu_0$. Since each $w_i$ is harmonic in $\R^n_+$, it follows that $w_i\geq 0$ in $\R^n_+$ as well. This can be readily shown by choosing $(w_i)_-$ as a test function. From the classification results, we may assume $w_i=U$ for any $1\leq i\leq\nu_0$. Note that $\Vert\nabla U\Vert_{L^2}^2=S_{\mathrm{E}}(n)^{n-1}$. It follows from \eqref{rr2} that $\nu_0=\nu$. Therefore, \eqref{rer1} and \eqref{rer2} hold. Finally, \eqref{rer3} follows directly from \eqref{rer1}, \eqref{rer2} and \eqref{estimate of bubbles}.
\hfill$\Box$

\vskip0.236in

	\begin{remark}\label{energy gap}
		In the above theorem, the condition \eqref{aqaqaq} can be replaced by
		\begin{equation*}
			\limsup_{k\rightarrow\infty}\int_{\R_+^n}|\nabla(u_k)_+|^2>0
		\end{equation*}
		when we restrict to the case $\nu=1$. This argument is a direct consequence of Theorem \ref{ql} and the following energy gap inequality:
		\begin{equation*}
			\norm{W}_{H^1(\R_+^n)}^2\geq 2S_{\mathrm{E}}(n)^{n-1}
		\end{equation*}
		for any sign-changing solution $W$ of \eqref{eqt}. To prove this inequality, we choose two test functions $W_+\coloneqq\max\{W,0\}$ and $W_-\coloneqq\min\{W,0\}$. Using equation \eqref{eqt} and the Escobar trace inequality, we have
		\begin{equation*}
			\norm{W_+}_{H^1(\R_+^n)}^2=\norm{W_+}_{L^{2^\dagger}(\R^{n-1})}^{2^\dagger}\leq S_{\mathrm{E}}(n)^{-\frac{2^\dagger}{2}}\norm{W_+}_{H^1(\R_+^n)}^{2^\dagger},
		\end{equation*}
		\begin{equation*}
			\norm{W_-}_{H^1(\R_+^n)}^2=\norm{W_-}_{L^{2^\dagger}(\R^{n-1})}^{2^\dagger}\leq S_{\mathrm{E}}(n)^{-\frac{2^\dagger}{2}}\norm{W_-}_{H^1(\R_+^n)}^{2^\dagger}.
		\end{equation*}
		Hence $\norm{W_+}_{H^1}^2,\norm{W_-}_{H^1}^2\geq S_{\mathrm{E}}(n)^{n-1}$. Since $\norm{W}_{H^1}^2=\norm{W_+}_{H^1}^2+\norm{W_-}_{H^1}^2$, we get the desired inequality.
	\end{remark}

\vskip0.36in
	
\section{Quantitative profile decomposition  of  the   Neumann   problem }\label{sec6}
	
	In this section, we follow the idea of Figalli and Glaudo in \cite{figalli} to derive a quantitative profile decomposition for \eqref{eqt}, which is the Euler-Lagrange equation related to the Escobar trace inequality:
	
	\begin{equation}\label{trace}
		\norm*{\nabla \varphi}_{L^2(\R^n_+)}^2 - S_{\mathrm{E}}(n) \norm{\varphi}_{L^{\frac{2(n-1)}{n-2}}(\partial\R^n_+)}^2 \ge 0,\quad \forall\, \varphi\in H^1(\R^n_+).
	\end{equation}
	

\vskip0.23in
	
\noindent{\bf  Proof of Theorem \ref{thm:main_close}.}
		Let $\sigma=\sum\limits_{i=1}^\nu \alpha_i U[z_i,\lambda_i]$ be the linear combination of Escobar bubbles that is closest to $u$ in the $H^1$-norm, that is
		\begin{equation*}
			\norm{\nabla u-\nabla\sigma}_{L^2(\R^n_+)}
			= \min_{\substack{
					\tilde\alpha_1,\dots,\tilde\alpha_\nu\in\R
					\\
					\tilde z_1,\dots,\tilde z_\nu\in\R^{n-1}
					\\
					\tilde \lambda_1,\dots,\tilde \lambda_\nu
			}} \norm*{\nabla u-\nabla\left(
				\sum\limits_{i=1}^\nu \tilde\alpha_i U[\tilde z_i, \tilde \lambda_i]
				\right)}_{L^2(\R_+^n)} .
		\end{equation*}
		Let $\rho\defeq u-\sigma$, and denote $U_i\defeq U[z_i,\lambda_i]$. From \eqref{condition_1}, it follows that $\norm{\nabla\rho}_{L^2(\R_+^n)}\le \delta$. Furthermore, the family $(\alpha_i,U_i)_{1\le i\le \nu}$ is $\delta'$-interacting for some $\delta'$ that goes to zero as $\delta$ goes to $0$. We can assume $(\alpha_i)_{1\le i\le \nu}$ are positive.
		
		Since $\sigma$ minimizes the $H^1$-distance from $u$, $\rho$ is $H^1$-orthogonal to the manifold composed of linear combinations of these $\nu$ Escobar bubbles. That is, for any $1\le i\le \nu$, the following orthogonal conditions hold:
		\begin{align}
			&\int_{\R^n_+}\nabla\rho\cdot\nabla \phi = 0 \ \text{for $\phi=U_i,\,\partial_{\lambda}U_i$ and $\partial_{z_j} U_i$, $1\le j\le n-1$.}\label{ortho-1}
		\end{align}
		Using the fact that each $U_i$ satisfies the equation \eqref{eqeqeq}, the above conditions are equivalent to
		\begin{align}
			& \int_{\partial\R_+^n} \rho\cdot U_i^{p-1}\phi = 0 \ \text{for $\phi=U_i,\,\partial_{\lambda}U_i$ and $\partial_{z_j}U_i$, $1\le j\le n-1$.} \label{ortho-1-1}
		\end{align}
		To estimate $\norm{\nabla \rho}_{L^2(\R_+^n)}$, applying orthogonal condition \eqref{ortho-1}, we obtain
		\begin{equation}\label{main est}
			\begin{split}
				\int_{\R_+^n}|\nabla\rho|^2 ={}& \int_{\R_+^n}\nabla \rho \cdot \nabla u -\int_{\partial\R_+^n} \rho \abs{u}^{p-1}u + \int_{\partial\R_+^n} \rho \abs{u}^{p-1}u\\
				\le{}& \norm{\lapl u+\abs{u}^{p-1}u}_{H^{-1}}\norm{\nabla\rho}_{L^2(\R^n_+)} + \int_{\partial\R_+^n}\rho \abs{u}^{p-1}u.
			\end{split}
		\end{equation}
		To control the second term, we use the elementary estimates
    		\begin{equation}\label{estimate:minus}
            \begin{aligned}
    			\abs*{\abs{a+b}^{p-1}(a+b)-\abs{a}^{p-1}a-p\abs{a}^{p-1}b} \le {}& \begin{cases}
    				C_n\big( \abs{a}^{p-2} \abs{b}^2 + \abs{b}^p \big) & n=3, \\
    				C_n\abs{b}^p & n\ge 4,
    			\end{cases}
    			\\
    			\abs*{\left(\sum\limits_{i=1}^{\nu}a_i\right)\left|\sum\limits_{i=1}^{\nu}a_i\right|^{p-1}
    				-\sum\limits_{i=1}^\nu \abs{a_i}^{p-1}a_i}
    			\lesssim {}&
    			\sum\limits_{1\le i\not=j \le \nu} \abs{a_i}^{p-1}\abs{a_j},
    		\end{aligned}
        \end{equation}
		that hold for any $a,b\in\R$ and for any $a_1,\dots,a_\nu\in\R$. Then
		\begin{align*}
			\abs*{\abs{u}^{p-1}u - \sum\limits_{i=1}^\nu \alpha_i^p U_i^p} \le{}& \abs*{(\rho+\sigma)\abs{\rho+\sigma}^{p-1} - \sigma\abs{\sigma}^{p-1}} + \abs*{\sigma\abs{\sigma}^{p-1} - \sum\limits_{i=1}^\nu \alpha_i^p U_i^p} \\
			\le{}& p\abs{\sigma}^{p-1}\abs{\rho} + C_{n,\nu}\left(\abs{\sigma}^{p-2}\rho^2 + \abs{\rho}^p + \sum\limits_{1\le i\neq j\le \nu} U_i^{p-1} U_j\right).
		\end{align*}
		Combining \eqref{ortho-1-1} we obtain
		\begin{align*}
			\int_{\partial\R_+^n} \abs{u}^{p-1}u\rho \le{}& p \int_{\partial\R_+^n} \sigma^{p-1}\rho^2 + C_{n,\nu}\bigg(\chi_{\{n=3\}}\int_{\partial\R_+^n} \abs{\sigma}^{p-2}\abs{\rho}^3 + \abs{\rho}^{p+1}\bigg)\\& + C_{n,\nu}\sum\limits_{1\le i\neq j\le \nu} \int_{\partial \R_+^n} \abs{\rho} U_i^{p-1} U_j ,
		\end{align*}
		where $\chi_{\{n=3\}}$ means that terms only appear when $n=3$. By H\"{o}lder inequality and the Escobar trace inequality \eqref{trace},
		\begin{align*}
			& \chi_{\{n=3\}}\int_{\partial\R_+^n} \abs{\rho}^3\abs{\sigma}^{p-2} \le \norm{\rho}_{L^{p+1}(\partial \R_+^n)}^3\cdot \norm{\sigma}_{L^{\frac{p+1}{p-2}}(\partial\R_+^n)}\lesssim \norm{\nabla\rho}_{L^2(\R_+^n)}^3,\\
			& \int_{\partial\R_+^n}\rho^{p+1} = \norm{\rho}_{L^{p+1}(\partial\R_+^n)}^{p+1} \lesssim \|\nabla\rho\|_{L^2(\R_+^n)}^{p+1},\\
			& \int_{\partial\R_+^n} \abs{\rho}U_i^{p-1}U_j \le
			\norm{\rho}_{L^{p+1}(\partial\R_+^n)}\norm{U_i^{p-1}U_j}_{L^{\frac{p+1}{p}}(\partial\R_+^n)} \lesssim \norm{\nabla\rho}_{L^2(\R_+^n)}\norm{U_i^{p-1}U_j}_{L^{\frac{p+1}{p}}(\partial \R_+^n)}.
		\end{align*}
		Thanks to Lemma \ref{estimate of bubbles}, if $n=3$, then for any $i\neq j$ it holds
		\begin{equation*}
			\|U_i^{p-1}U_j\|_{L^{\frac{p+1}{p}}(\partial\R_+^n)} \approx  \mu_{ij}^{1/2} \approx \int_{\partial \R_+^n} U_i^{p}U_j.
		\end{equation*}
		Hence,
		\begin{equation}\label{last term}
			\begin{aligned}
				\int_{\partial\R_+^n} \abs{u}^{p-1}u\rho \le{}& p \int_{\partial\R_+^n} \sigma^{p-1}\rho^2 + C_{n,\nu}\left( \norm{\nabla\rho}_{L^2(\R_+^n)}^3 + \norm{\nabla\rho}_{L^2(\R_+^n)}^{p+1}\right)\\
				&+C_{n,\nu}
				\sum\limits_{1\le i\neq j\le \nu}\norm{\nabla\rho}_{L^2(\R_+^n)}\int_{\partial\R_+^n}U_i^pU_j.
			\end{aligned}
		\end{equation}
		We state here the following two inequalities which are important to our analysis. We postpone their proofs to Proposition \ref{sppp} and Proposition \ref{prop:bubbles}.
		\begin{itemize}
			\item For $n\ge 3$, if $\delta'$ is small enough, then
			\begin{align}\label{spectrum estimate}
				\int_{\partial\R_+^n}\sigma^{p-1}\rho^2 \le \frac{c(n,\nu)}{p}\norm{\nabla \rho}_{L^2(\R_+^n)}^2
			\end{align}
			for some $0<c(n,\nu)<1$.
			\item For $n\ge 3$ and a given $\tilde{\epsilon}>0$, we have
			\begin{align}\label{bubble estimate}
				\int_{\partial\R_+^n}U_i^pU_j \le \tilde{\epsilon}\|\nabla \rho\|_{L^2(\R_+^n)} + C\norm{\lapl u+\abs{u}^{p-1}u}_{H^{-1}} + \norm{\nabla\rho}_{L^2(\R_+^n)}^{\min(2,p)} .
			\end{align}
		\end{itemize}
		With these estimates, we can choose $\tilde{\epsilon}$ so that $c(n,\nu) + \nu^2\tilde{\epsilon}C_{n,\nu} < 1$. Combining \eqref{spectrum estimate}, \eqref{bubble estimate} into \eqref{last term} and then into \eqref{main est}, we get
		\begin{align*}
			\big(1-c(n,\nu)-\nu^2 \tilde{\epsilon} c_{n,\nu}\big) \norm{\nabla\rho}_{L^2(\R_+^n)} \lesssim \norm{\lapl u+\abs{u}^{p-1}u}_{H^{-1}} + \norm{\nabla\rho}_{L^2(\R_+^n)}^2 + \norm{\nabla\rho}_{L^2(\R_+^n)}^3
		\end{align*}
		for $n=3$. Since $\norm{\nabla\rho}_{L^2(\R_+^n)} \le \delta$, by choosing $\delta$ small enough we obtain the desired estimate
		\begin{equation}\label{conclusion}
			\norm{\nabla\rho}_{L^2(\R_+^n)} \lesssim \norm{\lapl u+\abs{u}^{p-1}u}_{H^{-1}}
		\end{equation}
		for $n=3$. The only thing left is to check that all $\alpha_i$ can be replaced by $1$ and \eqref{eq:interaction_estimate_statement} holds as well. Thanks to the estimate \eqref{conclusion} and the proof of \eqref{bubble estimate} as in the next subsection, both facts are true.  \hfill$\Box$

\vskip0.36in

	\begin{remark}
		If $n>3$, the proof fails because
		\begin{equation*}
			\norm{U_i^{p-1}U_j}_{L^{\frac{p+1}{p}}(\partial\R_+^n)} \approx \begin{cases}
				\ln(\mu_{ij})^{\frac{n}{2(n-1)}}\mu_{ij}^{\frac{n}4} & n=4 \\
				\mu_{ij} & n\ge 5
			\end{cases} \gg \int_{\partial\R_+^n} U_i^p U_j \approx \mu_{ij}^{\frac{n-2}{2}}.
		\end{equation*}
		But if $\nu=1$, the approximation of $\int_{\partial\R_+^n} \abs{u}^{p-1}u\rho$ will not contain the crossing term $\int_{\partial\R_+^n} \abs{\rho}U_i^{p-1} U_j$. Since \eqref{spectrum estimate} holds for all $n\ge 3$, we can still get the desired stability result:
		
		For $n\ge 3$, there exist a small constant $\delta=\delta(n)>0$ and a large constant $C=C(n)>0$ such that the following statement holds. Let $u\in H^1(\R_+^n)$ be a function such that
		\begin{equation*}
			\norm{\nabla u - \nabla \tilde{U}}_{L^2(\R_+^n)} \le \delta,
		\end{equation*}
		where $\tilde{U}$ is a Escobar bubble, then there exists another Escobar bubble $U$ such that
		\begin{equation*}
			\norm{\nabla u-\nabla U}_{L^2(\R_+^n)} \le C \norm{\lapl u + \abs{u}^{p-1}u}_{H^{-1}},
		\end{equation*}
		where $p=\frac{n}{n-2}$.
	\end{remark}
    \vskip0.2in
	In the following we are devoted to \eqref{spectrum estimate} and \eqref{bubble estimate}. The first estimate \eqref{spectrum estimate} follows directly from the spectral properties of $\rho$. We need the following two lemmas:
	\begin{lemma}\label{lemma:spectral}
		For $n\ge 3$, $x_0\in\R^{n-1}$, $\lambda_0>0$, there exists $\Lambda_3>\Lambda_2=\frac{n}{n-2}$ such that
		\begin{equation*}
			\Lambda_3 \le \frac{\int_{\R_+^n}|\nabla W|^2\,dxdt}{\int_{\partial\R_+^n}U[x_0,\lambda_0]^{\frac{2}{n-2}}W^2\,dx}
		\end{equation*}
		for all $W$ which is orthogonal to $U[x_0,\lambda_0]$, $\partial_\lambda U[x_0,\lambda_0]$ and $\partial_{z_j}U[x_0,\lambda_0]$ ($1\le j\le n-1$).
	\end{lemma}
	For the proof of this lemma, see Ho \cite[Lemma 3.1]{Ho}.
	\begin{remark}
		A refined spectral gap inequality is given in Theorem \ref{spect0} in Section \ref{sec7}.
	\end{remark}
	\begin{lemma}\label{lemma:localization}
		For any $n\ge 3$, $\nu\in\N$ and $\epsilon>0$, there exists $\delta=\delta(n,\nu,\epsilon)>0$ such that if $\{U_i=U[z_i,\lambda_i]\}_{i=1}^\nu$ is a $\delta$-interacting family of $\nu$ Escobar bubbles, then there exist $\nu$ Lipschitz function $\Phi_i:\R^n\to [0,1]$ satisfying
		\begin{enumerate}
			\item Almost all mass of $U_i^{2^\dagger}$ on $\partial\R_+^n$ is in the region $\{\Phi_i=1\}$, that is
			\begin{equation*}
				\int_{\{\Phi_i=1\}\cap \partial\R_+^n} U_i^{2^\dagger} \ge (1-\epsilon)S_{\mathrm{E}}(n)^{n-1}
			\end{equation*}
			\item In the region $\{\Phi_i>0\}\cap \partial\R_+^n$, we have $\epsilon U_i > U_j$ for any $j\neq i$.
			\item The $L^n$-norm of $\Phi_i$ is small, that is
			\begin{equation*}
				\norm{\nabla \Phi_i}_{L^n(\R_+^n)} \le \epsilon.
			\end{equation*}
			\item For any $j\neq i$ such that $\lambda_j\le \lambda_i$, we have
			\begin{equation*}
				\frac{\sup_{\{\Phi_i>0\}\cap \partial\R_+^n} U_j}{\inf_{\{\Phi_i>0\}\cap \partial\R_+^n}U_j}<1+\epsilon.
			\end{equation*}
		\end{enumerate}
	\end{lemma}
	The key point is to construct the following cut-off function:
	\begin{equation*}
		\varphi=\varphi_{x_0,r,R}:\R^n\to [0,1],\quad \varphi=\begin{cases}
			1 & |x-x_0|\le r,\\
			\frac{\ln R-\ln |x-x_0|}{\ln R-\ln r}& r<|x-x_0|<R,\\
			0 & R<|x-x_0|.
		\end{cases}
	\end{equation*}
	For a fixed $1\le i \le \nu$, we may assume $U_i=U[0,1]$ and take $\Phi_i$ as the following form:
	\begin{equation*}
		\Phi_i\defeq \varphi_{0,\epsilon R,R} \prod_{j\in J}(1-\varphi_{(z_j,0),R_j,\epsilon^{-1}R_j}),
	\end{equation*}
	where $J=\{1\le j\le \nu : \lambda_j > 1 \text{ and } |z_j|<2R \}$. Then, we can choose suitable $R,\epsilon,R_j$ to ensure that $\Phi_i$ satisfies all conditions. Note that $\varphi_{(z,0),r,R}(x',0)$ is exactly the cut-off function in $\R^{n-1}=\partial \R_+^n$ with a similar expression, so the computations are similar to those of \cite[Lemma 3.9]{figalli} and we omit the details.
	
	With these two lemmas, we can now prove \eqref{spectrum estimate}:
	\begin{proposition}\label{sppp}
		Let $n\ge 3$ and $\nu\in \N$. There exists a constant $\delta=\delta(n,\nu)>0$ such that if $\sigma=\sum\limits_{i=1}^\nu\alpha_i U[z_i,\lambda_i]$ is a linear combination of Escobar bubbles and $\rho\in H^1(\R_+^n)$ satisfies \eqref{ortho-1-1} with $U_i=U[z_i,\lambda_i]$, then
		\begin{equation*}
			\int_{\partial\R_+^n} \sigma^{p-1}\rho^2 \le \frac{c}{p}\int_{\R_+^n}|\nabla\rho|^2,
		\end{equation*}
		where $p=\frac{n}{n-2}$ and $c=c(n,\nu)<1$.
	\end{proposition}

\vskip0.2in
\noindent{\bf Proof of Proposition \ref{sppp}.} Let $\Phi_1,\dots,\Phi_\nu$ be the cut-off functions built in Lemma \ref{lemma:localization} for some $\epsilon=\epsilon(\delta)$. Thanks to Lemma \ref{lemma:localization}-(2), it holds
		\begin{align*}
			\int_{\partial\R_+^n}\sigma^{p-1}\rho^2 \le{}& \int_{\{\sum\limits \Phi^2_i \ge 1\}\cap \partial\R_+^n} \left(\sum\limits_{i=1}^{\nu} \Phi_i^2\right)\sigma^{p-1}\rho^2 + \int_{\{\sum\limits \Phi^2_i < 1\}\cap \partial\R_+^n} \sigma^{p-1}\rho^2 \\
			\le{}& (1+C\epsilon^{p-1})\sum\limits_{i=1}^\nu \int_{\partial \R_+^n} \Phi_i^2\rho^2 U_i^{p-1} + \int_{\{\sum\limits \Phi^2_i < 1\}\cap \partial\R_+^n} \sigma^{p-1}\rho^2.
		\end{align*}
		By Lemma \ref{lemma:localization}-(1), using the H\"{o}lder and Escobar trace inequalities we obtain
		\begin{equation*}
			\int_{\{\sum\limits \Phi^2_i < 1\}\cap \partial\R_+^n} \sigma^{p-1}\rho^2 \le \left(\int_{\{\sum\limits \Phi^2_i < 1\}\cap \partial\R_+^n} \sigma^{p+1}\right)^{\frac{p-1}{p+1}}\norm{\rho}_{L^{p+1}(\partial\R_+^n)}^{2} \le C\epsilon^{\frac 1{n-1}}\norm{\nabla\rho}_{L^2(\R_+^n)}^2.
		\end{equation*}
		Let $\psi:\R_+^n\to \R$ be, up to scaling, one of the functions $U_i,\,\partial_{\lambda}U_i,\, \partial_{z_j}U_i$, with $\int_{\partial\R_+^n} \psi^2U_i^{p-1}=1$. Then by the orthogonal conditions \eqref{ortho-1-1} one gets
		\begin{align*}
			\left|\int_{\partial\R_+^n} (\rho\Phi_i)\psi U_i^{p-1}\right| ={}& \left|\int_{\partial\R_+^n} \rho\psi U_i^{p-1}(1-\Phi_i)\right| \le \left|\int_{\{\Phi_i<1\}\cap \partial\R_+^n} \rho\psi U_i^{p-1}\right| \\
			\le{}& \norm{\rho}_{L^{p+1}}\left(\int_{\partial\R_+^n} \psi^2U_i^{p-1}\right)^{1/2}\left(\int_{\{\Phi_i<1\}\cap \partial\R_+^n} U_i^{p+1} \right)^{\frac{p-1}{p+1}}\\
			\le{}& C\epsilon^{\frac{1}{n-1}}\norm{\nabla\rho}_{L^2(\R_+^n)},
		\end{align*}
		Which means $\rho\Phi_i$ is almost orthogonal to $\psi$. Hence, by Lemma \ref{lemma:spectral},
		\begin{equation*}
			\int_{\partial\R_+^n} (\rho\Phi_i)^2U_i^{p-1} \le \frac{1}{\Lambda_3}\int_{\R_+^n}|\nabla(\rho\Phi_i)|^2 + C\epsilon^{\frac{2}{n-1}}\norm{\nabla\rho}_{L^2(\R_+^n)}^2.
		\end{equation*}
		Note that
		\begin{equation*}
			\int_{\R_+^n}|\nabla(\rho\Phi_i)|^2 = \int_{\R_+^n}|\nabla\rho|^2\Phi_i^2 + \int_{\R_+^n}\rho^2|\nabla\Phi_i|^2 + 2 \int_{\R_+^n}\rho\Phi_i\nabla\rho\cdot\nabla\Phi_i.
		\end{equation*}
		Using the H\"{o}lder and Sobolev inequality (to use Sobolev inequality in the whole space, we need to extend $\rho\in H^1(\R_+^n)$ to $H_0^1(\R^n)$ by $\rho(x,t)\defeq\rho(x,|t|)$), we have
		\begin{align*}
			\int_{\R_+^n}\rho^2|\nabla\Phi_i|^2 \le{}& \norm{\rho}_{L^{2^*}(\R_+^n)}^2\norm{\nabla\Phi_i}_{L^n(\R_+^n)}^2 \lesssim \epsilon^2\norm{\nabla\rho}_{L^2(\R_+^n)}^2,\\
			\int_{\R_+^n}\rho\Phi_i\nabla\rho\cdot\nabla\Phi_i \le{}& \norm{\rho}_{L^{2^*}(\R_+^n)}\norm{\Phi_i}_{L^\infty(\R_+^n)}\norm{\nabla\rho}_{L^2(\R_+^n)}\norm{\nabla\Phi_i}_{L^n(\R_+^n)} \lesssim \epsilon\norm{\nabla\rho}_{L^2(\R_+^n)}^2.
		\end{align*}
		As a consequence of Lemma \ref{lemma:localization}-(2), $\{\Phi_i\}_{i=1}^\nu$ have disjoint supports, and so
		\begin{equation*}
			\sum\limits_{i=1}^\nu \int_{\R_+^n}|\nabla\rho|^2\Phi_i^2 \le \int_{\R_+^n}|\nabla\rho|^2.
		\end{equation*}
		Combining all these estimates, we finally get
		\begin{align*}
			\int_{\partial\R_n^n}\sigma^{p-1}\rho^2 \le \frac{1+o(1)}{\Lambda_3}\int_{\R_+^n}|\nabla\rho|^2 + o(1)\norm{\nabla\rho}_{L^2(\R_+^n)}^2 = \left(\frac{1}{\Lambda_3} + o(1)\right)\norm{\nabla\rho}_{L^2(\R_+^n)}^2,
		\end{align*}
		where $o(1)$ denotes some small quantities going to zero as $\epsilon\to 0$. Since $\Lambda_3>p$, we get the desired result. \hfill$\Box$

\vskip0.36in

	To prove the second estimate \eqref{bubble estimate}, we show the following stronger proposition:
	\begin{proposition}\label{prop:bubbles}
		Let $n\ge 3$ and $\nu\in \N$. For any $\hat{\epsilon}>0$, there exists $\delta=\delta(n,\nu,\hat{\epsilon})>0$ such that the following statement holds. Let $u=\rho + \sum\limits_{i=1}^\nu \alpha_iU_i$, where $(\alpha_i,U_i)_{1\le i\le \nu}$ is $\delta$-interacting, and $\rho$ satisfies the orthogonal conditions \eqref{ortho-1-1} and $\norm{\nabla\rho}_{L^2(\R_+^n)}\le 1$. Then for any $1\le i\le \nu$,
		\begin{equation}\label{estimate of alpha}
			|\alpha_i-1| \lesssim \hat{\epsilon}\norm{\nabla\rho}_{L^2(\R_+^n)} + \norm{\lapl u+\abs{u}^{p-1}u}_{H^{-1}} + \norm{\nabla\rho}_{L^2(\R_+^n)}^{\min(2,p)},
		\end{equation}
		where $p=\frac{n}{n-2}$. And for any pair of indices $i\neq j$,
		\begin{equation}\label{estimate of pairing}
			\int_{\partial\R_+^n} U_i^{p}U_j \lesssim \hat{\epsilon}\norm{\nabla\rho}_{L^2(\R_+^n)} + \norm{\lapl u+\abs{u}^{p-1}u}_{H^{-1}} + \norm{\nabla\rho}_{L^2(\R_+^n)}^{\min(2,p)}.
		\end{equation}
	\end{proposition}

\vskip0.2in
\noindent{\bf Proof of Proposition \ref{prop:bubbles}.}
		Let $U_i=U[z_i,\lambda_i]$ for $1\le i\le \nu$, and we can assume $\lambda_1\ge \lambda_2\ge \cdots \ge \lambda_\nu$. Let $\{\Phi_i\}_{i=1}^\nu$ be the cut-off functions built in Lemma \ref{lemma:localization} for some fixed $\epsilon>0$ depending on $\delta$. We prove the statement by induction on the index $i$. Fix $1\le i\le \nu$ and assume the statement holds for any $1\le j< i$, denote $U=U_i,\, \alpha=\alpha_i$, $V=\sum\limits_{j\neq i}\alpha_jU_j$ and $\Phi=\Phi_i$. Without loss of generality, we may assume $U=U[0,1]$.
		
		First we have the following identity on the boundary $\partial \R_+^n$:
		\begin{align*}
			& (\alpha-\alpha^p)U^p - p(\alpha U)^{p-1}V \\
			={}& -\frac{\partial u}{\partial t} - |u|^{p-1}u - \sum\limits_{j\neq i}\alpha_j U_j^p + \frac{\partial\rho}{\partial t} + p(\alpha U)^{p-1}\rho + \left[|\sigma+\rho|^{p-1}(\sigma+\rho)-\sigma^p - p\sigma^{p-1}\rho \right] \\
			& + \left[(\alpha U+V)^p - (\alpha U)^p - p(\alpha U)^{p-1} V\right] + \left[p\sigma^{p-1}\rho - p(\alpha U)^{p-1}\rho\right].
		\end{align*}
		By Lemma \ref{lemma:localization}-(2), in the region $\{\Phi>0\}$, we have
		\begin{align*}
			& \sum\limits_{j\neq i}\alpha_j U_j^p = o(U^{p-1}V), \\
			& \Big| |\sigma+\rho|^{p-1}(\sigma+\rho)-\sigma^p - p\sigma^{p-1}\rho \Big| \lesssim \abs{\rho}^p + \chi_{\{n=3\}}U^{p-2}\abs{\rho}^2, \\
			& \Big| (\alpha U+V)^p - (\alpha U)^p - p(\alpha U)^{p-1} V \Big| \lesssim \abs{V}^p + \chi_{\{n=3\}} U^{p-2}V^2 = o(U^{p-1}V), \\
			& \Big| p\sigma^{p-1}\rho - p(\alpha U)^{p-1}\rho \Big| = o(U^{p-1}\abs{\rho}).
		\end{align*}
		Here we denote $o(E)$ for any expression that goes to zero when divided by $E$ if $\delta\to 0$, and $\chi_{\{n=3\}}$ means terms that only appears when $n=3$. Applying these estimates into the identity we get
		\begin{equation}\label{est: step 1}
			\begin{split}
				& \left| (\alpha-\alpha^p)U^p - (p\alpha^{p-1}+ o(1))U^{p-1}V - \frac{\partial\rho}{\partial t} + \left(\frac{\partial u}{\partial t} + |u|^{p-1}u \right) - p(\alpha U)^{p-1}\rho\right|\\
				\lesssim{}& \abs{\rho}^p + o(U^{p-1}\abs{\rho}) + \chi_{\{n=3\}} U^{p-2}\rho^2.
			\end{split}
		\end{equation}
		Let $\xi$ be either $U$ or $\partial_\lambda U$, then $\int_{\partial\R_+^n}U^{p-1}\xi\rho = 0$ by the orthogonal conditions of $\rho$. Testing \eqref{est: step 1} by $\xi\Phi$, we get
		\begin{equation}\label{est: step 2}
			\begin{split}
				& \left|\int_{\partial\R_+^n} \Big[ (\alpha-\alpha^p)U^p-(p\alpha^{p-1}+o(1))U^{p-1}V\Big]\xi\Phi\right|\\
				\lesssim{}& \left|-\int_{\partial\R_+^n}\frac{\partial\rho}{\partial t}\xi\Phi+\int_{\partial\R_+^n}\left(\frac{\partial u}{\partial t} + |u|^{p-1}u\right)\xi\Phi\right|+\left|\int_{\partial\R_+^n}U^{p-1}\xi\rho\Phi\right| \\
				& + \int_{\partial\R_+^n}\abs{\rho}^p\abs{\xi}\Phi + o\left(\int_{\partial\R_+^n} U^{p-1}\abs{\xi}\abs{\rho}\Phi\right) + \chi_{\{n=3\}} \int_{\partial\R_+^n} U^{p-2}\rho^2\abs{\xi}\Phi.
			\end{split}
		\end{equation}
		We now bound each term in the right-hand side of \eqref{est: step 2}. Noticing that
		\begin{equation}\label{est: step 3.1}
			\begin{split}
				-\int_{\partial\R_+^n}\frac{\partial\rho}{\partial t}\xi\Phi ={}& \int_{\R_+^n} \nabla\rho\cdot\nabla(\xi\Phi) + \xi\Phi\Delta\rho =  \int_{\R_+^n} \nabla\rho\cdot\nabla(\xi(\Phi-1)) + \xi\Phi\lapl u.
			\end{split}
		\end{equation}
		\begin{equation}\label{est: step 3.2}
			\begin{split}
				\left|\int_{\partial\R_+^n}-\frac{\partial\rho}{\partial t}\xi\Phi+\int_{\partial\R_+^n}\left(\frac{\partial u}{\partial t} + |u|^{p-1}u\right)\xi\Phi\right|
				\le{}& \norm{\lapl u + \abs{u}^{p-1}u}_{H^{-1}} \norm{\nabla(\xi\Phi)}_{L^{2}(\R_+^n)}\\
				& + \norm{\nabla\rho}_{L^2(\R_+^n)}\norm{\nabla(\xi(\Phi-1))}_{L^2(\R_+^n)},\\
				\left|\int_{\partial\R_+^n}U^{p-1}\xi\rho\Phi\right| = \left|\int_{\partial\R_+^n}U^{p-1}\xi\rho(\Phi-1)\right| \lesssim{}& \norm{\nabla\rho}_{L^2(\R_+^n)}\left(\int_{\{\Phi<1\}\cap\partial\R_+^n} (U^{p-1}\abs{\xi})^{\frac{p+1}{p}}\right)^{\frac{p}{p+1}}, \\
				\int_{\partial\R_+^n}\abs{\rho}^p\abs{\xi}\Phi \le \int_{\partial\R_+^n} \abs{\rho}^p\abs{\xi} \lesssim{}& \norm{\nabla\rho}_{L^2(\R_+^n)}^p \norm{\xi}_{L^{p+1}(\partial\R_+^n)},\\
				\int_{\partial\R_+^n} U^{p-1}\abs{\xi}\abs{\rho}\Phi \le \int_{\partial\R_+^n} U^{p-1}\abs{\xi}\abs{\rho} \lesssim{}& \norm{\nabla\rho}_{L^2(\R_+^n)}\norm{U^{p-1}\xi}_{L^{\frac{p+1}{p}}(\partial\R_+^n)}, \\
				\chi_{\{n=3\}} \int_{\partial\R_+^n} U^{p-2}\rho^2\abs{\xi}\Phi \le \int_{\partial\R_+^n} U^{p-2}\rho^2\abs{\xi} \lesssim{}& \norm{\nabla\rho}_{L^2(\R_+^n)}^2\norm{U^{p-2}\xi}_{L^{\frac{p+1}{p-1}}(\partial\R_+^n)}.
			\end{split}
		\end{equation}
		
		Again, thanks to Lemma \ref{lemma:localization}-(1) and Lemma \ref{lemma:localization}-(3), and the observation that $\abs{\partial_\lambda U} \le \frac{n-2}{2}U$, we have
		\begin{align*}
			& \norm{\nabla(\xi(\Phi-1))}_{L^2(\R_+^n)}^2 = \int_{\R_+^n} |\nabla(\xi(\Phi-1))|^2\\
			={}& \int_{\R_+^n} |\nabla\xi|^2(1-\Phi)^2 + \int_{\R_+^n} |\xi|^2|\nabla\Phi|^2 + 2(\Phi-1)\xi\nabla\xi\cdot\nabla\Phi \\
			={}& -\int_{\R_+^n} \xi\, \nabla\cdot((1-\Phi)^2\nabla\xi) - \int_{\partial\R_+^n} (1-\Phi)^2\xi\frac{\partial\xi}{\partial t}\\
			&+ \int_{\R_+^n} |\xi|^2|\nabla\Phi|^2 + \int_{\R_+^n} 2(\Phi-1)\xi\nabla\xi\cdot\nabla\Phi \\
			\lesssim{}& \norm{\xi}_{L^{2^*}(\R_+^n)}\norm{\nabla\Phi}_{L^n(\R_+^n)}\norm{\nabla\xi}_{L^2(\R_+^n)} + \int_{\{\Phi<1\}\cap\partial\R_+^n} U^{p+1} + \norm{\nabla\Phi}_{L^n(\R_+^n)}^2\norm{\xi}_{L^{2^*}(\R_+^n)}^2 \\
			\lesssim{}& \epsilon + \epsilon^2 = o(1).
		\end{align*}
		Thus,
		\begin{equation}\label{est: step 4}
			\begin{alignedat}{3}
				& \norm{\nabla(\xi(\Phi-1))}_{L^2(\R_+^n)} = o(1), \quad  &&\norm{\nabla(\xi\Phi)}_{L^2(\R_+^n)}\lesssim 1, \quad \int_{\{\Phi<1\}\cap\partial\R_+^n} (U^{p-1}\abs{\xi})^{\frac{2(n-1)}{n}} = o(1),\\
				& \norm{\xi}_{L^{p+1}(\partial\R_+^n)} \lesssim 1, &&\norm{U^{p-1}\xi}_{L^{\frac{2(n-1)}{n}}(\partial\R_+^n)} \lesssim 1, \quad\quad  \norm{U^{p-2}\xi}_{L^{\frac{p+1}{p-1}}(\partial\R_+^n)}\lesssim 1.
			\end{alignedat}
		\end{equation}
		Using \eqref{est: step 3.1}, \eqref{est: step 3.2} and \eqref{est: step 4}, it follows by \eqref{est: step 2} that
		\begin{equation}\label{est: step 5}
			\begin{split}
				& \left|\int_{\partial\R_+^n} \Big[ (\alpha-\alpha^p)U^p-(p\alpha^{p-1}+o(1))U^{p-1}V\Big]\xi\Phi\right| \\
				\lesssim{}& o(1)\norm{\nabla\rho}_{L^2(\R_+^n)} + \norm{\nabla\rho}_{L^2(\R_+^n)}^{\min(2,p)} +\norm{\lapl u+\abs{u}^{p-1}u}_{H^{-1}}.
			\end{split}
		\end{equation}
		Now, split $V$ into $V_1\defeq \sum\limits_{j<i}\alpha_jU_j$ and $V_2\defeq \sum\limits_{j>i}\alpha_jU_j$. By induction, we may assume that the statement holds for all $j<i$, and by $\int_{\partial\R_+^n} U_i^p U_j = \int_{\R_+^n} \nabla U_i \cdot \nabla U_j = \int_{\partial\R_+^n} U_j^p U_i$, we get
		\begin{equation}\label{est: step 5.1}
			\int U^{p-1}V_1\abs{\xi}\Phi \lesssim o(1)\norm{\nabla\rho}_{L^2(\R_+^n)} +\norm{\lapl u+\abs{u}^{p-1}u}_{H^{-1}}+ \norm{\nabla\rho}^{\min(2,p)}_{L^2(\R_+^n)}.
		\end{equation}
		For $V_2$, thanks to Lemma \ref{lemma:localization}-(4), $V_2(x)=(1+o(1))V_2(0)$ in the region $\{\Phi>0\}\cap \partial\R_+^n$. To show \eqref{estimate of alpha}, if $\alpha=1$ then it's done. If not, let $\theta\defeq \frac{p\alpha^{p-1}V_2(0)}{\alpha-\alpha^p}$, then by \eqref{est: step 5} and \eqref{est: step 5.1} we have
		\begin{equation}\label{est: step 6.1}
			\begin{split}
				& \abs{\alpha-\alpha^p}\left|\int_{\partial\R_+^n} \big( U^p - (1+o(1))\theta U^{p-1} \big)\xi\Phi\right| \\
				\lesssim{}& o(1)\norm{\nabla\rho}_{L^2(\R_+^n)} + \norm{\lapl u+\abs{u}^{p-1}u}_{H^{-1}} + \norm{\nabla\rho}^{\min(2,p)}_{L^2(\R_+^n)}.
			\end{split}
		\end{equation}
		Since almost all the mass of $U$ on the boundary is in the region $\{\Phi=1\}$, and $\Phi\equiv 1$ on a large ball centered at $0$, we get
		\begin{equation*}
			\int_{\partial\R_+^n} \big( U^p - (1+o(1))\theta U^{p-1} \big)\xi\Phi = \int_{\partial\R_+^n} U^p\xi -\int_{\partial\R_+^n}\theta U^{p-1}\xi + o(1).
		\end{equation*}
		We can easily check that there exists $\epsilon'>0$ such that for any $\theta\in \R$,
		\begin{equation*}
			\max_{\xi\in \{U,\partial_\lambda U\}} \left(\left|\int_{\partial\R_+^n} U^p\xi -\int_{\partial\R_+^n}\theta U^{p-1}\xi\right|\right) \ge \epsilon'.
		\end{equation*}
		Thus, choosing suitable $\xi$ in \eqref{est: step 6.1} we get \eqref{estimate of alpha}. Choosing $\xi=U$ in \eqref{est: step 6.1} implies
		\begin{equation*}
			\left|\int_{\partial\R_+^n} U^pV\Phi \right| \lesssim o(1)\norm{\nabla\rho}_{L^2(\R_+^n)} + \norm{\lapl u+\abs{u}^{p-1}u}_{H^{-1}} + \norm{\nabla\rho}^{\min(2,p)}_{L^2(\R_+^n)}.
		\end{equation*}
		And in particular, since $\Phi\equiv 1$ in $B(0,1)$,
		\begin{equation*}
			\left|\int_{B(0,1)\cap\partial\R_+^n} U^pU_j\right| \lesssim o(1)\norm{\nabla\rho}_{L^2(\R_+^n)} + \norm{\lapl u+\abs{u}^{p-1}u}_{H^{-1}} + \norm{\nabla\rho}^{\min(2,p)}_{L^2(\R_+^n)}.
		\end{equation*}
		This, combining with Lemma \ref{estimate of bubbles 2}, shows \eqref{estimate of pairing} for all $j>i$. Since the case $j<i$ has been shown in \eqref{est: step 5.1}, we get that the estimate \eqref{estimate of pairing} holds for all $j\neq i$, and this concludes the proof by induction.  \hfill$\Box$

\vskip0.36in
As a direct consequence of Theorem \ref{ql1} and Theorem \ref{thm:main_close}, now we can prove Theorem  \ref{them4}  about the  stability result for nonnegative functions.

\vskip0.236in
\noindent{\bf Proof of Theorem  \ref{them4}.}  Based on the qualitative result Theorem \ref{ql1}, there exists $\epsilon>0$ such that, if $\norm{\lapl u+u^p}_{H^{-1}}<\epsilon$, then there exist $\nu$ Escobar bubbles $U_1,U_2,\dots, U_\nu$ such that
\begin{equation*}
\norm*{\nabla u - \sum\limits_{i=1}^{\nu}\nabla U_i}_{L^2(\R_+^n)} \le \delta.
\end{equation*}
This is exactly the hypothesis of Theorem \ref{thm:main_close}. When $\norm{\lapl u+u^p}_{H^{-1}}\geq\epsilon$, Our results follows directly.
\vskip0.13in
Finally, we show that linear decay is the sharpest in our cases. For $0<\delta<1$, we show that there exist a nonnegative function $u\in H^1(\mathbb{R}_{+}^n)$, a constant $C$ depending only on $n,\nu$ and a $\delta$-interacting family $U_1,...,U_\nu$ of Escobar bubbles such that
\begin{equation}\label{z=1}
\left\Vert \nabla\left(u-\sum\limits_{i=1}^{\nu}U_i\right)\right\Vert_{L^2}\leq 2\inf_{W_1,...,W_\nu}\left\Vert \nabla\left(u-\sum\limits_{i=1}^{\nu}W_i\right)\right\Vert_{L^2}\leq \delta
\end{equation}
and
\begin{equation}\label{z=2}
\left\Vert \Delta u +u^p\right\Vert_{H^{-1}}\leq C\left\Vert \nabla\left(u-\sum\limits_{i=1}^{\nu}U_i\right)\right\Vert_{L^2}.
\end{equation}
To prove the estimates \eqref{z=1}-\eqref{z=2}, we  first fix $\nu$ distinct directions $\theta_1,...,\theta_\nu$ in $\mathbb{S}^{n-2}$. For any $0<\epsilon\ll 1\ll R<\infty$, we set $u_{\epsilon,R}(x,t)=\sum\limits_{i=1}^{\nu}U_{R,i}(x,t)+\rho_{\epsilon}(x,t)$, where $U_{R,i}=U[R\theta_i,1]$ and $ \rho_{\epsilon}(x,t)=\left(1-\frac{\sqrt{|x|^2+|t-1|^2}}{\epsilon}\right)_{+}$. Note that $\rho_\epsilon$ is radially symmetric about the point $(0,1)$. It is straightforward to compute that $|\nabla\rho_\epsilon|=\frac{1}{\epsilon}$ within $B^n_{\epsilon}(0,1)$. And consequently, we have $\norm{\nabla\rho_\epsilon}_{L^2}\leq C_n \epsilon^\frac{n-2}{2}$ for some constant $C_n$. When $R$ is sufficiently large, the family $\{U_{R,i}\}_{1\leq i\leq \nu}$ forms a clear $\delta$-interacting set, and moreover, $\Vert\nabla U_{R,i}\Vert_{L^\infty(B^n_{\epsilon})}$ tend to $0$ uniformly as $R\rightarrow\infty$.
		
		Now fix $\epsilon=\left(\frac{\delta}{4C_n}\right)^{\frac{2}{n-2}}$, then there exists $R_{\epsilon}\gg1$ such that, when $R\geq R_{\epsilon}$, we have
		\begin{equation*}
			\frac{\delta}{8}\leq\inf_{W_1,...,W_\nu}\left\Vert \nabla\left(u_{\epsilon,R}-\sum\limits_{i=1}^{\nu}W_i\right)\right\Vert_{L^2}\leq \left\Vert \nabla\left(u_{\epsilon,R}-\sum\limits_{i=1}^{\nu}U_{R,i}\right)\right\Vert_{L^2}\leq\frac{\delta}{2}
		\end{equation*}
		and
		\begin{equation*}
			\begin{aligned}
				\left\Vert \Delta u_{\epsilon,R}+u_{\epsilon,R}^p\right\Vert_{H^{-1}}&\leq  \left\Vert \sum\limits_{i=1}^\nu\left(\Delta U_{R,i}+U_{R,i}^p\right)+\Delta\rho_{\epsilon}+\rho_{\epsilon}^p\right\Vert_{H^{-1}}+\delta\\
				&=\left\Vert \Delta\rho_{\epsilon}\right\Vert_{H^{-1}}+\delta=\frac{5}{4}\delta.
			\end{aligned}
		\end{equation*}
		Thus $u_{\epsilon,R}$ is a function satisfying the requirements.  \hfill$\Box$

\section{Refined estimates for the one bubble case}\label{sec7}
	In this section, we aim to give an explicit upper bound for the stability constant $C_{\mathrm{CP}}(n,1)$. To achieve it, let us first study the non-degeneracy of \eqref{eqt}.
	
	Consider the following conformal map from the unit ball $\B^n$ to $\R_+^n$ ($n\geq 3$):
	\begin{equation*}
		F(y',y_n)=\left(\frac{2y'}{(1+y_n)^2+|y'|^2},\frac{1-|y|^2}{(1+y_n)^2+|y'|^2}\right),\quad y'=(y_1,\dots,y_{n-1})\in\R^{n-1}.
	\end{equation*}
	The inverse map is given by
	\begin{equation*}
		F^{-1}(x,t)=\left(\frac{2x}{(1+t)^2+|x|^2},\frac{1-t^2-|x|^2}{(1+t)^2+|x|^2}\right),\quad x=(x_1,\dots,x_{n-1})\in\R^{n-1}.
	\end{equation*}
	For any function $\varphi\in H^1(\R_+^n)$, we can define $\F[\varphi]$ in $\B^n$ by
	\begin{equation*}
		\F[\varphi](y)=(U[0,1]^{-1}\varphi)\circ F(y)\quad \text{for }y\in\B^n.
	\end{equation*}
	It is not hard to verify the following identities (we refer to \cite[Proof of Proposition 1.2]{Escobar} and \cite[Proof of Lemma 3.1]{Ho} for similar identities):
	\begin{equation*}
		\int_{\R^{n-1}}|\varphi|^{2^\dagger}\,dx=\left(\frac{n-2}{2}\right)^{n-1}\int_{\partial \B^n}|\F[\varphi]|^{2^\dagger}d\sigma,
	\end{equation*}
	\begin{equation*}
		\int_{\R_+^n}|\nabla\varphi|^2\,dxdt=\left(\frac{n-2}{2}\right)^{n-2}\left(\int_{\B^n}|\nabla \F[\varphi]|^2\,dy+\frac{n-2}{2}\int_{\partial\B^n}\F[\varphi]^2\,d\sigma\right),
	\end{equation*}
	\begin{equation*}
		\int_{\R^{n-1}}U[0,1]^{p-1}\varphi^2\,dx=\left(\frac{n-2}{2}\right)^{n-1}\int_{\partial\B^n}\F[\varphi]^2\,d\sigma,
	\end{equation*}
	\begin{equation*}
		\Delta\varphi=0\quad\text{in }\R^n_+\iff \Delta\F[\varphi]=0\quad\text{in }\B^n,
	\end{equation*}
	\begin{equation*}
		\F[U[0,1]]=1,\quad\F[\partial_\lambda U[0,1]]=-\frac{n-2}{2}y_n,\quad\F[\partial_{z_i}U[0,1]]=-\frac{n-2}{2}y_i,\quad1\leq i\leq n-1.
	\end{equation*}
	In the following we always equip $H^1(\B^n)$ with the norm
	\begin{equation*}
		\norm{u}_{H^1(\B^n)}^2\coloneqq \left(\frac{n-2}{2}\right)^{n-2}\left(\int_{\B^n}|\nabla u|^2\,dy+\frac{n-2}{2}\int_{\partial\B^n}u^2\,d\sigma\right).
	\end{equation*}
	From Poincar\'e inequality, this norm is equivalent to the standard Sobolev norm on $\B^n$. Under this norm, $H^1(\B^n)$ is isometric to $H^1(\R_+^n)$ via map $\F$.
	
	It is well known (see \cite{stek} for example) that the quotient
	\begin{equation*}
		\frac{\int_{\B^n}|\nabla u|^2\,dy}{\int_{\partial\B^n}u^2\,d\sigma}
	\end{equation*}
	is related to the Steklov eigenvalue problem on $\B^n$:
	\begin{equation}\label{stek}
		\begin{cases}
			\Delta u=0&\text{in }\B^n\\
			\frac{\partial u}{\partial \Vec{n}}=\mu u&\text{on }\partial\B^n,
		\end{cases}
	\end{equation}
	where $\Vec{n}$ denotes the outward unit normal vector. The Steklov eigenspaces $E_k$ are the restrictions of the spaces $H_k^n$ of homogeneous harmonic polynomials of degree k in $\R^n$. The corresponding eigenvalue $\mu_k=k$ has multiplicity
	\begin{equation*}
		\dim\,E_k=C_{n+k-1}^{n-1}-C_{n+k-3}^{n-1}.
	\end{equation*}
	Note that $E_0$ is spanned by the constant function $1=\F[U[0,1]]$ and $E_1$ is spanned by the $n$ coordinate functions:
	\begin{equation*}
		\left\{y_i=-\frac{2}{n-2}\F[\partial_{z_i} U[0,1]],\ 1\leq i\leq n-1;\ y_n=-\frac{2}{n-2}\F[\partial_\lambda U[0,1]]\right\}.
	\end{equation*}
	The eigenvalue problem provides an orthogonal decomposition of $H^1(\B^n)$:
	\begin{equation}\label{decom}
		H^1(\B^n)=\bigoplus_{i=-1}^\infty E_i,
	\end{equation}
	where $E_{-1}\coloneqq \left\{u\in H^1(\B^n)|\;u=0 \ \text{on }\partial\B^n\right\}=H^1_0(\B^n)$.
	
	Using the isometry $\F$ and identities above, we can transform the equation \eqref{stek} to
	\begin{equation}\label{stek2}
		\begin{cases}
			\Delta u=0&\text{in }\R^n_+\\
			\frac{\partial u}{\partial t}=-\left(1+\frac{2\mu}{n-2}\right)U[0,1]^{p-1}u&\text{on }\R^{n-1}.
		\end{cases}
	\end{equation}
	Set
	\begin{equation}\label{kappa}
		\kappa_k=1+\frac{2\mu_k}{n-2}=1+\frac{2k}{n-2}
	\end{equation} and $R_i=\F^{-1}(E_i)$, then the decomposition \eqref{decom} yields
	\begin{equation}\label{decom2}
		H^1(\R_+^n)=\bigoplus_{i=-1}^\infty R_i,
	\end{equation}
	where $R_{-1}=\left\{u\in H^1(\R_+^n)|\;u=0\ \text{on }\R^{n-1}\right\}$ and $R_k$ is the eigenspace of \eqref{stek2} with eigenvalue $\kappa_k$ for $k\geq 0$. From the explicit expressions of $E_0$ and $E_1$, we have
	\begin{equation*}
		R_0=\F^{-1}(E_0)=\text{Span}\left\{U[0,1]\right\}
        \end{equation*}
        and
            
        \begin{equation*}R_1=\F^{-1}(E_1)=\text{Span}\left\{\partial_{z_i}U[0,1],1\leq i\leq n-1;\partial_{\lambda}U[0,1]\right\}.
	\end{equation*}
	
	Let $T_U=R_0\oplus R_1$ be the tangent space at $U[0,1]$ in the manifold $\mathcal{M}_{\mathrm{E}}$ of Escobar bubbles. Thanks to the arguments above, we can establish the following non-degeneracy result:
	\begin{proposition}[spectral gap]\label{spect0}
		Let $\rho\in T_U^{\perp}$. Then
		\begin{equation*}
			\norm{\rho}^2_{H^1( \R_+^n)}\geq \frac{n+2}{n-2}\int_{\R^{n-1}}U[0,1]^{p-1}\rho^2.
		\end{equation*}
		Moreover, the equality holds if and only if $\rho\in R_2$, i.e. $\F[\rho]$ is a homogeneous harmonic polynomial of degree $2$.
	\end{proposition}
	
	Now we can state and prove our main results.
	\begin{proposition}(asymptotic behavior near $\mathcal{M}_{\mathrm{E}}$)\label{asymp}
		Let $n\geq 3$. Assume $(u_k)_k\subset H^1(\R_+^n)$ is a sequence of functions such that
		\begin{equation*}
			\frac{1}{2}S_{\mathrm{E}}(n)^{n-1}\leq\int_{\R_+^n}|\nabla u_k|^2\leq \frac{3}{2}S_{\mathrm{E}}(n)^{n-1},
		\end{equation*}
		and
		\begin{equation*}
			\left\Vert \Delta u_k +|u_k|^{p-1}u_k\right\Vert_{H^{-1}}\rightarrow 0\quad as\; k\rightarrow \infty.
		\end{equation*}
		Then
		\begin{equation}\label{asympt}
			\liminf_{k\rightarrow\infty}\frac{\left\Vert \Delta u_k +|u_k|^{p-1}u_k\right\Vert_{H^{-1}}}{d(u_k,\mathcal{M}_{\mathrm{E}})}\geq \frac{2}{n+2}.
		\end{equation}
		Moreover, equality holds if and only if the following holds: Assume after suitable normalizations, translations and dilations, $d(u_k,\mathcal{M}_{\mathrm{E}})=\norm{u_k-U[0,1]}_{H^1}$. Let $u_k=\sum\limits_{i=-1}^{\infty}\rho_{k,i}$ be the unique decomposition with respect to \eqref{decom2}. Then (up to some subsequence)
		\begin{equation}\label{asympt1}
			\norm{u_k-U[0,1]-\rho_{k,2}}_{H^1}=o(\norm{\rho_{k,2}}_{H^1}).
		\end{equation}
	\end{proposition}

\vskip0.3in

\begin{remark}
A natural corollary of above theorem is that $C_{\mathrm{CP}}(n,1)\leq\frac{2}{n+2}$. It seems that the proof of Theorem \ref{thm:main_close} may only give the inequality \eqref{asympt}. The proof there cannot give such upper bound because the identification of equality case is not available there. One possible reason is that \eqref{main est} causes a loss.
\end{remark}
	

\vskip0.23in

	Before the proof, let us make two observations:
	\begin{equation*}
		\left\Vert \Delta u_k +|u_k|^{p-1}u_k\right\Vert_{H^{-1}}=\norm{u_k-\mathcal{P}[|u_k|^{p-1}u_k]}_{H^1},
	\end{equation*}
	\begin{align}\label{obs1}
		\int_{\R^{n-1}}u\mathcal{P}[v]=\int_{\R^{n-1}}v\mathcal{P}[u]\quad\text{for any }u,v\in L^{\frac{2(n-1)}{n}}(\R^{n-1}).
	\end{align}
	Recall that the operator $\mathcal{P}$ is defined in \eqref{neumann}.
	
	For simplicity, in the following proofs, we will write $U$ to denote $U[0,1]$.
	
\vskip0.3in

\noindent{\bf Proof of Proposition \ref{asymp}.}
		From Remark \ref{energy gap}, we can assume $d(u_k,\mathcal{M}_{\mathrm{E}})=\norm{u_k-U}_{H^1}$ without loss of generality. Take the decomposition $u_k=\sum\limits_{i=-1}^{\infty}\rho_{k,i}$ with respect to \eqref{decom2}. Since $U$ minimizes the distance between $u_k$ and $\mathcal{M}_{\mathrm{E}}$, one can derive $\rho_{k,1}=0$ by taking variations. Assume $\rho_{k,0}=\beta_kU$. Since $d(u_k,\mathcal{M}_{\mathrm{E}})\rightarrow 0$, we have $\beta_k\rightarrow 1$. Set $v_k=\rho_{k,-1}+\sum\limits_{i=3}^{\infty}\rho_{k,i}$ and $w_k=\rho_{k,-1}+\sum\limits_{i=2}^{\infty}\rho_{k,i}$. Then $u_k=\beta_kU+w_k$. Let us expand the dual norm:
			\begin{align}\label{rev1}
				\left\Vert \Delta u_k +|u_k|^{p-1}u_k\right\Vert_{H^{-1}}^2&=\norm{u_k-\mathcal{P}[|u_k|^{p-1}u_k]}_{H^1}^2\nonumber\\
				&=\norm{u_k}_{H^1}^2+\norm{\mathcal{P}[|u_k|^{p-1}u_k]}_{H^1}^2+2\left\langle u_k,\mathcal{P}[|u_k|^{p-1}u_k]\right\rangle_{H^1}\nonumber\\
				&=\norm{u_k}_{H^1}^2-2\int_{\R^{n-1}}|u_k|^{p+1}+\int_{\R^{n-1}}|u_k|^{p-1}u_k\mathcal{P}[|u_k|^{p-1}u_k]\,.
			\end{align}
        
		Next we aim to estimate the three terms in \eqref{rev1}. For the first two terms, from the orthogonal decomposition and the estimate \eqref{estimate:minus}, we have
		\begin{align}
            &\norm{u_k}_{H^1}^2=\beta_k^2\norm{U}_{H^1}^2+\norm{w_k}_{H^1}^2
        \end{align}
        and
        \begin{align}
			&\int_{\R^{n-1}}|u_k|^{p+1}=\beta_k^{p+1}\int_{\R^{n-1}}U^{p+1}+\frac{p(p+1)}{2}\beta_k^{p-1}\int_{\R^{n-1}}U^{p-1}w_k^2+o(\norm{w_k}_{H^1}^2).
        \end{align}
        For the third term, thanks to the linearity of $\mathcal{P}$, the observation \eqref{obs1} and the orthogonal conditions, we have
        \begin{align}
			&\int_{\R^{n-1}}|u_k|^{p-1}u_k\mathcal{P}[|u_k|^{p-1}u_k]\nonumber\\
			={}&\int_{\R^{n-1}}\left(\beta_k^{p}U^p+p\beta_k^{p-1}U^{p-1}w_k+A\right)\mathcal{P}\left[\beta_k^{p}U^p+p\beta_k^{p-1}U^{p-1}w_k+A\right]\nonumber\\
			={}&\int_{\R^{n-1}}\beta_k^{p}U^p\mathcal{P}[\beta_k^{p}U^p]+\int_{\R^{n-1}}p\beta_k^{p-1}U^{p-1}w_k\mathcal{P}[p\beta_k^{p-1}U^{p-1}w_k]+\int_{\R^{n-1}}A\mathcal{P}[A]\nonumber\\
			&+2\int_{\R^{n-1}}\beta_k^{p}U^p\mathcal{P}[p\beta_k^{p-1}U^{p-1}w_k]+2\int_{\R^{n-1}}p\beta_k^{p-1}U^{p-1}w_k\mathcal{P}[A]+2\int_{\R^{n-1}}\beta_k^{p}U^p\mathcal{P}[A]\nonumber\\
			={}&\beta_k^{2p}\int_{\R^{n-1}}U^{p+1}+p^2\beta_k^{2p-2}\int_{\R^{n-1}}U^{p-1}w_k\mathcal{P}[U^{p-1}w_k]+2\beta_k^{p}\int_{\R^{n-1}}UA\nonumber\\
            &+\int_{\R^{n-1}}A\mathcal{P}[A]+2p\beta_k^{p-1}\int_{\R^{n-1}}U^{p-1}w_k\mathcal{P}[A],
		\end{align}
        where $A=|u_k|^{p-1}u_k-\beta_k^{p}U^p-p\beta_k^{p-1}U^{p-1}w_k$. From the H\"older inequality, the inequality \eqref{escobar}, \eqref{ooo1} and the estimate \eqref{estimate:minus}, we can obtain
        \begin{align}
            \int_{\R^{n-1}}|A\mathcal{P}[A]|\leq \norm*{A}_{L^\frac{2n-2}{n}(\R^{n-1})}\norm*{\mathcal{P}[A]}_{L^\frac{2n-2}{n-2}(\R^{n-1})}\lesssim \norm*{A}_{L^\frac{2n-2}{n}(\R^{n-1})}^2=o(\norm{w_k}_{H^1}^2).
        \end{align}
        Similarly, we have
        \begin{align}
            \int_{\R^{n-1}}U^{p-1}w_k\mathcal{P}[A]=o(\norm{w_k}_{H^1}^2).
        \end{align}
        To estimate $\int_{\R^{n-1}}UA$, we split the integration domain $\R^{n-1}$ into two parts: $\{2|w_k|\leq U\}$ and $\{2|w_k|\geq U\}$. In the first part, we can expand $|u_k|^{p-1}u_k$ up to third order and derive
        \begin{align}\label{qqq1}
            UA=\frac{p(p-1)}{2}\beta_k^{p-2}U^{p-1}w_k^2+O(\chi_{\{p\geq2\}}U^{p-2}|w_k|^3+|w_k|^{p+1}).
        \end{align}
        In the second part, we can see \eqref{qqq1} also holds pointwisely. Therefore, we deduce that
        \begin{align}\label{vvv6}
            \int_{\R^{n-1}}UA=\frac{p(p-1)}{2}\beta_k^{p-2}\int_{\R^{n-1}}U^{p-1}w_k^2+o(\norm{w_k}_{H^1}^2).
        \end{align}
        Combining the  estimates \eqref{rev1}-\eqref{vvv6} above and using $\beta_k=1+o(1)$, we finally get
		\begin{equation*}
			\begin{aligned}
				\left\Vert \Delta u_k +|u_k|^{p-1}u_k\right\Vert_{H^{-1}}^2={}&(\beta_k-\beta_k^p)^2\norm{U}_{H^1}^2+\norm{w_k}_{H^1}^2-2p\int_{\R^{n-1}}U^{p-1}w_k^2\\
				&+p^2\int_{\R^{n-1}}U^{p-1}w_k\mathcal{P}[U^{p-1}w_k]+o(\norm{w_k}_{H^1}^2).
			\end{aligned}
		\end{equation*}
		Note that
		\begin{equation*}
			\norm{w_k}_{H^1}^2=\norm{\rho_{k,-1}}_{H^1}^2+\sum\limits_{i=2}^{\infty}\norm{\rho_{k,i}}_{H^1}^2,
		\end{equation*}
		\begin{equation*}
			\int_{\R^{n-1}}U^{p-1}w_k^2=\sum\limits_{i=2}^{\infty}\kappa_i^{-1}\norm{\rho_{k,i}}_{H^1}^2,
		\end{equation*}
		\begin{equation*}
			\int_{\R^{n-1}}U^{p-1}w_k\mathcal{P}[U^{p-1}w_k]=\sum\limits_{i=2}^{\infty}\kappa_i^{-2}\norm{\rho_{k,i}}_{H^1}^2,
		\end{equation*}
		where $\kappa_k$ is defined in \eqref{kappa}. Hence we obtain
		\begin{align}\label{dddd}
				\left\Vert \Delta u_k +|u_k|^{p-1}u_k\right\Vert_{H^{-1}}^2={}&(\beta_k-\beta_k^p)^2\norm{U}_{H^1}^2+\norm{\rho_{k,-1}}_{H^1}^2\nonumber\\
				&+\sum\limits_{i=2}^{\infty}(1-p\kappa_i^{-1})^2\norm{\rho_{k,i}}_{H^1}^2+o(\norm{w_k}_{H^1}^2)\nonumber\\
				\geq&(\beta_k-\beta_k^p)^2\norm{U}_{H^1}^2+\norm{\rho_{k,-1}}_{H^1}^2\nonumber\\
				&+(1-p\kappa_2^{-1})^2\sum\limits_{i=2}^{\infty}\norm{\rho_{k,i}}_{H^1}^2+o(\norm{w_k}_{H^1}^2)\nonumber\\
				\geq&(\beta_k-\beta_k^p)^2\norm{U}_{H^1}^2+(1-p\kappa_2^{-1})^2\norm{w_k}_{H^1}^2+o(\norm{w_k}_{H^1}^2).
			\end{align}
		The denominator of the quotient can be evaluated by
		\begin{equation*}
			d(u_k,\mathcal{M}_{\mathrm{E}})^2=\norm{u_k-U}_{H^1}^2=(1-\beta_k)^2\norm{U}_{H^1}^2+\norm{w_k}_{H^1}^2.
		\end{equation*}
		Now the full quotient holds
		\begin{equation}\label{dddd1}
			\frac{\left\Vert \Delta u_k +|u_k|^{p-1}u_k\right\Vert_{H^{-1}}^2}{d(u_k,\mathcal{M}_{\mathrm{E}})^2}\geq \frac{(\beta_k-\beta_k^p)^2\norm{U}_{H^1}^2+(1-p\kappa_2^{-1})^2\norm{w_k}_{H^1}^2+o(\norm{w_k}_{H^1}^2)}{(1-\beta_k)^2\norm{U}_{H^1}^2+\norm{w_k}_{H^1}^2}.
		\end{equation}
		Since $\beta_k-\beta_k^p=(p-1+o(1))(1-\beta_k)$ and $p-1>1-p\kappa_2^{-1}$, \eqref{asympt} follows directly from \eqref{dddd1}.
		
		Assume that the equality in \eqref{asympt} holds, then \eqref{dddd} indicates
		\begin{equation*}
			\norm{\rho_{k,-1}}_{H^1}^2+\sum\limits_{i=3}^{\infty}(1-p\kappa_i^{-1})^2\norm{\rho_{k,i}}_{H^1}^2=o((1-p\kappa_2^{-1})^2\norm{\rho_{k,2}}_{H^1}^2)
		\end{equation*}
		and \eqref{dddd1} implies
		\begin{equation*}
			|1-\beta_k|=o(\norm{w_k}_{H^1}^2).
		\end{equation*}
		These are equivalent to \eqref{asympt1}. The proof is complete. \hfill$\Box$

\vskip0.3in

\noindent{\bf Proof of Theorem \ref{them5}.}
The main idea is to test the quotient with a suitable function. Let $u=U+\varepsilon\rho$ for some $\rho\in R_2,\norm{\rho}_{H^1}=1$ to be determined and $\varepsilon$ sufficiently small. From the implicit function theorem, we know $d(u,\mathcal{M}_{\mathrm{E}})=\norm{u-U}_{H^1}$. Since $|\rho|\leq CU$ by definition, here we can expand the dual norm up to third order:
\begin{equation*}
			\begin{aligned}
				\left\Vert \Delta u +|u|^{p-1}u\right\Vert_{H^{-1}}^2&=\norm{u}_{H^1}^2-2\int_{\R^{n-1}}|u|^{p+1}+\int_{\R^{n-1}}|u|^{p-1}u\mathcal{P}[|u|^{p-1}u],\\
				\norm{u}_{H^1}^2&=\norm{U}_{H^1}^2+\varepsilon^2\norm{\rho}_{H^1}^2=\norm{U}_{H^1}^2+\varepsilon^2,
			\end{aligned}
		\end{equation*}
		\begin{equation*}
			\begin{aligned}
				\int_{\R^{n-1}}|u|^{p+1}=&\int_{\R^{n-1}}U^{p+1}+\frac{p(p+1)}{2}\varepsilon^2\int_{\R^{n-1}}U^{p-1}\rho^2\\
				&+\frac{p(p+1)(p-1)}{6}\varepsilon^3\int_{\R^{n-1}}U^{p-2}\rho^3+o(\varepsilon^3)\\
				=&\norm{U}_{H^1}^2+\frac{p(p+1)}{2}\varepsilon^2\kappa_2^{-1}+\frac{p(p+1)(p-1)}{6}\varepsilon^3\int_{\R^{n-1}}U^{p-2}\rho^3+o(\varepsilon^3),
			\end{aligned}
		\end{equation*}
		\begin{align*}
			&\int_{\R^{n-1}}|u|^{p-1}u\mathcal{P}[|u|^{p-1}u]\\={}&\int_{\R^{n-1}}\left(U^p+pU^{p-1}\varepsilon\rho+\frac{p(p-1)}{2}U^{p-2}\varepsilon^2\rho^2+B\right)\mathcal{P}\left[U^p+pU^{p-1}\varepsilon\rho\right]\\
			&+\int_{\R^{n-1}}\left(U^p+pU^{p-1}\varepsilon\rho+\frac{p(p-1)}{2}U^{p-2}\varepsilon^2\rho^2+B\right)\mathcal{P}\left[\frac{p(p-1)}{2}U^{p-2}\varepsilon^2\rho^2+B\right]\\
			={}&\int_{\R^{n-1}}U^p\mathcal{P}[U^p]+p^2\varepsilon^2\int_{\R^{n-1}}U^{p-1}\rho\mathcal{P}[U^{p-1}\rho]+\frac{p^2(p-1)^2}{4}\varepsilon^4\int_{\R^{n-1}}U^{p-2}\rho^2\mathcal{P}[U^{p-2}\rho^2]\\
			&+\int_{\R^{n-1}}B\mathcal{P}[B]+2\int_{\R^{n-1}}\left(\varepsilon pU^{p-1}\rho+\varepsilon^2\frac{p(p-1)}{2}U^{p-2}\rho^2+B\right)\mathcal{P}[U^p]\\
			&+p^2(p-1)\varepsilon^3\int_{\R^{n-1}}U^{p-2}\rho^2\mathcal{P}[U^{p-1}\rho]+2\int_{\R^{n-1}}\left(\varepsilon pU^{p-1}\rho+\varepsilon^2\frac{p(p-1)}{2}U^{p-2}\rho^2\right)\mathcal{P}[B]\\
			={}&\int_{\R^{n-1}}U^{p+1}+\left(p^2\kappa_2^{-1}+p(p-1)\right)\varepsilon^2\int_{\R^{n-1}}U^{p-1}\rho^2++p^2(p-1)\kappa_2^{-1}\varepsilon^3\int_{\R^{n-1}}U^{p-2}\rho^3\\
			&+2\int_{\R^{n-1}}BU+o(\varepsilon^3)\\
			={}&\int_{\R^{n-1}}U^{p+1}+\left(p^2\kappa_2^{-1}+p(p-1)\right)\varepsilon^2\int_{\R^{n-1}}U^{p-1}\rho^2+p^2(p-1)\kappa_2^{-1}\varepsilon^3\int_{\R^{n-1}}U^{p-2}\rho^3\\
			&+\frac{p(p-1)(p-2)}{3}\varepsilon^3\int_{\R^{n-1}}U^{p-2}\rho^3+o(\varepsilon^3)\\
			={}&\norm{U}_{H^1}^2+\left(p^2\kappa_2^{-1}+p(p-1)\right)\varepsilon^2\kappa_2^{-1}+o(\varepsilon^3)\\
			&+\varepsilon^3\left(p^2(p-1)\kappa_2^{-1}+\frac{p(p-1)(p-2)}{3}\right)\int_{\R^{n-1}}U^{p-2}\rho^3,
		\end{align*}
		where $B=|u|^{p-1}u-U^p-pU^{p-1}\varepsilon\rho-\frac{p(p-1)}{2}U^{p-2}\varepsilon^2\rho^2$. Combining terms above, we can compute
		\begin{equation*}
			\begin{aligned}
				\left\Vert \Delta u +|u|^{p-1}u\right\Vert_{H^{-1}}^2=\varepsilon^2(1-p\kappa_2^{-1})^2+\varepsilon^3\frac{p(p-1)}{3}(\kappa_2^{-1}-1)\int_{\R^{n-1}}U^{p-2}\rho^3+o(\varepsilon^3).
			\end{aligned}
		\end{equation*}
		Note that
		\begin{equation*}
			d(u,\mathcal{M}_{\mathrm{E}})^2=\norm{u-U}_{H^1}^2=\varepsilon^2.
		\end{equation*}
		The full quotient is
		\begin{equation*}
			\frac{\left\Vert \Delta u +|u|^{p-1}u\right\Vert_{H^{-1}}^2}{d(u,\mathcal{M}_{\mathrm{E}})^2}=(1-p\kappa_2^{-1})^2+\varepsilon\frac{p(p-1)}{3}(\kappa_2^{-1}-1)\int_{\R^{n-1}}U^{p-2}\rho^3+o(\varepsilon).
		\end{equation*}
		Since $\kappa_2=\frac{n+2}{n-2}$, it suffices to take a suitable $\rho$ such that $\int_{\R^{n-1}}U^{p-2}\rho^3<0$. Let
		\begin{equation*}
			\rho=-\F^{-1}[y_1y_2+y_2y_3+y_3y_1],
		\end{equation*}
		then
		\begin{equation*}
			\begin{aligned}
				\int_{\R^{n-1}}U^{p-2}\rho^3&=-\left(\frac{n-2}{2}\right)^{n-1}\int_{\B^n}(y_1y_2+y_2y_3+y_3y_1)^3\,dy\\
				&=-\left(\frac{n-2}{2}\right)^{n-1}\int_{\B^n}6y_1^2y_2^2y_3^2\,dy<0.
			\end{aligned}
		\end{equation*}
		The proof is complete.  \hfill$\Box$
        
\vskip0.3in
\noindent{\textbf{
Acknowledgements}}\;\;We would like to thank the anonymous referees for very carefully reading this manuscript and  valuable comments that greatly improve this paper.

\vskip0.3in
\noindent{\Large\textbf{Declarations}}
~\\
~\\
\textbf{Conflict of interest}\quad On behalf of all authors, the corresponding author states that there is no Conflict of interest.~\\
{\textbf{Data Availability Statements}}\quad All data generated or analyzed during this study are included in this article.
\vskip0.5in


\addcontentsline{toc}{section}{References}\begin{thebibliography}{99}
		\bibitem{Ary} Aryan, S.: Stability of Hardy Littlewood Sobolev inequality under bubbling. Calc. Var. Partial Differential Equations. 62, no. 8, Paper No. 223, 42 pp (2023)
		
		\bibitem{Aubin} Aubin, T.: Problemes isop\'erim\'etriques et espaces de Sobolev. J. Differential Geom. 11, 573-598 (1976)
		
		\bibitem{stek} Girouard, A., Polterovich, I.: Spectral geometry of the Steklov problem (survey article). J. Spectr. Theory. 7, no. 2, 321-359 (2017)
		
		\bibitem{bartsch} Bartsch, T., Weth, T., Willem, M.: A Sobolev inequality with remainder term and critical equations on domains with topology for the polyharmonic operator. Calc. Var. Partial Differential Equations. 18, no. 3, 253-268 (2003)
		
		\bibitem{beck} Beckner, W.: Functionals for multilinear fractional embedding. Acta Math. Sin. (Engl. Ser.) 31, no. 1, 1-28 (2015)
		
		\bibitem{bianchi1991} Bianchi, G., Egnell, H.: A note on the Sobolev inequality. J. Funct. Anal. 100, 18-24 (1991)
		
		\bibitem{bonforte1} Bonforte, M., Dolbeault, J., Nazaret, B., Simonov, N.: Stability in Gagliardo-Nirenberg-Sobolev Inequalities Flows, Regularity, and the Entropy Method. Mem. Amer. Math. Soc. 308, no. 1554, iii+166 pp (2025)
		
		\bibitem{bonforte2} Bonforte, M., Dolbeault, J., Nazaret, B., Simonov, N.: Stability in Gagliardo-Nirenberg inequalities-Supplementary material. arXiv preprint (2020), arXiv:2007.03419
		
		\bibitem{Brezis-Lieb} Br\'ezis, H., Lieb, E.: Sobolev inequalities with remainder terms. J. Funct. Anal. 62, 73-86 (1985)
		
		\bibitem{carlen1} Carlen, E.: Duality and stability for functional inequalities. Ann. Fac. Sci. Toulouse Math. (6), 26, no. 2, 319-350 (2017)
		
		\bibitem{figalli-Carlen} Carlen, E., Figalli, A.: Stability for a GNS inequality and the log-HLS inequality, with application to the critical mass Keller-Segel equation. Duke Math. J. 162, no. 3, 579-625 (2013)
		
		\bibitem{Chenlu} Chen, L., Lu, G., Tang, H.: Stability of Hardy-Littlewood-Sobolev inequalities with explicit lower bounds. Adv. Math. 450, Paper No. 109778, 28 pp (2024)

        \bibitem{Chenlu1} Chen, L., Lu, G., Tang, H.: Optimal asymptotic lower bound for stability of fractional Sobolev inequality and the stability of Log-Sobolev inequality on the sphere. arXiv preprint (2024), arXiv:2312.11787v4

        \bibitem{Chenlu2} Chen, L., Lu, G., Tang, H.: Optimal stability of Hardy-Littlewood-Sobolev and Sobolev inequalities of arbitrary orders with dimension-dependent constants. arXiv preprint (2024), arXiv:2405.17727v2
		
		\bibitem{chen-frank} Chen, S., Frank, R., Weth, T.: Remainder terms in the fractional Sobolev inequality. Indiana Univ. Math. J. 62, 1381-1397 (2013)

        \bibitem{Che} Chen, H., Kim, S., Wei, J.: Sharp quantitative stability estimates for critical points of fractional Sobolev inequalities. arXiv preprint (2024), arXiv:2408.07775
		
		\bibitem{cianchi-fusco-maggi-pratelli} Cianchi, A., Fusco, N., Maggi, F., Pratelli, A.: The sharp Sobolev inequality in quantitative form. J. Eur. Math. Soc. 11, no. 5, 1105-1139 (2009)
		
		\bibitem{cicalese-Leonardi} Cicalese, M., Leonardi, G.: A selection principle for the sharp quantitative isoperimetric inequality. Arch. Ration. Mech. Anal. 206, no. 2, 617-643 (2012)
		
		\bibitem{9} Ciraolo, G., Figalli, A., Maggi, F.: A quantitative analysis of metrics on $\R^n$ with almost constant positive scalar curvature, with applications to fast diffusion flows. Int. Math. Res. Not. 21, 6780-6797 (2018)

        \bibitem{Dai} Dai, W., Qin, G.: Liouville theorem for poly-harmonic functions on $R_+^n$. Arch. Math. (Basel) 115, no. 3, 317-327 (2020)

        \bibitem{De} De Nitti, N., Glaudo, F., K\"{o}nig T.: Non-degeneracy, stability and symmetry for the fractional Caffarelli-Kohn-Nirenberg inequality. To appear in: Commun. Partial Differ. Equations (2025). arXiv:2403.02303

		
		\bibitem{konig2} De Nitti, N., K\"{o}nig, T.: Stability with explicit constants of the critical points of the fractional Sobolev inequality and applications to fast diffusion. J. Funct. Anal. 285, no. 9, Paper No. 110093, 30 pp (2023)
		
		\bibitem{dsw} Deng, B., Sun, L., Wei, J.: Sharp quantitative estimates of Struwe's Decomposition. Duke Math. J. 174, no. 1, 159-228 (2025)
		
		\bibitem{Deng1} Deng, S., Tian, X.: On the stability of Caffarelli-Kohn-Nirenberg inequality in $\R^2$. arXiv preprint (2023), arXiv:2308.04111
		
		\bibitem{Do} Do, A., Flynn, J., Lam, N., Lu, G.: $L^p$-Caffarelli-Kohn-Nirenberg inequalities and their stabilities. arXiv preprint (2023), arXiv:2310.07083
		
		\bibitem{dolbeault-Toscani} Dolbeault, J., Toscani, G.: Stability results for logarithmic Sobolev and Gagliardo-Nirenberg inequalities. Int. Math. Res. Not. no. 2, 473-498 (2016)
		
		\bibitem{dol-fig} Dolbeault, J., Esteban, M., Figalli, A., Frank, R., Loss, M.: Sharp stability for Sobolev and log-Sobolev inequalities, with optimal dimensional dependence. Camb. J. Math. 13, no. 2, 359-430 (2025)
		
		\bibitem{fractr} Einav, A., Loss, M.: Sharp trace inequalities for fractional Laplacians. Proc. Amer. Math. Soc. 140, no. 12, 4209-4216 (2012)
		
		\bibitem{Escobar} Escobar, J.: Sharp constant in a Sobolev trace inequality. Indiana Univ. Math. J. 37, no. 3, 687-698 (1988)
		
		\bibitem{figalli-fusco} Figalli, A., Fusco, N., Maggi, F., Millot, V., Morini, M.: Isoperimetry and stability properties of balls with respect to nonlocal energies. Comm. Math. Phys. 336, no. 1, 441-507 (2015)
		
		\bibitem{figalli} F{}igalli, A., Glaudo, F.: On the Sharp Stability of Critical Points of the Sobolev Inequality. Arch. Ration. Mech. Anal. 237, 201-258 (2020)
		
		\bibitem{figalli-maggi-Pratelli-2013} Figalli, A., Maggi, F., Pratelli, A.: Sharp stability theorems for the anisotropic Sobolev and log-Sobolev inequalities on functions of bounded variation. Adv. Math. 242, 80-101 (2013)

		\bibitem{figalli-Maggi-Pratelli-iso} Figalli, A., Maggi, F., Pratelli, A.: A mass transportation approach to quantitative isoperimetric inequalities. Invent. Math. 182, no. 1, 167-211 (2010)
		
		\bibitem{figalli-Neumayer} F{}igalli, A., Neumayer, R.: Gradient stability for the Sobolev inequality: the case $p\geq2$. J. Eur. Math. Soc. 21, 319-354 (2019)
		
		\bibitem{figalli-Zhang1} F{}igalli, A., Zhang, Y.: Sharp gradient stability for the Sobolev inequality. Duke Math. J. 171, 2407-2459 (2022)
		
		\bibitem{Frank2} Frank, R., Peteranderl, J.: Degenerate Stability of the Caffarelli-Kohn-Nirenberg Inequality along the Felli-Schneider Curve. Calc. Var. Partial Differential Equations. 63, no. 2, Paper No. 44 (2024)
		
		\bibitem{fusco-maggi-pratelli} Fusco, N., Maggi, F., Pratelli, A.: The sharp quantitative isoperimetric inequality. Ann. of Math. (2) 168, no. 3, 941-980 (2008)
		
		\bibitem{Gidas} Gidas, B., Ni, W., Nirenberg, L.: Symmetry and related properties via the maximum principle. Comm. Math. Phys. 68, no. 3, 209-243 (1979)
		
		\bibitem{konig3} K\"{o}nig, T.: On the sharp constant in the Bianchi-Egnell stability inequality. Bull. Lond. Math. Soc. 55, no. 4, 2070-2075 (2023)

        \bibitem{konig4} K\"{o}nig, T.: An exceptional property of the one-dimensional Bianchi-Egnell inequality.  Calc. Var. Partial Differential Equations 63, no. 5, Paper No. 123, 21 pp (2024)

        \bibitem{konig1} K\"{o}nig, T.: Stability for the Sobolev inequality: Existence of a minimizer. J. Eur. Math. Soc. (2025), published online first
		
		\bibitem{liu} Liu, K., Zhang, Q., Zou, W.: On the stability of critical points of the Hardy-Littlewood-Sobolev inequality. arXiv preprint (2023), arXiv:2306.15862

        \bibitem{Lu} Lu, Q., Yang, M., Zhao, S.: Remainder terms, profile decomposition and sharp quantitative stability in the fractional nonlocal Sobolev-type inequality with $n>2s$. arXiv preprint (2025), arXiv:2503.06636


		
		\bibitem{maggi} Maggi, F.: Some methods for studying stability in isoperimetric type problems. Bull. Amer. Math. Soc. 45, no. 3, 367-408 (2008)

        \bibitem{carlo} Mercuri, C., Willem, M.: A global compactness result for the p-Laplacian involving critical nonlinearities. Discrete Contin. Dyn. Syst. 28, no. 2, 469-493 (2010)
		
		\bibitem{neumayer1} Neumayer, R.: A strong form of the quantitative Wulff inequality. SIAM J. Math. Anal. 48, no. 3, 1727-1772 (2016)
		
		\bibitem{neumayer2}Neumayer, R.: On minimizers and critical points for anisotropic isoperimetric problems. 2018 MATRIX annals, 293-302, MATRIX Book Ser., 3, Springer (2020)
		
		\bibitem{neumayer3} Neumayer, R.: A note on strong-form stability for the Sobolev inequality. Calc. Var. Partial Differential Equations 59, no. 1, Paper No. 25, 8 pp (2020)
		
		\bibitem{ou} Ou, B.: Positive harmonic functions on the upper half space satisfying a nonlinear boundary condition. Differential Integral Equations 9, no. 5, 1157-1164 (1996)
		
		\bibitem{Ho} Ho, P.: A note on the Sobolev trace inequality. Proc. Amer. Math. Soc. 150, no. 3, 1257-1267 (2022)
		
		\bibitem{yang} Piccione, P., Yang, M., Zhao, S.: Quantitative profile decomposition and stability for a nonlocal Sobolev inequality. J. Differential Equations 417, 64-104 (2025)
		
		\bibitem{ruffini} Ruffini, B.: Stability theorems for Gagliardo-Nirenberg-Sobolev inequalities: a reduction principle to the radial case. Rev. Mat. Complut. 27, no. 2, 509-539 (2014)

		\bibitem{struwe1984} Struwe, M.: A global compactness result for elliptic boundary value problems involving limiting nonlinearities. Math. Z. 187, 511-517 (1984)
		
		\bibitem{Struwe2} Struwe, M.: Variational methods. Applications to nonlinear partial differential equations and Hamiltonian systems. Fourth edition. A Series of Modern Surveys in Mathematics. Springer-Verlag, Berlin (2008)

        \bibitem{Sun} Sun, Li., Xiong, J.: Classification theorems for solutions of higher order boundary conformally invariant problems, I. J. Funct. Anal. 271, no. 12, 3727-3764 (2016)
		
		\bibitem{Talenti} Talenti, G.: Best constant in Sobolev inequality. Ann. Mat. Pura Appl. 110, 353-372 (1976)
		
		\bibitem{wang} Wang, Z., Willem, M.: Caffarelli-Kohn-Nirenberg inequalities with remainder terms. J. Funct. Anal. 203, no. 2, 550-568 (2003)
		
		\bibitem{Wu-Wei} Wei, J., Wu, Y.: On the stability of the Caffarelli-Kohn-Nirenberg inequality. Math. Ann. 384, 1509-1546 (2022)
		
		\bibitem{Wu-Wei2} Wei, J., Wu, Y.: Stability of the Caffarelli–Kohn–Nirenberg inequality: the existence of minimizers. Math. Z. 308, no. 4, 26 pp (2025) 

        \bibitem{Yan} Yang, M., Zhao, S.: Stability estimates for critical points of a nonlocal Sobolev-type inequality. arXiv preprint (2025), arXiv:2501.01927

        \bibitem{zho1} Zhou, Y., Zou, W.: Quantitative stability for the Caffarelli-Kohn-Nirenberg inequality. arXiv preprint (2023), arXiv:2312.15735

        \bibitem{zho2} Zhou, Y., Zou, W.: Degenerate stability of critical points of the Caffarelli-Kohn-Nirenberg inequality along the Felli-Schneider curve. arXiv preprint (2024), arXiv:2407.10849

	\end{thebibliography}
\end{document}